\documentclass[10pt,a4paper]{article}

\usepackage[hidelinks]{hyperref}
\usepackage{lineno}

\usepackage{amssymb, bm}
\usepackage{amsmath}
\usepackage{amsthm}
\usepackage{mathrsfs}
\usepackage{siunitx}
\usepackage[margin=2cm]{geometry}
\usepackage{subfig}
\usepackage{float}
\usepackage{xcolor}
\usepackage{subfloat}

\usepackage{mathtools}
\usepackage{bbm}
\usepackage{multirow}
\usepackage{stmaryrd}

\usepackage{algorithm}     %serve per scrivere lo pseudo codice
\usepackage{algpseudocode} %serve per scrivere lo pseudo codice

\usepackage[utf8]{inputenc} % usually not needed (loaded by default)
\usepackage[T1]{fontenc}

\usepackage{tikz}
\usetikzlibrary{decorations.markings}
\usetikzlibrary{shapes.geometric}
\usetikzlibrary{shapes.arrows}
\usetikzlibrary{fit,calc}
\usetikzlibrary{quotes,angles}
\usetikzlibrary{arrows.meta}
\usepackage{animate}  
\usepackage{graphicx}
\usepackage{pgf}

\usepackage{pgfplots}
\usepackage[normalem]{ulem}

\modulolinenumbers[5]

\newtheorem{remark}{Remark}
\newtheorem{lemma}{Lemma}

\newcommand{\vertiii}[1]{{\left\vert\kern-0.25ex\left\vert\kern-0.25ex\left\vert #1 
   \right\vert\kern-0.25ex\right\vert\kern-0.25ex\right\vert}}

\usepackage{geometry}
\makeatletter
\def\blfootnote{\xdef\@thefnmark{$\star$}\@footnotetext}
\makeatother
\newenvironment{Authors}%
  {\begin{center}\begin{bfseries}}%
  {\end{bfseries}\end{center}}
\newenvironment{Addresses}%
  {\begin{flushleft}\begin{itshape}}%
  {\end{itshape}\end{flushleft}}
  \newcommand{\email}[1]{\hspace*{\stretch{1}}\emph{\texttt{#1}}}

 \usepackage{fancyhdr}  
  \fancypagestyle{plain}{
\fancyhead{}
\fancyhead[C]{\hfill Submitted to Elsevier, October 2023}
 }
\begin{document}

\thispagestyle{plain}

\title{A non-overlapping optimization-based domain decomposition approach to component-based model reduction of incompressible flows}
 \date{}
 \maketitle

 \maketitle
\vspace{-50pt} 
 
\begin{Authors}
Tommaso Taddei$^{1}$,  Xuejun Xu$^{2,3}$,
Lei Zhang$^2$.
\end{Authors}

\begin{Addresses}
$^1$
Univ. Bordeaux, CNRS, Bordeaux INP, IMB, UMR 5251, F-33400 Talence, France\\ Inria Bordeaux Sud-Ouest, Team MEMPHIS, 33400 Talence, France, \email{tommaso.taddei@inria.fr} \\[3mm]
$^2$
School of Mathematical Sciences, Tongji University, Shanghai 200092 , China, \email{22210@tongji.edu.cn} \\[3mm]
$^3$
Institute of Computational Mathematics, AMSS, Chinese Academy of Sciences, Beijing 100190, China, \email{xxj@lsec.cc.ac.cn} \\
\end{Addresses}

%65N30 = Finite elements, Rayleigh-Ritz and Galerkin methods, finite methods
%65N50 = Finite elements, Rayleigh-Ritz and Galerkin methods, finite methods
% 41A45 = Approximation by arbitrary linear expressions
% 35L02 = First-order hyperbolic equations
%90C26 = mathematical programming  nonconvex problems

% 35J15 = Second-order elliptic equations 

%-----------------------------
%      your text
%-----------------------------

\begin{abstract}
We present a component-based model order reduction procedure to efficiently and accurately solve parameterized incompressible flows governed by the Navier-Stokes equations. Our approach leverages a non-overlapping optimization-based domain decomposition technique to determine the control variable that minimizes jumps across the interfaces between sub-domains.
To solve the resulting constrained optimization problem, we propose both Gauss-Newton and sequential quadratic programming methods, which effectively transform the constrained problem into an unconstrained formulation.
Furthermore, we integrate model order reduction techniques into the optimization framework, to speed up computations. In particular, we incorporate localized training and adaptive enrichment to reduce the burden associated with the training of the local reduced-order models.
Numerical results are presented to demonstrate the validity and effectiveness of the overall methodology. 
\end{abstract}

\noindent
\emph{Keywords:} 
component-based model order reduction; optimization-based domain decomposition; non-overlapping methods; Navier-Stokes equations.
\medskip

\newcommand{\bs}[1]{\boldsymbol{#1}}

%\linenumbers
%\modulolinenumbers[5]

\section{Introduction}
Parameterized model order reduction (pMOR) techniques \cite{Quarteroni2015_ROMbook,Hesthaven_ROM_book2016,rozza2008reduced,benner_model_order_reduction_vol2} have gained widespread popularity in   science and engineering  to reduce the computational cost in scenarios that involve repetitive computational tasks, such as many-query and real-time applications. Given the parameter domain $\mathcal{P}$ and {a} parameterized partial differential equation (PDE) {of interest}, pMOR strategies rely on an offline/online computational decomposition: in the offline stage, which is computationally expensive and performed only once, a reduced basis (RB) approximation space is generated by exploiting several high-fidelity (HF) solutions (e.g., finite element, finite volume)  to the parameterized PDE for properly chosen parameter values, and a reduced order model (ROM) is then devised; in the online stage, for any new parameter value, the ROM can be solved with computational cost  independent of the HF discretization size $N_{\rm hf}$, {to ensure} significant computational savings. 
Efficient training algorithms, such as proper orthogonal decomposition (POD, \cite{berkooz1993proper,volkwein2011model}) and the weak-Greedy algorithm \cite{rozza2008reduced} are available to construct the reduced order basis (ROB). Additionally, effective projection-based techniques \cite{Carlberg2011,carlberg2017galerkin} can be employed to devise ROMs that are suitable for online calculations.

{The combination of} RB methods and domain decomposition (DD) methods offers further advantages \cite{Huynh_Knezevic_Patera_2013,Maier_Haasdonk_2014,Barnett2022}.
First,  
localized  pMOR techniques do not require global HF solutions over the whole domain: this feature has the potential to dramatically reduce the offline computational burden for large-scale systems.
Second, localization simplifies the task of defining a parameterization of the problem and enables model  reduction of systems with parameter-induced \emph{topology changes} (cf. section \ref{sec:model_pb}).
Third,
the DD framework offers the flexibility to seamlessly integrate ROMs with full order models (FOMs, generated by the HF discretization) or to accommodate multi-physics applications based on independent software.

Various approaches have been proposed to combine RB methods and DD methods {which differ in the way local ROMs are coupled at components' interfaces}. 
In the reduced basis element (RBE) method \cite{Maday_Ronquist_2002,Maday_Ronquist_2004,Lovgren_Maday_Ronquist_2006}, local ROMs are glued together using Lagrange multipliers. This method has been introduced in the context of the Laplace equation \cite{Maday_Ronquist_2002,Maday_Ronquist_2004} and subsequently applied to the Stokes equations \cite{Lovgren_Maday_Ronquist_2006}. A more recent application of the RBE method to the unsteady 3D Navier-Stokes equations can be found in \cite{Pegolotti2021}, where a spectral Lagrange multiplier on the 2D interfaces is employed to couple local solutions. 
Another approach is the static condensation RBE (scRBE) method \cite{Huynh_Knezevic_Patera_2013,Eftang_Patera_2013,benaceur2022port}, which ensures the component coupling through a static condensation procedure \cite{Craig1968}. Additionally, approximation spaces for the interfaces (ports) between the components are also constructed \cite{Eftang_Patera_2013,benaceur2022port} to further reduce the computational complexity associated with the static condensation system. Another advantage of the scRBE method is the interchangeability of the components, {which enables} the study of different systems from a single  library of {parameterized archetype components}. 
The RB-DD-finite-element (RDF) method \cite{Iapichino_Quarteroni_Rozza_2016} uses parametric boundary conditions in the local problems to define versatile local RB spaces for handling of networks composed of repetitive geometries characterized by different parameters. A detailed review of  these methods can be found in \cite{Iapichino_Quarteroni_Rozza_2016}. 

Iterative techniques based on substructuring and the Schwarz alternating methods \cite{quarteroni1999domain,mota2017schwarz} have   been adapted to the  pMOR framework  \cite{Buffoni_Telib_Iollo_2009,Maier_Haasdonk_2014,Barnett2022,
Zappon_Manzoni_Gervasio_Quarteroni_2022,Discacciati_Hesthaven_2023}. In \cite{Buffoni_Telib_Iollo_2009}, both {a} non-overlapping Dirichlet–Neumann iterative scheme and a  Schwarz method for overlapping sub-domains are proposed to ensure coupling between the FOM and the ROM. The coupling is achieved by ensuring the solution compatibility between the FOM solution trace and ROM solution trace at the interface. Specifically, only Galerkin-free ROMs are considered in the work of \cite{Buffoni_Telib_Iollo_2009}. 
Galerkin-based ROMs are explored in the context of DD in \cite{Barnett2022}, the authors develop a versatile coupling framework for both FOM-ROM coupling and ROM-ROM coupling,
{which can be applied}  to both overlapping and non-overlapping domains. 
Similarly, in \cite{Maier_Haasdonk_2014}   Galerkin-based ROMs  are employed {to speed up  Dirichlet-Neumann DD iterations}. 
A Dirichlet-Neumann DD-ROM is developed in \cite{Zappon_Manzoni_Gervasio_Quarteroni_2022} to handle non-conforming interfaces. Here, the Dirichlet and Neumann interface data are transferred using the INTERNODES method \cite{Deparis2016}. In \cite{Discacciati_Hesthaven_2023}, the authors present a DD-ROM technique which is designed for heterogeneous systems:  in this approach, components are treated separately, and a parametrization of the interface data is used to generate HF snapshots. 
%algebraic domain-decomposition. 

 {Moreover, several authors have  proposed to formulate the coupling problem as a minimization statement \cite{Bergmann_2018,Iollo_Sambataro_Taddei_2023}. 
In \cite{Bergmann_2018}, the optimization problem is framed as the minimization of the difference between the ROM reconstruction and the corresponding FOM solution within the overlapping region between the ROM and the FOM domain. This approach adopts Galerkin-free ROMs and is applied to approximating  incompressible flows,  such as the interaction between an airfoil and a vortex, and the incompressible turbulent flow past a vehicle with varying geometry. 
The one-shot overlapping Schwarz method \cite{Iollo_Sambataro_Taddei_2023} consists in a constrained optimization statement that penalizes the jump at the interfaces of the components, while adhering to the approximate fulfillment of the PDE within each sub-domain. This approach
has been validated for  a steady nonlinear mechanics problem  and also applied to an unsteady  nonlinear mechanics problem with internal variables \cite{sambataro2022component}, in combination with overlapping partitions. 
The results of \cite{Iollo_Sambataro_Taddei_2023} showed that the minimization framework, which enables the application of effective optimization solvers for nonlinear least-square problems, ensures rapid convergence to the solution and is also robust with respect to the overlapping size.

In the present work, we aim to extend the method of
\cite{Iollo_Sambataro_Taddei_2023}  to incompressible flows in non-overlapping domains: our point of departure is the variational formulation proposed in  
 \cite{Gunzburger_Peterson_Kwon_1999} and further developed in  \cite{Gunzburger_Heinkenschloss_Lee_2000,Gunzburger_Lee_2000_elliptic,Gunzburger_Lee_2000}. 
As in \cite{Gunzburger_Peterson_Kwon_1999},
we formulate the DD problem as an optimal  control problem where the control is given by the flux on the components' interfaces and the dependent variables are velocity and  pressure in each subdomain;
our formulation reads as  a constrained minimization problem where 
the objective functional measures the jump in the dependent variables  across the
common boundaries between subdomains,  while the constraints are the partial differential
equations in each subdomain. 
We modify the formulation of  \cite{Gunzburger_Peterson_Kwon_1999}
to incorporate  an auxiliary control variable for the continuity equation which 
weakly
ensures continuous finite-dimensional pressure across the interface; 
furthermore, we propose a specialized
sequential quadratic programming (SQP) method to efficiently solve the optimization problem without resorting to  Lagrange multipliers.
We remark that non-overlapping   techniques are of particular interest for 
\emph{heterogeneous} DD \cite{Quarteroni1992} tasks that  necessitate the 
 combination of different discretization methods  in each subdomain. Non-overlapping methods are also of interest for 
 interface problems with high-contrast coefficients \cite{Gorb_Kurzanova_2018} and for 
  fluid flows in repetitive networks 
  \cite{Pegolotti2021,Iapichino_Quarteroni_Rozza_2016,benaceur2022port}
  such as the vascular system.

We here consider two-dimensional steady-state simulations at moderate Reynolds number; however, our ultimate goal is to devise a flexible computational tool to simulate vascular flows in real, patient-specific geometries. We interpret complex networks as the union of a small number of parameterized components. 
In order to avoid expensive global solves at training stage, we propose a combined localized training and global enrichment strategy that exclusively uses local HF solves to create local approximations for the archetype components, thus avoiding the need for computationally demanding global HF solves during the training phase.

% repetitive geometries, with a specific application to vascular flows \cite{Pegolotti2021}. The concept of archetype components, discussed in Iappicino's work on scRBE. Smetana & Taddei localized training and online enrichment \cite{Smetana2023}, as well as Discattis2023

Our work is related to several previous contributions to component-based (CB) pMOR.
First, 
the variational formulation  is strongly related to the recent work by Prusak et al.  \cite{prusak2023optimisation}. The authors of  \cite{prusak2023optimisation} consider separate spaces for velocity and pressure and rely on pressure supremizer enrichment in combination with Galerkin projection to ensure stability of the local problems; furthermore, they resort to a Lagrangian multiplier and gradient-based methods as in \cite{Gunzburger_Peterson_Kwon_1999} to solve the global optimization problem. Instead, we consider a single reduced space for velocity and pressure; we rely on both the Galerkin projection and a Petrov-Galerkin formulation for the local problems; and we rely on the Gauss-Newton and SQP methods for optimization without resorting to Lagrange multipliers.
 Finally, the authors of \cite{prusak2023optimisation}  do not discuss the problem of localized training, which is of paramount importance for the success of CB techniques.
 Second, we  emphasize that several authors have previously developed 
 CB-pMOR methods for incompressible flows in  repetitive geometries \cite{Iapichino_Quarteroni_Rozza_2016,benaceur2022port}; in particular,  the work by Pegolotti and coauthors \cite{Pegolotti2021} first considered a CB-pMOR for the unsteady incompressible Navier-Stokes equations in realistic three-dimensional  geometries.
 Third, the localized training and global enrichment  strategies are  an extension of the method proposed in 
  \cite{Smetana2023}:   localized training strategies have been previously proposed  in 
 \cite{Eftang_Patera_2013,benaceur2022port,Discacciati_Hesthaven_2023};
 similarly,  enrichment techniques have been considered 
  in several efforts for linear elliptic PDEs (see, e.g.,  \cite{buhr2017arbilomod}).
}
 
 This paper is organized as follows. In section \ref{sec:incomp_N-S}, we introduce the optimization-based domain decomposition method and the model problem considered in this work.
  In section \ref{sec:FOM}, we review the 
variational formulation  introduced in 
\cite{Gunzburger_Peterson_Kwon_1999};  we present our new formulation; and 
we discuss the solution method based on Gauss-Newton and sequential quadratic programming.
Then in section \ref{sec:rom} we discuss the integration of projection-based ROMs into the proposed optimization framework and the hybrid solver that combines both the FOM solver and the ROM solver. 
In sections \ref{sec:FOM} and 
 \ref{sec:rom}  we illustrate the method for a simplified geometric configuration with two components.
Section \ref{sec:loc_train} is dedicated to the presentation of  the localized training and the  {adaptive} enrichment techniques. 
Finally, in section \ref{sec:numerical_res}, we present numerical results that validate the effectiveness of our   methodology.

\section{Optimization-based domain decomposition method for the Navier-Stokes equations}
\label{sec:incomp_N-S}
In this work, we consider the incompressible Navier-Stokes equations:
\begin{equation}
\label{eq:incompressible_strong}
\left\{
\begin{array}{ll}
-\nu \Delta \mathbf{u} + ( \mathbf{u}\cdot \nabla)\mathbf{u} + \nabla p
= \mathbf{f}
& {\rm in} \; \Omega, \\
\nabla \cdot \mathbf{u} = 0 
& {\rm in} \; \Omega, \\
\mathbf{u}|_{\Gamma_{\rm dir}}=\mathbf{u}_{\rm in},
\;\;
\mathbf{u}|_{\Gamma_{\rm dir}^0}=0,
\;\;
\left(
\nu \nabla \mathbf{u} - p\mathbf{I}
\right) \mathbf{n}
|_{\Gamma_{\rm neu}}=0,
&
\\
\end{array}
\right.
\end{equation}
where $\nu>0$ denotes the kinematic viscosity of the fluid, 
$\Omega$ is a bounded Lipschitz domain;
{the open sets} 
$\Gamma_{\rm dir},\Gamma_{\rm dir}^0,\Gamma_{\rm neu}$
{constitute} 
 a partition of $\partial \Omega$, {which are associated to}   non-homogeneous Dirichlet boundary conditions, homogeneuous Dirichlet boundary conditions and Neumann boundary conditions,
 respectively.
We consider two-dimensional problems; the extension to the three-dimensional case and to unsteady problems is beyond the scope of this paper.

\subsection{Optimization-based domain decomposition}
\label{sec:opt_based_dd_continuous}

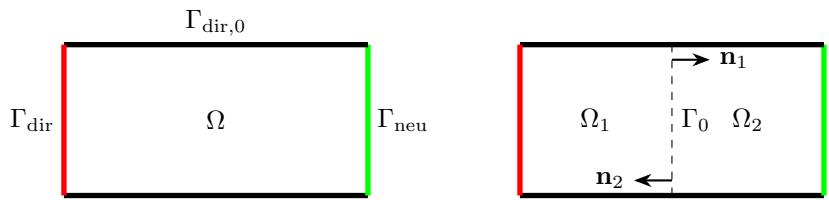
\begin{figure}[H]
\centering
\begin{tikzpicture}

% Draw the main rectangle representing the domain
\draw[black, line width=2pt] (0, 0) -- (4, 0); % Bottom side (red)
\draw[green, line width=2pt] (4, 0) -- (4, 2); % Right side (blue)
\draw[black, line width=2pt] (4, 2) -- (0, 2); % Top side (green)
\draw[red, line width=2pt] (0, 2) -- (0, 0); % Left side (orange)

\node at (2, 1) {$\Omega$};
\node[left] at (0, 1) {$\Gamma_{\rm dir}$};
\node[above] at (2, 2) {$\Gamma_{\rm dir,0}$};
\node[right] at (4, 1) {$\Gamma_{\rm neu}$};

% Draw the main rectangle representing the domain
\draw[black, line width=2pt] (6, 0) -- (10, 0); % Bottom side (red)
\draw[green, line width=2pt] (10, 0) -- (10, 2); % Right side (blue)
\draw[black, line width=2pt] (10, 2) -- (6, 2); % Top side (green)
\draw[red, line width=2pt] (6, 2) -- (6, 0); % Left side (orange)

% Draw the partition lines and label the sub-domains
\draw[dashed] (8, 0) -- (8, 2);
\node[right] at (8, 1) {$\Gamma_{0}$};
\node at (7, 1) {$\Omega_1$};
\node at (9, 1) {$\Omega_2$};

\draw[thick, -Stealth, black] (8,1.8) -- (8.5,1.8);
\node[right] at (8.5, 1.8) {$\mathbf{n}_1$};

\draw[thick, -Stealth, black] (8,0.2) -- (7.5,0.2);
\node[left] at (7.5, 0.2) {$\mathbf{n}_2$};

\end{tikzpicture}
\caption{The domain $\Omega$ and a partition into two non-overlapping sub-domains.}
\label{fig:two-subdomains}
\end{figure}

For the purpose of clarity, we introduce the optimization-based domain decomposition method in the case of two sub-domains. Note that this approach can be readily extended to accommodate many sub-domains, as  discussed in the subsequent sections. Consider a non-overlapping partition of $\Omega$ into two open sub-domains $\Omega_1$ and $\Omega_2$ such that $\overline{\Omega}=\overline{\Omega}_1\cup \overline{\Omega}_2$, as illustrated in Figure \ref{fig:two-subdomains}. The interface that separates the two sub-domains is denoted by $\Gamma_0$ so that ${\Gamma_0}=\overline{\Omega}_1\cap \overline{\Omega}_2$. The vectors $\mathbf{n}_i$, $i=1,2$, are the unit outward normals of $\Omega_i$ on $\Gamma_0$ (we thus have  $\mathbf{n}_1=-\mathbf{n}_2$). 
We define the local Dirichlet and Neumann conditions
for each component $\Omega_i$, $i=1,2$ {as}
\begin{subequations}
\label{eq:local_problems}
\begin{equation}
\label{eq:local_problems_a}
\Gamma_{i,\rm dir} = 
\Gamma_{\rm dir} \cap \partial \Omega_i,
\quad
 \Gamma_{i,\rm dir}^0 = 
\Gamma_{\rm dir}^0 \cap \partial \Omega_i,
\quad
 \Gamma_{i,\rm neu} = 
\Gamma_{\rm neu} \cap \partial \Omega_i,
%\quad
%i=1,\ldots,{N_{\rm dd}},
\end{equation}
and the spaces
\begin{equation}
\label{eq:local_problems_b}
\mathcal{X}_i:=
\left\{
(\mathbf{v},q)\in [H^1(\Omega_i)]^2 \times L^2(\Omega_i) \; : \;
\mathbf{v}|_{\Gamma_{i,\rm dir}^0} = 0
\right\},
\quad
\mathcal{X}_{i,0}:=\left\{
(\mathbf{v},q)\in \mathcal{X}_i \; : \; 
\mathbf{v}|_{\Gamma_{i,\rm dir}} = 0
\right\},
\quad
\mathcal{G}:=
[L^{2}(\Gamma_0)]^2.
\end{equation}

The local solution $(\mathbf{u}_i,p_i)\in \mathcal{X}_i$ is fully determined
 {by} the flux $\mathbf{g}$ at the interface $\Gamma_0$: as in \cite{Gunzburger_Peterson_Kwon_1999}, we thus refer to $\mathbf{g}$ as the control. 
Given the control  $\mathbf{g}\in \mathcal{G}$,
the velocity-pressure pair  
 $({\mathbf{u}}_i,\,{{p}}_i)$ 
  satisfies
$\mathbf{u}_i|_{\Gamma_{i,\rm dir}} = \mathbf{u}_{\rm in}|_{\Gamma_{i,\rm dir}}$ and
\begin{equation}
\label{eq:local_problems_c}
\mathcal{R}_i({\mathbf{u}}_i,\,{{p}}_i, \mathbf{v},\,q )
+ \mathcal{E}_i (\mathbf{g}, \mathbf{v})   = 0
\quad
\forall \, (\mathbf{v},\,q)\in \mathcal{X}_{i,0},
\end{equation}
where
\begin{equation} 
\label{eq:local_problems_d}
\mathcal{R}_i({\mathbf{u}}_i,\,{{p}}_i, \mathbf{v},\,q )=
\int_{\Omega_i}
\Big(
\nu \nabla \mathbf{u}_i : \nabla \mathbf{v} \, - \, p_i (\nabla \cdot \mathbf{v})
\, - \,  
q(\nabla \cdot \mathbf{u}_i)
\, - \,  
 \mathbf{f}_i\cdot \mathbf{v}
\Big)
\, dx,
\quad
\mathcal{E}_i (\mathbf{g}, \mathbf{v}) 
=
\,  (-1)^i\, 
\int_{\Gamma_{0}} \mathbf{g} \cdot \mathbf{v} \, dx, 
\end{equation}
for $i=1,2$. Here, the orientation of the flux $\mathbf{g}$ is chosen to be the same as $\mathbf{n}_1$, i.e., from $\Omega_1$ to $\Omega_2$; the choice of the orientation  is completely arbitrary. 
Note that an arbitrary choice of the control $\mathbf{g}$ does not guarantee that the local solutions $(\mathbf{u}_i,p_i)$ are solutions to \eqref{eq:incompressible_strong};
however, if    $( \mathbf{u}_1 - \mathbf{u}_2){\vert_{\Gamma_0}} =  0$, we find that the field $(\mathbf{u},p)$ such that
$ (\mathbf{u}\vert_{\Omega_1}, p \vert_{\Omega_1})$  and
$ (\mathbf{u}\vert_{\Omega_2}, p \vert_{\Omega_2})$  
 satisfy \eqref{eq:local_problems_c} is a weak solution to the global problem 
\eqref{eq:incompressible_strong}.
The \emph{optimal control} $\mathbf{g}$
should hence guarantee velocity equality at the interface $\Gamma_0$.
\end{subequations}

Gunzburger and coauthors
\cite{Gunzburger_Peterson_Kwon_1999,Gunzburger_Lee_2000} 
 proposed the  following optimization-based domain-decomposition formulation  to compute the desired control and the local solutions: 
\begin{equation}
\label{eq:lsq}
\min_{\substack{
(\mathbf{u}_1,p_1)\in \mathcal{X}_1; \\
(\mathbf{u}_2,p_2)\in \mathcal{X}_2; \\
\mathbf{g} \in \mathcal{G}} } \; 
 \frac{1}{2}\int_{\Gamma_0}\left\vert\mathbf{u}_1-\mathbf{u}_2\right\vert^2 dx + \frac{\delta}{2}\int_{\Gamma_0}\vert\mathbf{g}\vert^2 dx \quad
 {\rm s.t.} \;\;
 \left\{
 \begin{array}{l}
 \displaystyle{\mathcal{R}_i({\mathbf{u}}_i,\,{{p}}_i, \mathbf{v},\,q )
+ \mathcal{E}_i (\mathbf{g}, \mathbf{v})   = 0
\quad
\forall \, (\mathbf{v},\,q)\in \mathcal{X}_{i,0}}
\\[3mm]
\displaystyle{
\mathbf{u}_i|_{\Gamma_{i,\rm dir}} = \mathbf{u}_{\rm in}|_{\Gamma_{i,\rm dir}},}
  \\ 
 \end{array}
\right.
\; i=1,2.
\end{equation}
The second term in the objective function of \eqref{eq:lsq} is a regularizer that is designed to penalize controls of excessive size; the positive constant $\delta$ is chosen to control the relative importance of the two terms in 
the objective.
The proofs of the well-posedness of the optimization formulation, as well as the convergence of the optimal solution to the solution to  \eqref{eq:incompressible_strong} as the regularization parameter $\delta $  approaches $ 0$, can be found in \cite{Gunzburger_Lee_2000}.

\subsection{Model problem}
\label{sec:model_pb}
{As in \cite{Pegolotti2021}},  
we assume that the geometry of interest  can be effectively approximated  {through instantiations  of the elements of a library of archetype components; the instantiated components are obtained by low-rank geometric transformations of the archetype components}. 
{As in \cite{Pegolotti2021}, we consider  a library with  two archetype components:
``junction'' and ``channel'';}
the two archetype components are depicted in Figure \ref{fig:archetype_components}, where 
a number is assigned to each component edge. These edge numbers 
{indicate boundary face groups that are associated with the ports and the different types of  boundary conditions}.
Specifically, for the junction, edge numbers $\{1,4,7\}$ denote the ports and
edge numbers $\{2,3,5,6,8,9\}$ indicate homogeneous Dirichlet boundaries;
while for the channel, edge numbers $\{1,2\}$ represent the ports and edge numbers $\{3,4\}$ correspond to homogeneuous Dirichlet boundaries. 

A system can then be constructed by instantiating the two archetype components as follows:
$$
\overline{\Omega} = \mathop{\bigcup}\limits_{i=1}^{N_{\rm dd}} \overline{\Omega}_i, \quad \text{where} \quad \Omega_i= \Phi^{L_i}(\widetilde{\Omega}^{L_i},\mu_i),\quad i=1,\ldots,N_{\rm dd}, 
$$
where $L_i\in \{1,2\}$ denotes the label of the $i$-th component of the system, $\widetilde{\Omega}^{1}$, $\widetilde{\Omega}^{2}$ represent the two archetype components, $\Phi^{L_i}$ encompasses geometric transformations such as rotation, translation and non-rigid deformation that are applied to the archetype component  to obtain the corresponding instantiated component that appears in the  target system. The deformation
of the $i$-th component  is governed by the geometric parameter $\mu_i$;  the vector $\mu_i$ includes a scaling factor $\gamma$, the angle $\theta$ and a shift $\mathbf{x}_{\rm shift}$ that characterize the linear map that ensures the exact fitting of consecutive elements at ports. For the   junction component, the the vector $\mu_i$  also includes the angle $\alpha$, which represents the angle between the main vessel and the branch vessel,  as shown in Figure \ref{fig:archetype_components}(a); for the channel, the vector $\mu_i$ includes the constant $h_c$, which is used in the parameterization of the bottom boundary of the channel as $y=-h_c\,(4t\,(1-t))^{\alpha_c}$, with $t\in [0,1]$ and $\alpha_c=4$.

We prescribe a parabolic (Poiseuille) profile at the left boundary  $\mathbf{u}_{\rm in}$ and we prescribe homogeneous Neumann conditions at the other boundary ports. In conclusion, the complete system configuration is uniquely prescribed by (i) the component labels $\{ L_i \}_{i=1}^{N_{\rm dd}}$ and the geometric parameters $\mu={\rm vec}(\mu_1,\ldots,\mu_{N_{\rm dd}} )$, and
(ii) the Reynolds number $\text{Re}$ at the inlet.
We define the Reynolds number as $\text{Re}=\frac{Hu_0}{\nu}$, where $H=1$ denotes the diameter of the vessel at the inlet, $u_0$ represents the centerline velocity imposed at the inlet, and $\nu$ is the kinematic viscosity. 
In the numerical implementation, we set $\nu=\frac{1}{\text{Re}_{\rm ref}}$ in all the components of the network, and we consider the parametric inflow condition
$u_0(\text{Re}) =\frac{\text{Re}}{\text{Re}_{\rm ref}}$.

Figure \ref{fig:instantiation_example} illustrates two examples of target system, which consist of $3$ and $4$ components, respectively: the red numbers indicate the indices of the components, while the blue numbers indicate the internal ports. Note that the two systems are not isomorphic to each other: parameter variations hence induce \emph{topology changes} that prevent the application of standard monolithic pMOR techniques.

%\paragraph*{Remark} 
\begin{remark}
{We here observe that each component  includes mixed Dirichlet-Neumann boundary conditions:}  the presence of Neumann conditions prevents the problem of pressure indeterminacy (up to an additive constant), and the existence of Dirichlet conditions eliminates the need for any additional compatibility condition \cite{Gunzburger_Heinkenschloss_Lee_2000} concerning the control variable $\mathbf{g}$.
\end{remark}

\begin{figure}
  \centering
  
  \subfloat[junction.]{\includegraphics[width=7cm]{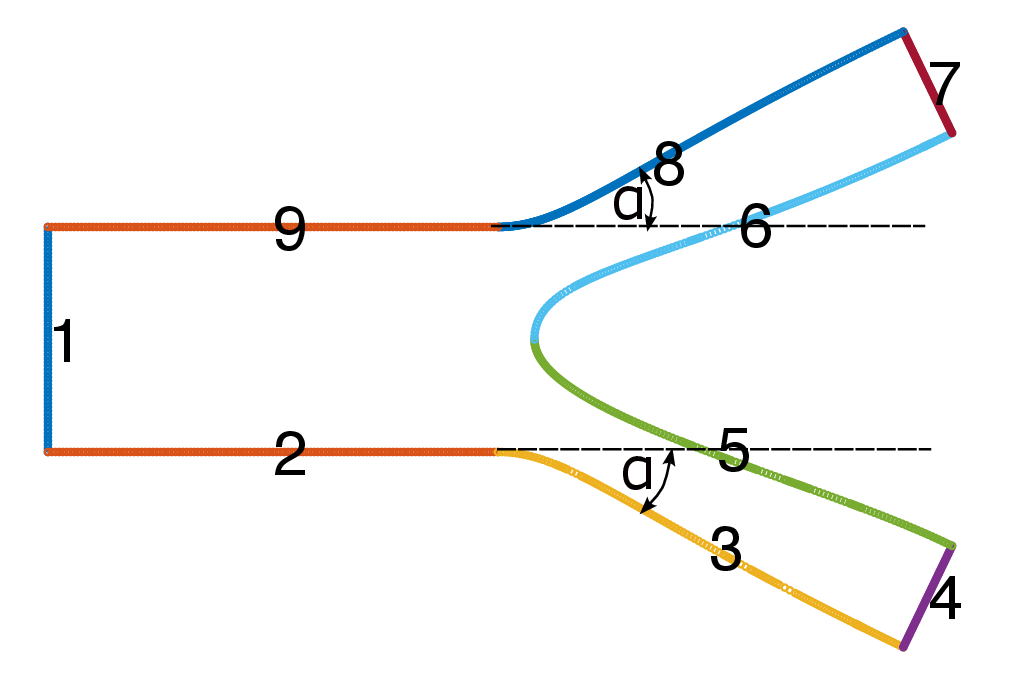}}
  \hfill
  \subfloat[channel.]{\includegraphics[width=7cm]{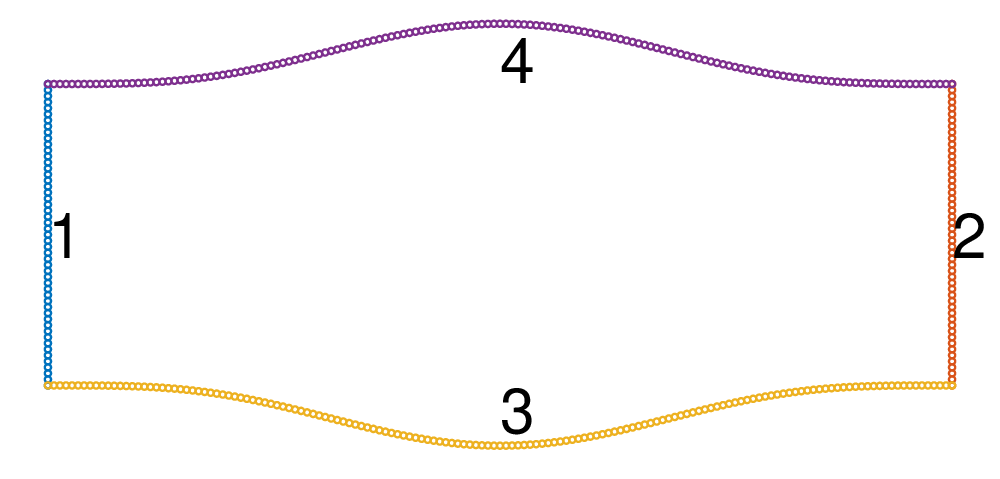}}
  
  \caption{archetype components.}
  \label{fig:archetype_components}
\end{figure}

\begin{figure}[H]
\centering
\subfloat[]{
\includegraphics[width=8.3cm]{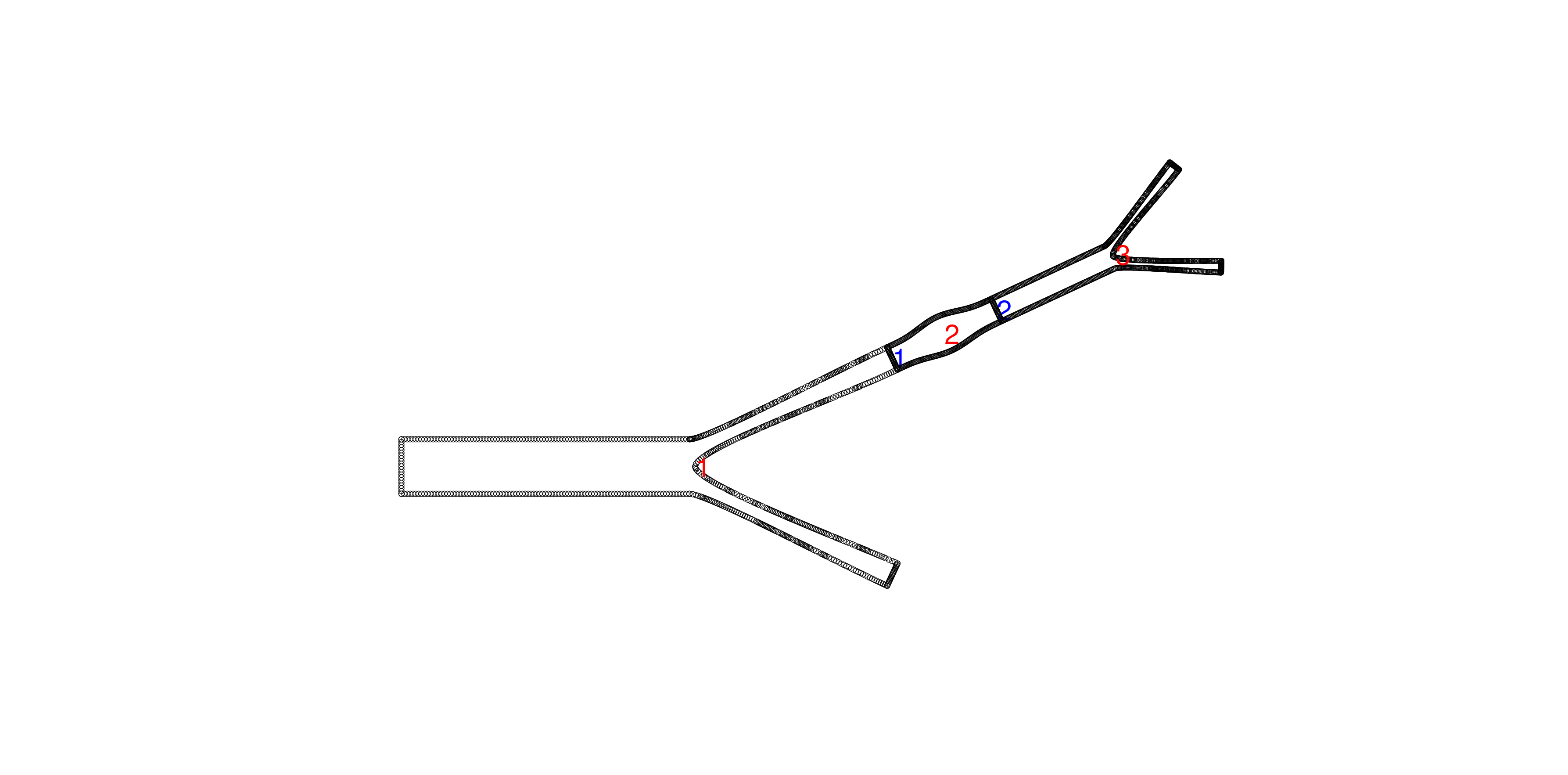}}
~~
\subfloat[]{
\includegraphics[width=7.3cm]{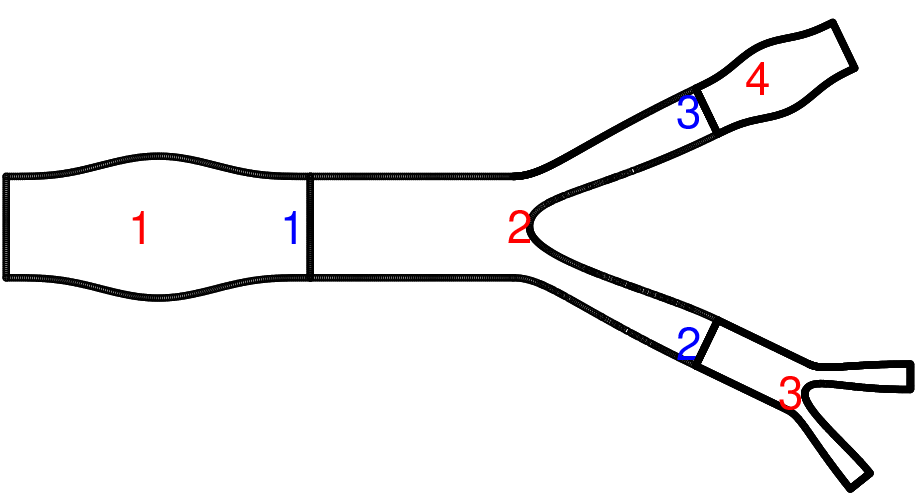}}

\caption{two example of target systems.}
\label{fig:instantiation_example}
\end{figure}

\begin{remark}
\label{remark:dirichlet_conditions}
We observe that the boundary face group 1 for the two archetype components either 
corresponds to an internal interface or to the inlet Dirichlet condition (for the first component of the network). In order to handle this scenario, we can either modify the objective function to include the extra-term
$\int_{\Gamma_{\rm dir}} | \mathbf{u} -  \mathbf{u}_{\rm in}  |^2 \, dx$ or to distinguish between inflow and internal channel and junction components. The latter option leads to a library with   ($N_{\rm c}=4$) archetype components. We here opt for the second strategy.
\end{remark}

\section{High-fidelity discretization}
\label{sec:FOM}

\subsection{Finite element spaces}
\label{sec:notation_discrete}
We proceed to discretize the optimization statement \eqref{eq:lsq}. Towards this end, we introduce the HF spaces
$
V_i^{\rm hf } \subset [H_{0,\Gamma_{i,\rm dir}^0}^1(\Omega_i)]^2$, 
$Q_i^{\rm hf } \subset  L^2(\Omega_i)$.
We further define the tensor product spaces
$\mathcal{X}_i^{\rm hf} = V_i^{\rm hf } \times Q_i^{\rm hf }$ and the lifted space 
$\mathcal{X}_{i,0}^{\rm hf} = V_{i,0}^{\rm hf } \times Q_i^{\rm hf }$
with $V_{i,0}^{\rm hf } = \{\mathbf{v}\in    V_{i}^{\rm hf }: \mathbf{v}|_{\Gamma_{i,\rm dir}} = 0\}$ for $i=1,2$. We denote by 
$\{ \bs{\varphi}_{i,j}  \}_{j=1}^{N_{i}^{\mathbf{u}}}$ a basis of $V_{i}^{\rm hf }$ and by $\{ {\psi}_{i,j}  \}_{j=1}^{N_{i}^{p}}$ a basis of $Q_i^{\rm hf }$; we use notation $\underline{\bullet}$ to indicate the FE vector associated with the FE field $\bullet$. We further define the trace spaces
$\Lambda_i^{\rm hf} = 
\{\tau_{\Gamma_0}  \mathbf{v} 
  :  \mathbf{v}\in V_{i}^{\rm hf } \}$ and
$\Xi_i^{\rm hf} = 
\{\tau_{\Gamma_0}  q :  q\in Q_{i}^{\rm hf } \}$, where 
$\tau_{\Gamma_0} \bullet := \bullet \vert_{{\Gamma_{0}}}$ indicates the trace of the field $\bullet$ on 
 $\Gamma_0$.  We here consider conforming meshes   such that nodes at the interface shared by the two sub-domains coincide, that is
$\Lambda_1^{\rm hf}=\Lambda_2^{\rm hf}=\Lambda^{\rm hf}$
and
$\Xi_1^{\rm hf}=\Xi_2^{\rm hf}=\Xi^{\rm hf}$;
this assumption is exploited in the    technical result of  \ref{appendix:pressure_jump}; 
 nevertheless, the  formulation can be trivially extended to non-conforming grids. We further define the global spaces
$\mathcal{X}^{\rm hf} = V^{\rm hf } \times Q^{\rm hf }$ and
$\mathcal{X}_0^{\rm hf} = V_0^{\rm hf } \times Q^{\rm hf }$ with 
 $V_{0}^{\rm hf } = \{\mathbf{v}\in    V^{\rm hf }: \mathbf{v}|_{\Gamma_{\rm dir}} = 0\}$.

In this work, we adopt a stabilized FE formulation that incorporates the Streamline Upwind/Petrov-Galerkin (SUPG) \cite{Tezduyar1991,Tezduyar2000} and the Pressure-Stabilized Petrov–Galerkin (PSPG) \cite{Hughes1986} stabilizations. The PSPG technique allows the  use of  the same polynomial degree for both pressure and velocity discretizations;  the SUPG technique enhances robustness for high Reynolds numbers. The detailed description of these stabilization formulas is given in \ref{sec:stablized_FE}.
In conclusion, we consider the following local problems, which are the counterpart of \eqref{eq:local_problems_c}:
\begin{subequations}
\label{eq:local_problems_tmp}
\begin{equation}
 \left\{
 \begin{array}{l}
 \displaystyle{\mathcal{R}_i^{\rm hf}({\mathbf{u}}_i,\,{{p}}_i, \mathbf{v},\,q )
+ \mathcal{E}_i (\mathbf{g}, \mathbf{v})   = 0
\quad
\forall \, (\mathbf{v},\,q)\in \mathcal{X}_{i,0}^{\rm hf}}
\\[3mm]
\displaystyle{
\mathbf{u}_i|_{\Gamma_{i,\rm dir}} = \bs{\Phi}_{i,\mathbf{u}_{\rm in}}
}
  \\ 
 \end{array}
\right.
\; i=1,2.
\end{equation}
where $\bs{\Phi}_{i,\mathbf{u}_{\rm in}}\in V_{i}^{\rm hf}$ is the interpolant of the nodal values of $\mathbf{u}_{\rm in}$ on $\Gamma_{i,\rm dir}$ \cite[p. 174]{quarteroni1999domain}. In view of the discussion below, we rewrite the HF residual as
\begin{equation}
\mathcal{R}_i^{\rm hf}({\mathbf{u}}_i,\,{{p}}_i, \mathbf{v},\,q )
=
\mathcal{R}_{i,\rm u}^{\rm hf}({\mathbf{u}}_i,\,{{p}}_i, \mathbf{v})
+
\mathcal{R}_{i,\rm p}^{\rm hf}({\mathbf{u}}_i,\,{{p}}_i, q );
\end{equation}
the first term corresponds to the residual of the momentum equation \eqref{eq:incompressible_strong}$_1$, while the second term corresponds to the residual of the continuity equation 
\eqref{eq:incompressible_strong}$_2$.
\end{subequations}

\subsection{Variational formulation}
\label{sec:FOM_base}
{Exploiting the previous notation, we can introduce the HF counterpart of the 
optimization formulation \eqref{eq:lsq}:
\begin{equation}
\label{eq:lsq_discr_standard}
\min_{\substack{
(\mathbf{u}_1,p_1)\in \mathcal{X}_1^{\rm hf}; \\
(\mathbf{u}_2,p_2)\in \mathcal{X}_2^{\rm hf}; \\
\mathbf{g} \in \Lambda^{\rm hf}}} \; 
 \frac{1}{2}\int_{\Gamma_0}\left\vert\mathbf{u}_1-\mathbf{u}_2\right\vert^2 dx + \frac{\delta}{2}\int_{\Gamma_0}\vert\mathbf{g}\vert^2 dx \quad
 {\rm s.t.} \;\;
 \left\{
 \begin{array}{l}
 \displaystyle{\mathcal{R}_i^{\rm hf}({\mathbf{u}}_i,\,{{p}}_i, \mathbf{v},\,q )
+ \mathcal{E}_i (\mathbf{g}, \mathbf{v})   = 0
\quad
\forall \, (\mathbf{v},\,q)\in \mathcal{X}_{i,0}^{\rm hf}}
\\[3mm]
\displaystyle{
\mathbf{u}_i|_{\Gamma_{i,\rm dir}} = \bs{\Phi}_{i,\mathbf{u}_{\rm in}},}
\quad
\; i=1,2.
  \\ 
 \end{array}
\right.
\end{equation}
This formulation coincides with the statement considered in \cite{Gunzburger_Lee_2000} and also 
\cite{prusak2023optimisation} --- with the minor difference that we here rely on a stabilized FE formulation for the local problems.
In the remainder of this section, we discuss an alternative HF formulation that will be used  to define the reduced-order model.

Formulation \eqref{eq:lsq_discr_standard} does not ensure the continuity of pressure across the internal interfaces: we prove this result rigorously in \ref{appendix:pressure_jump}; here, we provide a sketch of the proof that justifies our new DD statement.
If we denote by $(\mathbf{u}^{\rm hf}, p^{\rm hf})\in \mathcal{X}^{\rm hf}$ the solution to the global problem such that
$\mathcal{R}^{\rm hf}( \mathbf{u}^{\rm hf},   
p^{\rm hf}, \mathbf{v}, q ) = 0$ and we neglect for simplicity the stabilization term, we obtain
$$
\mathcal{R}_{\rm p}^{\rm hf}( \mathbf{u}^{\rm hf},    q )
=
\int_{\Omega_1}  \nabla \cdot  \mathbf{u}^{\rm hf} q \, dx
+
\int_{\Omega_2}  \nabla \cdot  \mathbf{u}^{\rm hf} q \, dx
= 0  \quad
\forall \, q\in Q^{\rm hf}.
$$
 Since $Q^{\rm hf}$ is a space of continuous functions, it is in general false that 
 $\mathcal{R}_{i,\rm p}^{\rm hf}( \mathbf{u}^{\rm hf}|_{\Omega_i},    q ) =
 \int_{\Omega_i}  \nabla \cdot  \mathbf{u}^{\rm hf} q \, dx = 0$ for all  
$q\in Q_i^{\rm hf}$, $i=1,2$; nevertheless, it is possible to show that there exists $h^{\star}\in \Xi^{\rm hf}$ such that
$$
\mathcal{R}_{i,\rm p}^{\rm hf}( \mathbf{u}^{\rm hf}|_{\Omega_i},    q ) + 
(-1)^i \int_{\Gamma_0} h^{\star} q \, dx = 0 
 \quad
\forall \, q\in Q_i^{\rm hf}.
$$
Similarly, there exists 
$\mathbf{g}^{\star}\in \Lambda^{\rm hf}$ such that
$$
\mathcal{R}_{i,\rm u}^{\rm hf}( 
\mathbf{u}^{\rm hf}|_{\Omega_i}, 
p^{\rm hf}|_{\Omega_i},
   \mathbf{v} ) + 
(-1)^i \int_{\Gamma_0} \mathbf{g}^{\star} \cdot \mathbf{v} \, dx
= 0
 \quad
\forall \, \mathbf{v} \in V_i^{\rm hf}, \;\; i=1,2.
$$
We conclude that the tuple
$( \mathbf{u}^{\rm hf}|_{\Omega_1}, 
p^{\rm hf}|_{\Omega_1} ,  \mathbf{u}^{\rm hf}|_{\Omega_2}, 
p^{\rm hf}|_{\Omega_2},  \mathbf{g}^{\star},  h^{\star}  )$ is a solution to the minimization problem
$$
\begin{array}{ll}
\displaystyle{
\min_{\substack{
(\mathbf{u}_1,p_1)\in \mathcal{X}_1^{\rm hf}; \\
(\mathbf{u}_2,p_2)\in \mathcal{X}_2^{\rm hf}; \\
\mathbf{g} \in \Lambda^{\rm hf},  
h \in \Xi^{\rm hf}   }} 
} &
\displaystyle{
 \frac{1}{2}\int_{\Gamma_0}\left\vert\mathbf{u}_1-\mathbf{u}_2\right\vert^2 dx +
 \frac{1}{2}\int_{\Gamma_0} \left( p_1-p_2\right)^2 dx  
}\\[3mm]
&
\displaystyle{
 {\rm s.t.} \;\;
 \left\{
 \begin{array}{l}
 \displaystyle{\mathcal{R}_i^{\rm hf}({\mathbf{u}}_i,\,{{p}}_i, \mathbf{v},\,q )
+ \mathcal{E}_i (\mathbf{g}, \mathbf{v})  
+
(-1)^i \int_{\Gamma_0} h q \, dx  = 0
\quad
\forall \, (\mathbf{v},\,q)\in \mathcal{X}_{i,0}^{\rm hf}},
\\[3mm]
\displaystyle{
\mathbf{u}_i|_{\Gamma_{i,\rm dir}} = \bs{\Phi}_{i,\mathbf{u}_{\rm in}},}
  \\ 
 \end{array}
\right.
i=1,2.
} 
 \\
\end{array}
$$

This discussion suggests to consider a modified formulation that explicitly penalizes the jump of the pressure field.
We introduce the state  $\mathbf{w}_i:=\text{vec}({\mathbf{u}}_i,{p}_i ) $, $i=1,2$ and the control
$\mathbf{s}:=\text{vec}({\mathbf{g}},h)$;
we introduce the control space
$\mathcal{S}^{\rm hf} = \Lambda^{\rm hf}  \times \Xi^{\rm hf}$ equipped with the norm
\begin{subequations}
\label{eq:lsq_discr_final}
\begin{equation}
\label{eq:calS_norm}
 \vertiii{ \mathbf{s} = {\rm vec} \left(
  \mathbf{g}, h \right) }^2
=
\int_{\Gamma_0}  
\big| \nabla_{\Gamma_0}   \mathbf{g}\big| ^2
+  | \mathbf{g} |^2
  + h^2 \, dx,
\end{equation}
where $\nabla_{\Gamma_0} \mathbf{g}$ denotes the gradient of $ \mathbf{g}$ in the tangential direction;
we use notation $\mathbf{w}_i(1:2)$ to indicate the first two components of the vector-valued function $\mathbf{w}_i$.
Then, we introduce the variational formulation:
\begin{equation}
\label{eq:lsq_discr_final_a}
\min_{\substack{
\mathbf{w}_1 \in \mathcal{X}_1^{\rm hf}; \\
\mathbf{w}_2 \in \mathcal{X}_2^{\rm hf}; \\
\mathbf{s} \in \mathcal{S}^{\rm hf}}} \; 
 \mathcal{F}_{\delta}\left(
 {\mathbf{w}}_1, {\mathbf{w}}_2,\mathbf{s}
 \right)
 \quad
 {\rm s.t.} \;\;
 \left\{
 \begin{array}{l}
 \displaystyle{\mathcal{R}_i^{\rm hf}({\mathbf{w}}_i,  \mathbf{z} )
+ \mathcal{E}_i^{\rm hf} (\mathbf{s}, \mathbf{z})   = 0
\quad
\forall \,  \mathbf{z} \in \mathcal{X}_{i,0}^{\rm hf}},
\\[3mm]
\displaystyle{\mathbf{w}_i(1:2)|_{\Gamma_{i,\rm dir}} = \bs{\Phi}_{i,\mathbf{u}_{\rm in}},}
  \\ 
 \end{array}
\right.
\; i=1,2;
\end{equation}
where  
\begin{equation} 
\label{eq:lsq_discr_final_b}
 \mathcal{F}_{\delta}\left(
 {\mathbf{w}}_1, {\mathbf{w}}_2,\mathbf{s}
 \right) := 
\frac{1}{2}
 \|    \mathbf{w}_1-\mathbf{w}_2 \|_{ L^2(\Gamma_0)}^2
 +
 \frac{1}{2}
 \delta  \vertiii{ \mathbf{s} }^2,
\end{equation}
\end{subequations}
 and
$\mathcal{E}_i^{\rm hf} (\mathbf{s}, \mathbf{v}, q)
=
\mathcal{E}_i(\mathbf{g}, \mathbf{v})
+
(-1)^i \int_{\Gamma_0} h q \, dx$. 
Note that we replaced the $L^2$ norm for the  control $\mathbf{g}$ with the $H^1$ norm: as discussed in section \ref{sec:hf_results}, we empirically observe that the use of the $H^1$ norm significantly reduces the oscillations in the profile of $\mathbf{g}$.

Some comments are in order.
First, the addition of the pressure jump and of the control $h$ ensures that the optimal pressure is continuous in the limit $\delta\to 0$. Note that at the continuous level the test space $Q=L^2(\Omega)$ is discontinuous; therefore, the control $h$ is   unnecessary.
Similarly, if we rely on a  P$0$ discretization for the pressure field \cite{Na_Xu_2022}, the pressure jump is 
also unnecessary.
Second, since velocity and pressure have different units and might also have very different magnitudes,
it might be necessary to rescale the objective function to avoid stability issues  (see, e.g.,  \cite{washabaugh2016use}).
In our numerical experiments, we solve the equations in non-dimensional form, and we do not include any scaling factor.

\subsection{Solution methods for \eqref{eq:lsq_discr_final}}
\label{sec:FOM_solve}
As in \cite{Iollo_Sambataro_Taddei_2023} and also \cite{Gunzburger_Lee_2000}, we resort to a gradient-based optimization method to find local minima of 
\eqref{eq:lsq_discr_final}. In more detail, we consider the Gauss-Newton method (GNM) and sequential quadratic programming (SQP)  \cite{Nocedal2006}.
As discussed below, both methods rely on static condensation to devise a reduced system for the control $\mathbf{s}$.

\subsubsection{Gauss-Newton method}
\label{sec:FOM_gnm}
We define the local solution map
$\mathcal{H}_i: \mathcal{S}^{\rm hf} \to \mathcal{X}_i^{\rm hf} $ such that 
$\mathcal{H}_i(\mathbf{s}) (1:2)|_{\Gamma_{i,\rm dir}} = \bs{\Phi}_{i,\mathbf{u}_{\rm in}}$ and 
\begin{equation}
\label{eq:constraint_gn}
\mathcal{R}_i^{\rm hf}(\mathcal{H}_i(\mathbf{s}),  \mathbf{z} )
+ \mathcal{E}_i^{\rm hf} (\mathbf{s}, \mathbf{z})   = 0
\quad
\forall \,  \mathbf{z} \in \mathcal{X}_{i,0}^{\rm hf},
\quad i=1,\,2.
\end{equation}
Then, we rewrite \eqref{eq:lsq_discr_final} as an unconstrained optimization problem:
\begin{equation}
\label{eq:unconstrained_statement_b}
\min_{\mathbf{s} \in \mathcal{S}^{\rm hf}} \mathcal{F}_{\delta}^{\rm gn}(\mathbf{s}) =
\mathcal{F}_{\delta}(\mathcal{H}_1(\mathbf{s}) ,\mathcal{H}_2(\mathbf{s}) ,     \mathbf{s}).
\end{equation}
If we define the space $\mathfrak{X}^{\rm hf}  = \Lambda^{\rm hf} \times \mathcal{S}^{\rm hf}$ equipped with the norm
$\| \mathbf{r} = {\rm vec}(\mathbf{w} , \mathbf{g} , h )  \|_{\mathfrak{X}^{\rm hf}}^2 = \| \mathbf{w} \|_{L^2(\Gamma_0)}^2 + 
\|   \mathbf{g} \|_{H^1(\Gamma_0)}^2 + \|  h \|_{L^2(\Gamma_0)}^2$ and the operator 
$F_{\delta}: \mathcal{S}^{\rm hf} \to \mathfrak{X}^{\rm hf}$ such that
$F_{\delta}(\mathbf{s}) =  {\rm vec}(
\tau_{\Gamma_0}  \left(
\mathcal{H}_1(\mathbf{s}) - \mathcal{H}_2(\mathbf{s}) \right), \sqrt{\delta}\mathbf{s})$, we can rewrite \eqref{eq:unconstrained_statement_b} as a nonlinear least-square problem, that is
\begin{subequations}
\label{eq:GNM}
\begin{equation}
\label{eq:unconstrained_statement_c}
\min_{\mathbf{s} \in \mathcal{S}^{\rm hf}} \mathcal{F}_{\delta}^{\rm gn}(\mathbf{s}) =
\frac{1}{2} \big\|  F_{\delta}(\mathbf{s}) \big\|_{\mathfrak{X}^{\rm hf}} ^2.
\end{equation}
The unconstrained problem \eqref{eq:unconstrained_statement_c} can be solved efficiently using GNM: given the initial condition 
$\mathbf{s}^{it=0}$, we repeatedly solve,
for $it=0,1,\ldots$, 
\begin{equation}
\label{eq:GNM_iteration} 
\mathbf{s}^{it+1}
={\rm arg} \min_{   \mathbf{s} \in \mathcal{S}^{\rm hf}  }
\frac{1}{2}
\big\|
F_{\delta}(\mathbf{s}^{it}) 
+
\frac{\partial F_{\delta} (\mathbf{s}^{it})   }{\partial \mathbf{s} } \left( \mathbf{s} - \mathbf{s} ^{it}  \right)
\big\|_{\mathfrak{X}^{\rm hf}} ^2.
\end{equation}
with the termination condition
\begin{equation}
\label{eq:termination_condition_gnm}
\frac{\vertiii{ \mathbf{s}^{it+1} - \mathbf{s} ^{it}  
}
}{\vertiii{  \mathbf{s} ^{it}  }    }
\leq tol,
\end{equation}
where $tol>0$ is a predefined tolerance.
\end{subequations}

We observe that GNM requires the explicit calculation of 
$F_{\delta}$ and 
the gradient of $F_{\delta}$ with respect to the control at $  \mathbf{s} ^{it}$: the former involves the solution to the local problems \eqref{eq:constraint_gn} for all components, while the latter is given by 
\begin{equation}
\label{eq:explicit_gradient_gnm}
\frac{\partial F_{\delta} (\mathbf{s}^{it})   }{\partial \mathbf{s} }
=
 \left[
 \begin{array}{c}
 \displaystyle{
\tau_{\Gamma_0}\left( 
  \frac{\partial  \mathcal{H}_{1} (\mathbf{s}^{it})   }{\partial \mathbf{s} } -
\frac{\partial  \mathcal{H}_{2} (\mathbf{s}^{it})   }{\partial \mathbf{s} }
\right)
 }\\[3mm]
  \displaystyle{  \sqrt{\delta} \texttt{id}  }
 \end{array}
\right]
\quad
{\rm with} \;\;
\frac{\partial  \mathcal{H}_i (\mathbf{s})   }{\partial \mathbf{s} }
=
-
\left(
\frac{\partial  \mathcal{R}_i^{\rm hf} (  \mathcal{H}_i (\mathbf{s}))   }{\partial \mathbf{w}_i }
\right)^{-1}  
\mathcal{E}_i^{\rm hf},
\end{equation}
and $\texttt{id}$ is the identity map.
We notice that the 
 evaluation of $\frac{\partial F_{\delta} (\mathbf{s}^{it})   }{\partial \mathbf{s} }$
involves the solution to $N^{\rm \mathbf{s}}$ linear systems where $N^{\rm \mathbf{s}}$ is the cardinality of the space $\mathcal{S}^{\rm hf}$;  it   is hence computationally feasible only if the dimension of the control is moderate: this observation highlights the importance of port reduction
\cite{Eftang_Patera_2013}
 for optimization-based methods. Conversely, we remark that the computation of $\mathcal{H}_1 (\mathbf{s}^{it}), \mathcal{H}_2 (\mathbf{s}^{it})$ and their derivatives is embarrassingly parallel with respect to the number of components: as discussed in  \cite{Iollo_Sambataro_Taddei_2023}, GNM enables  effective parallelization of the solution procedure if compared to standard multiplicative Schwartz iterative methods, provided that
 the computational cost is dominated by the solution to the local problems \eqref{eq:constraint_gn}.
Finally, we remark that the least-square problem in \eqref{eq:GNM_iteration} can be solved by explicitly assembling  the normal equations; 
alternatively, we might employ the QR factorization \cite{Carlberg2011}.  We omit the details.

\subsubsection{Sequential quadratic programming (SQP)}
\label{sec:FOM_sqp}
The SQP method solves a sequence of optimization subproblems, each of which optimizes a quadratic model of the objective subject to a linearization of the constraints. 
Since the objective (cf. \eqref{eq:lsq_discr_final_b}) is quadratic, we hence find the iterative method
\begin{subequations}
\label{eq:SQP}
\begin{equation}
\label{eq:SQP_iteration_a}
\left( \mathbf{w}_1^{it+1},\mathbf{w}_2^{it+1}, \mathbf{s}^{it+1} \right)
={\rm arg} 
\min_{\substack{
\mathbf{w}_1 \in \mathcal{X}_1^{\rm hf}; \\
\mathbf{w}_2 \in \mathcal{X}_2^{\rm hf}; \\
\mathbf{s} \in \mathcal{S}^{\rm hf}}} \; 
 \mathcal{F}_{\delta}\left(
 {\mathbf{w}}_1, {\mathbf{w}}_2,\mathbf{s}
 \right)
 \quad
 {\rm s.t.} \;\;
 \left\{
 \begin{array}{l}
 \displaystyle{
 \mathcal{R}_{i,it}^{\rm hf}(\mathbf{z} )
 +
  \mathcal{J}_{i,it}^{\rm hf}( \mathbf{w}_i- \mathbf{w}_i^{it} ,     \mathbf{z} )
+ \mathcal{E}_i^{\rm hf} (\mathbf{s}, \mathbf{z})   = 0}
\\[3mm]
\displaystyle{
{\mathbf{w}}_i(1:2)|_{\Gamma_{i,\rm dir}} = \bs{\Phi}_{i,\mathbf{u}_{\rm in}},} \quad
\forall \,  \mathbf{z} \in \mathcal{X}_{i,0}^{\rm hf}, \; 
\; i=1,2;
  \\ 
 \end{array}
\right.
\end{equation}
where the linear forms $\{  \mathcal{R}_{i,it}^{\rm hf} \}_i$ and the bilinear forms
$\{  \mathcal{J}_{i,it}^{\rm hf} \}_i$ are given by
\begin{equation}
\label{eq:SQP_iteration_b}
\mathcal{R}_{i,it}^{\rm hf}(\mathbf{z} ) =
\mathcal{R}_{i}^{\rm hf}( \mathbf{w}_i^{it},  \mathbf{z} ),
\quad
\mathcal{J}_{i,it}^{\rm hf}(\mathbf{w}, \mathbf{z} ) =
\frac{\partial \mathcal{R}_{i}^{\rm hf}}{\partial \mathbf{w}_{i}}\left[  \mathbf{w}_i^{it}    \right]
\left( \mathbf{w}, \mathbf{z}     \right),
\quad
\forall \, \mathbf{w}\in \mathcal{X}_i^{\rm hf},
\mathbf{z}\in \mathcal{X}_{i,0}^{\rm hf},
\;it=0,1,\ldots.
\end{equation}
In the numerical experiments, we consider the same termination condition \eqref{eq:termination_condition_gnm} used for GNM.
\end{subequations}

The optimization problem \eqref{eq:SQP_iteration_a} is quadratic with linear constraints. The solution to \eqref{eq:SQP_iteration_a} hence satisfies
\begin{subequations}
\label{eq:SQP_iteration_sc}
\begin{equation}
\label{eq:SQP_iteration_sc_a}
\left\{
\begin{array}{l}
\displaystyle{
\mathbf{s}^{it+1}
={\rm arg} \min_{   \mathbf{s} \in \mathcal{S}^{\rm hf}  }
\big\|
\widetilde{F}_{\delta}^{it} 
+
\widetilde{\mathcal{J}}_{\delta}^{it} 
\left( \mathbf{s} - \mathbf{s} ^{it}  \right)
\big\|_{\mathfrak{X}^{\rm hf}} ^2;
}
\\[3mm]
\displaystyle{
\mathbf{w}_i^{it+1} = 
\mathbf{w}_i^{it} -
\left( \mathcal{J}_{i,it}^{\rm hf}\right)^{-1} 
\left(
\mathcal{R}_{i,it}^{\rm hf}
+
\mathcal{E}_i^{\rm hf}
\mathbf{s}^{it+1}
\right),
\quad
i=1,2;
}
\\
\end{array}
\right.
\end{equation}
where
\begin{equation}
\label{eq:SQP_iteration_sc_b}
\widetilde{F}_{\delta}^{it}
=
 \left[
 \begin{array}{l}
 \tau_{\Gamma_0}\left(  
 \mathbf{w}_1^{it} - \mathbf{w}_2^{it} 
\right) 
 \\[3mm]
 \sqrt{\delta} \mathbf{s}^{it}   \\
 \end{array}
\right],
\quad
\widetilde{\mathcal{J}}_{\delta}^{it}  =
 \left[
 \begin{array}{c}
 \displaystyle{
  \tau_{\Gamma_0}\left(  
 \left( \mathcal{J}_{1,it}^{\rm hf}\right)^{-1} 
\mathcal{E}_1^{\rm hf}
-
 \left( \mathcal{J}_{2,it}^{\rm hf}\right)^{-1} 
\mathcal{E}_2^{\rm hf}
\right) 
  }\\[3mm]
  \displaystyle{  \sqrt{\delta} \texttt{id}  }
 \end{array}
\right].
\end{equation}
In our implementation, we rely on \eqref{eq:SQP_iteration_sc} to solve \eqref{eq:SQP_iteration_a}.
\end{subequations}

As for GNM, we   obtain a least-square problem for the control by applying static condensation: while in the previous section we first derived the unconstrained statement 
(cf.  \eqref{eq:unconstrained_statement_b}) and then we applied the optimization method, here we first optimize using SQP and then we apply static condensation at each iteration of the optimization algorithm.

Since the underlying PDE model is nonlinear, GNM requires to perform Newton subiterations to solve the local problems
\eqref{eq:constraint_gn} (see also the definition of $F_{\delta}(\mathbf{s}^{it})$ in 
\eqref{eq:GNM_iteration}); conversely, SQP does not involve subiterations. 
 The cost per iteration of SQP is hence significantly inferior to the cost of GNM. We empirically observe that the SQP approach mitigates the potential convergence issues of the sub-iterations for the local problems, particularly at the very early stages of the optimization loop.

We observe that \eqref{eq:GNM_iteration}  and 
\eqref{eq:SQP_iteration_sc_a}$_1$ are formally equivalent,
while 
\eqref{eq:explicit_gradient_gnm} and 
\eqref{eq:SQP_iteration_sc_b} share the same structure. We conclude that the SQP and GNM approaches can be implemented using the same data structures and can be parallelized in the same way. We omit the details.

\begin{remark}
For high-Reynolds number flows,
it is important to enhance the 
 robustness of our approach by  resorting to pseudo transient continuation (PTC) \cite{Kelley_Keyes_1998}. 
 PTC introduces an additional pseudo-temporal integration with adaptive time step, that is performed until convergence to a steady-state solution.
 If we resort to the backward Euler scheme for the discretization of the time derivative, at each   PTC step we solve the relaxed problem:
 \begin{equation}
\label{eq:lsq_discr_final_PTC}
\min_{\substack{
\mathbf{w}_1 \in \mathcal{X}_1^{\rm hf}; \\
\mathbf{w}_2 \in \mathcal{X}_2^{\rm hf}; \\
\mathbf{s} \in \mathcal{S}^{\rm hf}}} \; 
 \mathcal{F}_{\delta}\left(
 {\mathbf{w}}_1, {\mathbf{w}}_2,\mathbf{s}
 \right)
 \quad
 {\rm s.t.} \;\;
 \left\{
 \begin{array}{l}
 \displaystyle{
 \frac{1}{\Delta t_k}
 \int_{\Omega_i}
\Big(\mathbf{w}_i(1:2)
 -  \mathbf{w}_i^k(1:2)
 \Big) \cdot  \mathbf{v}     \, dx
 \,+ \,
  \mathcal{R}_i^{\rm hf}({\mathbf{w}}_i,  \mathbf{z} )
+ \mathcal{E}_i^{\rm hf} (\mathbf{s}, \mathbf{z})   = 0}
\\[3mm]
\displaystyle{
{\mathbf{w}}_i(1:2)|_{\Gamma_{i,\rm dir}} = \bs{\Phi}_{i,\mathbf{u}_{\rm in}},
\quad
\forall \,  \mathbf{z}=(\mathbf{v},q) \in \mathcal{X}_{i,0},
\;
\; i=1,2.
}
  \\ 
 \end{array}
\right.
\end{equation}
 where the index $k$ refers to the temporal loop and $\Delta t_k$ is chosen adaptively based on the residual of the steady-state equations.
We refer to \cite{Kelley_Keyes_1998} and to the references therein for further details.
Note that \eqref{eq:lsq_discr_final_PTC} is formally equivalent to \eqref{eq:lsq_discr_final_b}: it can hence be solved using the same procedure outlined above.
As discussed in \ref{sec:stablized_FE},  the time derivative should also be included in the SUPG and PSPG stabilization terms.
\end{remark}

\section{Projection-based reduced order formulation}
\label{sec:rom}
We rely on the formulation
\eqref{eq:lsq_discr_final_a} to define the CB-ROM.
Towards this end,
first, we identify a low-rank approximation of the control $\mathbf{s}^{\rm hf}$ and the local states 
$\mathbf{w}_1^{\rm hf}$, $\mathbf{w}_2^{\rm hf}$;
second, we devise local ROMs for the approximation of the solution maps
\eqref{eq:constraint_gn};
third, we devise specialized GNM and SQP methods for 
the formulation
\eqref{eq:lsq_discr_final_a} based on approximate solution  maps. 
We conclude the section by discussing the implementation of hybrid formulations that combine full-order and reduced-order local solution maps.
We remark that in order to further enhance online performance we should also reduce the online costs associated with the computation of the
$ \| \cdot   \|_{\mathfrak{X}^{\rm hf}}$
 norm in \eqref{eq:GNM_iteration} and 
 \eqref{eq:SQP_iteration_sc_a}$_1$ (cf.
 \cite{Iollo_Sambataro_Taddei_2023}): we do not address this issue in the present work.
 
\subsection{Construction of the local basis}
\label{sec:local_basis}
We denote by
$\{ \mu^{(k)} = {\rm vec}(\mu_1^{(k)},\mu_2^{(k)})  \}_{k=1}^{n_{\rm train}}$ a set of $n_{\rm train}$ global configurations; we further denote by
 $\{ \mathbf{w}_{i,k}^{\rm hf} : i=1,2, k=1, \ldots, n_{\rm train} \}$ and 
  $\{ \mathbf{s}_{k}^{\rm hf} : k=1, \ldots, 
  n_{\rm train} \}$ 
 the corresponding HF state and control estimates based on \eqref{eq:lsq_discr_final_a}.
 We resort to POD to devise a low-dimensional approximation space for the local solution manifolds and for the control
 \begin{equation}
 \label{eq:POD_application}
\left\{
\begin{array}{l}
\displaystyle{
\big[
\mathcal{Z}_i = {\rm span} \{  \boldsymbol{\zeta}_{i,j} \}_{j=1}^n
\big]
=\texttt{POD} \left(
\{ \mathbf{w}_{i,k}^{\rm hf} -  
\bs{\Psi}_{i,\mathbf{u}_{\rm in}}^{(k)}
 \}_{k=1}^{n_{\rm train}},
 \|   \cdot \|_{\mathcal{X}_i}, n
\right);
}
\\[3mm]
\displaystyle{
\big[
\mathcal{W} = {\rm span} \{  \boldsymbol{\eta}_{j} \}_{j=1}^m
\big]
=\texttt{POD} \left(
\{ \mathbf{s}_{k}^{\rm hf} 
 \}_{k=1}^{n_{\rm train}},
 \vertiii{\cdot }, m
\right).
}
\\
\end{array}
\right. 
 \end{equation}
 Here, the function $\texttt{POD} \left(
 \mathcal{D}, \|\cdot \|, n
\right)$ returns the POD space of dimension $n$ associated with the snapshot dataset $ \mathcal{D}$ and the norm 
 $ \|\cdot \|$ using the method of snapshots \cite{sirovich1987turbulence}. To ease the presentation, the integers $n$ and $m$ are here chosen \emph{a priori}: in practice, we should choose $n,m$ based on the  energy criterion. The fields 
 $\bs{\Psi}_{1,\mathbf{u}_{\rm in}}^{(k)},\bs{\Psi}_{2,\mathbf{u}_{\rm in}}^{(k)}$  satisfy the boundary conditions in \eqref{eq:lsq_discr_final_a}; we refer to 
  section \ref{sec:loc_train} for the explicit expression; this implies that the local space 
 $\mathcal{Z}_i$ is contained in $\mathcal{X}_{i,0}$, for $i=1,2$.
 In the remainder, we further use notation
 $\mathcal{Z}_i^{\rm dir} = \{  \bs{\Psi}_{i,\mathbf{u}_{\rm in}}(\mu) +  \boldsymbol{\zeta}_{i} : \boldsymbol{\zeta}_{i} \in \mathcal{Z}_i \}$ to identify the affine approximation spaces that incorporate Dirichlet boundary conditions.
 Furthermore, given 
  $\mathbf{w}_i \in \mathcal{Z}_i^{\rm dir}$ 
  and  
  $\mathbf{s}\in \mathcal{W}$, we define the generalized coordinates $\boldsymbol{\alpha}_1,\boldsymbol{\alpha}_2  \in \mathbb{R}^n$ and
  $\boldsymbol{\beta} \in \mathbb{R}^m$ such that
  \begin{equation}
  \label{eq:vec2field}
  \mathbf{w}_i(\boldsymbol{\alpha}_i; \mu)  = 
   \bs{\Psi}_{i,\mathbf{u}_{\rm in}}(\mu) + 
  \sum_{j=1}^n \left(  \boldsymbol{\alpha}_i \right)_j  \boldsymbol{\zeta}_{i,j},
  \;\;i=1,2; 
  \qquad
    \mathbf{s}(\boldsymbol{\beta})  = 
  \sum_{j=1}^m \left(  \boldsymbol{\beta} \right)_j  \boldsymbol{\eta}_{j}.
  \end{equation}

 \subsection{Construction of the local reduced-order models}
\label{sec:local_ROM}
 We rely on (Petrov-)Galerkin projection to define the local ROMs.
 
\paragraph{Galerkin projection.} We consider the local solution maps $\widehat{\mathcal{H}}_i^{\rm g}:
\mathcal{W} \to \mathcal{Z}_i^{\rm dir}
$  such that
\begin{equation}
\label{eq:local_solution_map_galerkin_a}
\mathcal{R}_i^{\rm hf}(\widehat{\mathcal{H}}_i^{\rm g}(\mathbf{s}),  \mathbf{z} )
+ \mathcal{E}_i^{\rm hf} (\mathbf{s}, \mathbf{z})   = 0
\quad
\forall \,  \mathbf{z} \in \mathcal{Z}_i,
\quad i=1,\,2.
\end{equation}
 It is useful to rewrite \eqref{eq:local_solution_map_galerkin_a} in fully-algebraic form. 
  Towards this end, 
 we  define the discrete residuals
 $\widehat{\mathbf{R}}_i^{\rm g}: \mathbb{R}^n \to \mathbb{R}^n$ and 
 $\widehat{\mathbf{E}}_i^{\rm g}\in \mathbb{R}^{n\times n}$ such that
  \begin{subequations}
\label{eq:local_solution_map_galerkin_algebraic}
\begin{equation}
\label{eq:local_solution_map_galerkin_algebraic_a}
\left(
\widehat{\mathbf{R}}_i^{\rm g}(\boldsymbol{\alpha})
\right)_j
=
\mathcal{R}_i^{\rm hf}\left(
  \mathbf{w}_i(\boldsymbol{\alpha}_i) , 
  \boldsymbol{\zeta}_{i,j}
\right),
\quad
\left(
\widehat{\mathbf{E}}_i^{\rm g}
\right)_{j,k}
=
\mathcal{E}_i^{\rm hf}\left(
  \boldsymbol{\eta}_{k} , 
  \boldsymbol{\zeta}_{i,j}
\right),
\quad
i=1,2,
j=1,\ldots,n,
k=1,\ldots,m;
\end{equation}
 and the local algebraic solution maps
 $\underline{\widehat{\mathcal{H}}}_i^{\rm g}:
\mathbb{R}^m \to \mathbb{R}^n$ such that
\begin{equation}
\label{eq:local_solution_map_galerkin_algebraic_b}
\widehat{\mathbf{R}}_i^{\rm g}
\left( 
\underline{\widehat{\mathcal{H}}}_i^{\rm g}
(\boldsymbol{\beta} )
   \right)
+
\widehat{\mathbf{E}}_i^{\rm g} \boldsymbol{\beta}
=0,
\quad
i=1,2.
\end{equation}
\end{subequations}

  \paragraph{Least-square Petrov-Galerkin (LSPG, \cite{Carlberg2011}) projection.}
 Given the reduced space $\mathcal{Y}_i\subset \mathcal{X}_{i,0}$, we introduce the local solution maps
 $\widehat{\mathcal{H}}_i^{\rm pg}:
\mathcal{W} \to \mathcal{Z}_i^{\rm dir}
$  such that
\begin{equation}
\label{eq:local_solution_map_LSPG}
\widehat{\mathcal{H}}_i^{\rm pg}(\mathbf{s})
= {\rm arg} \min_{ \boldsymbol{\zeta}\in \mathcal{Z}_i^{\rm dir}   }
\sup_{\mathbf{z}\in \mathcal{Y}_i} 
\frac{  
\mathcal{R}_i^{\rm hf}(\boldsymbol{\zeta},  \mathbf{z} )
+ \mathcal{E}_i^{\rm hf} (\mathbf{s}, \mathbf{z}) 
 }{\|   \mathbf{z} \|_{ \mathcal{X}_i  }}.
\end{equation}
For $\mathcal{Y}_i =  \mathcal{X}_{i,0}$, \eqref{eq:local_solution_map_LSPG} is referred to as \emph{minimum residual} projection.
In view of the derivation of the algebraic 
counterpart of 
 \eqref{eq:local_solution_map_LSPG},
 we denote by $\{ \bs{\upsilon}_{i,k}   \}_{k=1}^{j_{\rm es}}$ an orthonormal basis of $\mathcal{Y}_i$;
 \begin{subequations}
\label{eq:local_solution_map_LSPG_algebraic}
 then, we define the algebraic residuals 
 \begin{equation}
\label{eq:local_solution_map_LSPG_algebraic_a}
\left(
\widehat{\mathbf{R}}_i^{\rm pg}(\boldsymbol{\alpha}_i)
\right)_\ell
=
\mathcal{R}_i^{\rm hf}\left(
  \mathbf{w}_i(\boldsymbol{\alpha}_i) , 
  \boldsymbol{\upsilon}_{i,\ell}
\right),
\quad
\left(
\widehat{\mathbf{E}}_i^{\rm pg}
\right)_{\ell,k}
=
\mathcal{E}_i^{\rm hf}\left(
  \boldsymbol{\eta}_{k} , 
  \boldsymbol{\upsilon}_{i,\ell}
\right),
\end{equation}
with 
$i=1,2,
\ell=1,\ldots, j_{\rm es},
k=1,\ldots,m$;
 and the local algebraic solution maps
 $\underline{\widehat{\mathcal{H}}}_i^{\rm pg}:
\mathbb{R}^m \to \mathbb{R}^n$ such that
\begin{equation}
\label{eq:local_solution_map_LSPG_algebraic_b}
\underline{\widehat{\mathcal{H}}}_i^{\rm pg}(\boldsymbol{\beta})
=
{\rm arg} \min_{\boldsymbol{\alpha}\in \mathbb{R}^n}
\Big|
\widehat{\mathbf{R}}_i^{\rm pg}
\left( 
\boldsymbol{\alpha}
   \right)
+
\widehat{\mathbf{E}}_i^{\rm g} \boldsymbol{\beta}
\Big|,
\quad
i=1,2.
\end{equation}
\end{subequations} 
 We observe that \eqref{eq:local_solution_map_LSPG_algebraic_b} reads as a nonlinear least-square problem that can be solved efficiently using GNM; the combination  of LSPG ROMs within the DD formulation  \eqref{eq:lsq_discr_final_a} is challenging: we address this issue in the next section.
 
 The ROM \eqref{eq:local_solution_map_LSPG} depends on the choice of the test space $\mathcal{Y}_i$.  
 Following \cite{taddei2021space,taddei2021discretize},
 we propose to construct the test space
 $\mathcal{Y}_i$ using POD. Given the snapshots 
  $\{  \mathbf{w}_{i,k}^{\rm hf}  \}_{k}$ and the ROB
$ \{  \boldsymbol{\zeta}_{i,j} \}_{j=1}^n$, we compute the Riesz elements
$\boldsymbol{\psi}_{i,j,k}\in \mathcal{X}_{i,0}^{\rm hf}$ such that
\begin{subequations}
\label{eq:test_space}
\begin{equation}
\label{eq:test_space_a}
\left(
\boldsymbol{\psi}_{i,j,k},
\mathbf{z}
\right)_{\mathcal{X}_i}
\; = \;
\frac{\partial \mathcal{R}_{i}^{\rm hf}}{\partial \mathbf{w}_{i}}\left[  \mathbf{w}_{i,k}^{\rm hf}    \right]
\left( \boldsymbol{\zeta}_{i,j}, \mathbf{z}     \right),
\quad
\forall \; \mathbf{z}   \in \mathcal{X}_{i,0}^{\rm hf},
\end{equation}
for $i=1,2$, $j=1,\ldots,n$, $k=1,\ldots,n_{\rm train}$.
Then, we apply POD to find the low-dimensional bases
$\mathcal{Y}_1$ and $\mathcal{Y}_2$,
\begin{equation}
\label{eq:test_space_b}
\big[
\mathcal{Y}_i = {\rm span} \{  \boldsymbol{\upsilon}_{i,j} \}_{j=1}^{j_{\rm es}}
\big]
=\texttt{POD} \left(
\{ \boldsymbol{\psi}_{i,j,k} \, : \,  
j=1,\ldots,n, 
k=1, \ldots,  n_{\rm train},
\},
\|  \cdot\|_{\mathcal{X}_{i}},  j_{\rm es}
\right),
\quad
i=1,2.
\end{equation}
As in  \cite{taddei2021space,taddei2021discretize}, we choose $ j_{\rm es}=2n$; we refer to 
 \cite[Appendix C]{taddei2021space} for a rigorous justification of the choice of the test space for linear inf-sup stable problems.
\end{subequations}
 
 \begin{remark}
 The solution to \eqref{eq:local_solution_map_galerkin_a} and 
\eqref{eq:local_solution_map_LSPG} is expensive  due to the need to evaluate the HF residual and its Jacobian at each iteration. To reduce the computational burden, several authors have proposed to resort to hyper-reduction strategies \cite{ryckelynck2009hyper}  to speed up assembly costs at prediction stage.
We refer to the recent review \cite{farhat2020computational} for a detailed presentation of the subject. Since the local problems 
 \eqref{eq:local_solution_map_galerkin_a} and 
\eqref{eq:local_solution_map_LSPG} fit in the framework of monolithic pMOR, standard hyper-reduction techniques can be employed. We refer to a future work for the development and the assessment  of hyper-reduction techniques for the DD formulation of this work.
 \end{remark}

\subsection{Global formulation}
\label{sec:global_ROM}
We first introduce  the algebraic  counterpart of the objective \eqref{eq:lsq_discr_final_b}.
We denote by $\{  (\mathbf{x}_q^{\Gamma}, \omega_q^{\Gamma} )  \}_{q=1}^{N_{\Gamma}}$ the FE quadrature rule of $\int_{\Gamma_0}[\bullet] \, dx$ and we define the matrices $\mathbf{A}_1,\mathbf{A}_2\in \mathbb{R}^{3N_{\Gamma} \times n}$ and the vector 
 $\mathbf{b} \in \mathbb{R}^{3N_{\Gamma}}$ such that
\begin{subequations}
\label{eq:objective_function_reduced}
\begin{equation}
\label{eq:objective_function_reduced_a}
\left(   \mathbf{A}_i     \right)_{q+(\ell-1)N_{\Gamma},j}
 =
\sqrt{\omega_q^{\Gamma}}
    \left(
    \boldsymbol{\zeta}_{i,j}(   \mathbf{x}_q^{\Gamma}   )
      \right)_{\ell},
      \quad
\left(   \mathbf{b}     \right)_{q+(\ell-1)N_{\Gamma}}      
   =
     \sqrt{\omega_q^{\Gamma}}
    \left(
\bs{\Psi}_{1,\mathbf{u}_{\rm in}}(   \mathbf{x}_q^{\Gamma}   )
-
\bs{\Psi}_{2,\mathbf{u}_{\rm in}}(   \mathbf{x}_q^{\Gamma}   )
      \right)_{\ell}   
 \end{equation}
 with $q=1,\ldots,N_{\Gamma}$,
 $\ell=1,2,3$, $j=1,\ldots,n$; then, we rewrite the objective function as
 \begin{equation}
\label{eq:objective_function_reduced_b}
\boldsymbol{\mathcal{F}}_{\delta}\left(
 {\boldsymbol{\alpha}}_1, {\boldsymbol{\alpha}}_2,{\boldsymbol{\beta}}
 \right)
 = 
 \mathcal{F}_{\delta}\left(
 {\mathbf{w}}_1(   {\boldsymbol{\alpha}}_1 ), {\mathbf{w}}_2( {\boldsymbol{\alpha}}_2),\mathbf{s}({\boldsymbol{\beta}})
 \right)
 =
 \frac{1}{2}
\big|
\mathbf{A}_1 \boldsymbol{\alpha}_1
-
\mathbf{A}_2 \boldsymbol{\alpha}_2
\big|^2
+
\frac{\delta}{2}
\big|
\boldsymbol{\beta}
\big|^2.
\end{equation}  
 \end{subequations}

For the Galerkin local ROMs, the DD ROM can be obtained by simply projecting  \eqref{eq:lsq_discr_final_b} onto the reduced spaces, that is
 \begin{subequations}
 \label{eq:Galerkin_DD}
  \begin{equation}
\label{eq:Galerkin_DD_a}
\min_{\substack{
\mathbf{w}_1 \in \mathcal{Z}_1^{\rm dir}; \\
\mathbf{w}_2 \in \mathcal{Z}_2^{\rm dir}; \\
\mathbf{s} \in \mathcal{W}  }} \; 
 \mathcal{F}_{\delta}\left(
 {\mathbf{w}}_1, {\mathbf{w}}_2,\mathbf{s}
 \right)
 \quad
 {\rm s.t.} \;\;
 \displaystyle{\mathcal{R}_i^{\rm hf}({\mathbf{w}}_i,  \mathbf{z} )
+ \mathcal{E}_i^{\rm hf} (\mathbf{s}, \mathbf{z})   = 0
\quad
\forall \,  \mathbf{z} \in \mathcal{Z}_{i}}
\quad
\; i=1,2.
\end{equation}
Note that non-homogeneous Dirichlet conditions are encoded in the choice of the ansatz. Exploiting the previous notation, we obtain the algebraic  counterpart of 
 \eqref{eq:Galerkin_DD_a}.
   \begin{equation}
\label{eq:Galerkin_DD_b}
\min_{\substack{
\boldsymbol{\alpha}_1, \boldsymbol{\alpha}_2 \in \mathbb{R}^n; \\
\boldsymbol{\beta}  \in \mathbb{R}^m
}} \; 
\boldsymbol{\mathcal{F}}_{\delta}\left(
 {\boldsymbol{\alpha}}_1, {\boldsymbol{\alpha}}_2,{\boldsymbol{\beta}}
 \right)
 \quad
 {\rm s.t.} \;\;
\widehat{\mathbf{R}}_i^{\rm g}
\left( 
\boldsymbol{\alpha}_i 
   \right)
+
\widehat{\mathbf{E}}_i^{\rm g} \boldsymbol{\beta}
=0,
\quad
i=1,2.
\end{equation}
Problem \eqref{eq:Galerkin_DD_b} can be solved using either GNM or SQP; as for the HF model, the methods require the computation of the derivatives of the  local solution maps
 \eqref{eq:local_solution_map_galerkin_algebraic_b}, which satisfy
   \begin{equation}
\label{eq:Galerkin_DD_c}
 \frac{\partial \underline{\widehat{\mathcal{H}}}_i^{\rm g}}{\partial  \boldsymbol{\beta} }(  
\boldsymbol{\beta}
  )  
 =
 - \left(
  \frac{\partial 
\widehat{\mathbf{R}}_i^{\rm g}  
  }{\partial  \boldsymbol{\alpha}_i }
 \left[ 
\underline{\widehat{\mathcal{H}}}_i^{\rm g}(  
\boldsymbol{\beta}
  )  \right]
 \right)^{-1} 
 \widehat{\mathbf{E}}_i^{\rm g}.
\end{equation}
 Note that \eqref{eq:Galerkin_DD_c} can be computed using standard FE routines that are readily available for the full-order model.
   \end{subequations}
 
The combination of
\eqref{eq:lsq_discr_final_a} with the LSPG ROM 
 \eqref{eq:local_solution_map_LSPG_algebraic_b} 
is more involved since the resulting component-based ROM cannot be interpreted as the projection of
\eqref{eq:lsq_discr_final_a} onto suitable low-dimensional spaces.
We here rely on an approximate SQP procedure. At each iteration, given the triplet 
$(\boldsymbol{\alpha}_1^{it}, \boldsymbol{\alpha}_2^{it},
\boldsymbol{\beta}^{it})$, we compute 
\begin{subequations}
\label{eq:SQP_LSPG}
\begin{equation}
\label{eq:SQP_LSPG_a}
\widehat{\mathbf{R}}_i^{{\rm pg},it}
=
\widehat{\mathbf{R}}_i^{\rm pg}(\boldsymbol{\alpha}_i^{it})
\in \mathbb{R}^{j_{\rm es}},
\quad
\widehat{\mathbf{J}}_i^{{\rm pg},it}
=
\frac{\partial \widehat{\mathbf{R}}_i^{\rm pg}}{\partial \boldsymbol{\alpha}_i}
(\boldsymbol{\alpha}_i^{it})
\in \mathbb{R}^{j_{\rm es}\times n};
\end{equation}
then, we solve the minimization problem
\begin{equation}
\label{eq:SQP_LSPG_b}
\min_{\substack{
\boldsymbol{\alpha}_1, \boldsymbol{\alpha}_2 \in \mathbb{R}^n; \\
\boldsymbol{\beta}  \in \mathbb{R}^m
}} \; 
\boldsymbol{\mathcal{F}}_{\delta}\left(
 {\boldsymbol{\alpha}}_1, {\boldsymbol{\alpha}}_2,{\boldsymbol{\beta}}
 \right)
 \quad
 {\rm s.t.} \;\;
\left(
\widehat{\mathbf{J}}_i^{{\rm pg},it}
\right)^\top 
 \left(
 \widehat{\mathbf{R}}_i^{{\rm pg},it}
 +
 \widehat{\mathbf{J}}_i^{{\rm pg},it}
  \left(
  {\boldsymbol{\alpha}}_i - \boldsymbol{\alpha}_i^{it}
 \right) 
 +
 \widehat{\mathbf{E}}_i^{\rm pg}   {\boldsymbol{\beta}}
 \right)
=0,
\quad
i=1,2.
\end{equation}
\end{subequations}
We observe that for $n=j_{\rm es}$ the constraints imply that
$ \widehat{\mathbf{R}}_i^{{\rm pg},it}
 +
 \widehat{\mathbf{J}}_i^{{\rm pg},it}
  \left(
{\boldsymbol{\alpha}}_i - \boldsymbol{\alpha}_i^{it}
\right) 
+
\widehat{\mathbf{E}}_i^{\rm pg}   {\boldsymbol{\beta}}
=0$ for $i=1,2$. We hence recover
the  standard SQP procedure.

A thorough convergence  analysis of the SQP procedure \eqref{eq:SQP_LSPG} is beyond the scope of the present work.
Here, we observe that 
if $\bs{\alpha}_i^{it} \to \bs{\alpha}_i^\star$ for $i=1,2$ and $\bs{\beta}^{it} \to \bs{\beta}^\star$,   the constraints
in \eqref{eq:SQP_LSPG_b} reduce to 
$$
\left(
\frac{\partial \widehat{\mathbf{R}}_i^{\rm pg}}{\partial \boldsymbol{\alpha}_i}
(\boldsymbol{\alpha}_i^{\star})
\right)^\top
\left(
\widehat{\mathbf{R}}_i^{\rm pg}
(\boldsymbol{\alpha}_i^{\star})
+
\widehat{\mathbf{E}}_i^{\rm pg}
\bs{\beta}^\star
\right) = 0,
\quad
i=1,2.
$$
Given $i\in \{1,2\}$, the latter
  implies that 
$\boldsymbol{\alpha}_i^{\star}$
is a stationary point  of the function
$\boldsymbol{\alpha}_i\mapsto
\big|
\widehat{\mathbf{R}}_i^{\rm pg}
(\boldsymbol{\alpha}_i)
+
\widehat{\mathbf{E}}_i^{\rm pg}
\bs{\beta}^\star
\big|^2$; provided that 
\eqref{eq:local_solution_map_LSPG_algebraic_b} admits a unique solution, we hence find that
$\boldsymbol{\alpha}_i^{\star} = \underline{\widehat{\mathcal{H}}}_i^{\rm pg}(\boldsymbol{\beta}^{\star})$.

\subsection{Enrichment of the trial space}
\label{sec:enrichment_basis}
 In  \eqref{eq:POD_application}, we construct the state and control spaces independently. We might hence obtain that the matrices 
$ \frac{\partial \underline{\widehat{\mathcal{H}}}_i^{\rm g}}{\partial  \boldsymbol{\beta} }(  
\boldsymbol{\beta}
  )$ and 
  $ \frac{\partial \underline{\widehat{\mathcal{H}}}_i^{\rm pg}}{\partial  \boldsymbol{\beta} }(  
\boldsymbol{\beta}
  )$ are rank-deficient: as empirically shown in the numerical examples, rank deficiency of the sensitivity matrices leads to instabilities of the ROM and to poor approximations of the control $\mathbf{s}$. To address this issue, we propose to enrich the trial spaces $\mathcal{Z}_1,\mathcal{Z}_2$ with the perturbed snapshots
$\{ \widetilde{\mathbf{w}}_{i,j,k} \}_{i,j,k}$  
  \begin{equation}
  \label{eq:enriching_modes}
  \mathcal{R}_{i}^{\rm hf}(\mathbf{w}_{i,k}^{\rm hf},  \mathbf{z} )
  +
\frac{\partial \mathcal{R}_{i}^{\rm hf}}{\partial \mathbf{w}_{i}}\left[  \mathbf{w}_{i,k}^{\rm hf}    \right]
\left( \widetilde{\mathbf{w}}_{i,j,k} + 
\bs{\Psi}_{i,\mathbf{u}_{\rm in}}
- \mathbf{w}_{i,k}^{\rm hf}\;  ,   \;    \mathbf{z}     \right)
  +  \mathcal{E}_i^{\rm hf}( \boldsymbol{\eta}_j
  \;  ,   \;       
  \mathbf{z}   )
  =
  0
  \quad
  \forall \, \mathbf{z}   \in \mathcal{X}_{i,0}^{\rm hf}.
\end{equation}
In more detail, given the snapshots  $\{ \mathbf{w}_{i,k}^{\rm hf} : i=1,2, k=1, \ldots, n_{\rm train} \}$ and the reduced spaces 
  $\mathcal{Z}_1,\mathcal{Z}_2, \mathcal{W}$, we compute the perturbations
  $\{ \widetilde{\mathbf{w}}_{i,j,k}  \}_{j,k} \subset  \mathcal{X}_{i,0}^{\rm hf}$ for $i=1,2$, and then we update the reduced spaces $\mathcal{Z}_1$ and $\mathcal{Z}_2$ as follows:
  \begin{equation}
  \label{eq:hierarchicalPOD}
  \mathcal{Z}_i^{\rm new} = 
  \mathcal{Z}_i \oplus \mathcal{Z}_i',
  \quad
  {\rm with} \;
  \big[
\mathcal{Z}_i' 
\big]
=\texttt{POD} \left(
\{ 
\Pi_{\mathcal{Z}_i^\perp}
 \widetilde{\mathbf{w}}_{i,j,k}: j=1,\ldots,m, k=1, \ldots, n_{\rm train} \},
 \|   \cdot \|_{\mathcal{X}_i}, n'
\right),
 \end{equation}
 where  $\Pi_{\mathcal{Z}_i^\perp} \bullet $ denotes the projection of $\bullet$ onto the orthogonal complement of the space $\mathcal{Z}_i$ and $n'$ is a given integer.
 
 Some comments are in order. The hierarchical construction of the state approximation space  \eqref{eq:hierarchicalPOD} has been proposed in a similar context in \cite{haasdonk2017reduced}. The integer  $n'$ should be sufficiently large to ensure stability of the DD formulation; we further comment on the selection of $n'$  in the numerical experiments.
Finally, in \ref{appendix:enrichment}, we provide a formal justification of the enrichment strategy for a linear  problem.
  
\subsection{Hybrid solver}
\label{sec:hybrid_solver}

In the introduction, we anticipated the importance of developing a DD formulation that enables the seamless coupling of local, independently generated models.
  We here illustrate how to combine the HF model introduced in section \ref{sec:FOM} with the local ROM introduced in section \ref{sec:rom}. To provide a concrete reference, we assume that the HF model 
 \eqref{eq:constraint_gn} is solved in $\Omega_1$ and that the LSPG ROM   
  \eqref{eq:local_solution_map_LSPG} is solved in $\Omega_2$. 
  %; furthermore, we simply state the algebraic formulation.

We set
$N_1^{\mathbf{w}} =N_1^{\mathbf{u}}+N_1^{p} $ and we 
 define the basis (cf. section \ref{sec:notation_discrete})
  $$
  \{ \boldsymbol{\xi}_{1,j}  \}_{j=1}^{N_1^{\mathbf{w}}}
  =
  \left\{
  {\rm vec}(\bs{\varphi}_{1,1},0 ),\ldots,
  {\rm vec}(\bs{\varphi}_{1,N^{\mathbf{u}}},0 ),
   {\rm vec}(0,0,\psi_{1,1} ),\ldots,
     {\rm vec}(0,0,\psi_{1,N_1^p} ) 
  \right\}.
  $$
  We introduce   the vector-valued representation of the lifted state field
  $\mathring{\underline{\mathbf{w}}}_1
=
  {\underline{\mathbf{w}}}_1
  -
\underline{\bs{\Psi}}_{1,\mathbf{u}_{\rm in}} \in \mathbb{R}^{ N_1^{\mathbf{w}}   }  $. Then, we introduce the matrices
(see
\eqref{eq:local_solution_map_LSPG_algebraic_a}
and 
\eqref{eq:objective_function_reduced_a})
$\mathbf{A}_1^{\rm hf}    \in \mathbb{R}^{3 N_{\Gamma} \times N_1^{\mathbf{w}}  }$ and 
$\mathbf{E}_1^{\rm hf}    \in \mathbb{R}^{ N_1^{\mathbf{w}}\times m}$ such that
$$
\left(   \mathbf{A}_1^{\rm hf}     \right)_{q+(\ell-1)N_{\Gamma},j}
 =
\sqrt{\omega_q^{\Gamma}}
    \left(
    \boldsymbol{\xi}_{1,j}(   \mathbf{x}_q^{\Gamma}   )
      \right)_{\ell},
      \quad
 \left(
\widehat{\mathbf{E}}_1^{\rm hf}
\right)_{j,k}
=
\mathcal{E}_i^{\rm hf}\left(
  \boldsymbol{\eta}_{k} , 
 \boldsymbol{\xi}_{1,j}
\right).
$$
 Then, we can state the SQP method for the hybrid coupled problem: 
 \begin{subequations}
  \label{eq:SQP_hybrid}
  \begin{equation}
 \label{eq:SQP_hybrid_a}
 \min_{\substack{
\mathring{\underline{\mathbf{w}}}_1 \in \mathbb{R}^{N_1^{\mathbf{w}}}; \\
\boldsymbol{\alpha}_2  \in \mathbb{R}^n; \\
\boldsymbol{\beta}  \in \mathbb{R}^m
}} \; 
\frac{1}{2}
\big|
\mathbf{A}_1^{\rm hf} \mathring{\underline{\mathbf{w}}}_1
-
\mathbf{A}_2 \boldsymbol{\alpha}_2
+
\mathbf{b}
\big|^2
+\frac{\delta}{2} | \boldsymbol{\beta}   |^2
 \quad
 {\rm s.t.} \;\;
 \left\{
 \begin{array}{l}
 \displaystyle{
 {\mathbf{R}}_1^{{\rm hf},it}
 +
 {\mathbf{J}}_1^{{\rm hf},it}
 \left(
 \mathring{\underline{\mathbf{w}}}_1 -
 \mathring{\underline{\mathbf{w}}}_1^{it}
 \right)
 +
 \widehat{\mathbf{E}}_1^{\rm hf}
 {\boldsymbol{\beta}}
 =
 0;
 }
 \\[3mm]
\displaystyle{
\left(
\widehat{\mathbf{J}}_2^{{\rm pg},it}
\right)^\top 
 \left(
 \widehat{\mathbf{R}}_2^{{\rm pg},it}
 +
 \widehat{\mathbf{J}}_2^{{\rm pg},it}
  \left(
  {\boldsymbol{\alpha}}_2 - \boldsymbol{\alpha}_2^{it}
 \right) 
 +
 \mathbf{E}_2^{\rm pg}   {\boldsymbol{\beta}}
 \right)
=0.}
\\
\end{array}
 \right.
\end{equation}
where
\begin{equation}
 \label{eq:SQP_hybrid_b}
\left( {\mathbf{R}}_1^{{\rm hf},it}  \right)_j
=
 \mathcal{R}_1^{\rm hf}(
 {\mathbf{w}}_1^{it},
 \boldsymbol{\xi}_{1,j} ),
 \quad
 \left(  {\mathbf{J}}_1^{{\rm hf},it} \right)_{j,k}
   =
   \frac{\partial \mathcal{R}_1^{\rm hf}}{\partial \mathbf{w}_{i}}\left[  
   {\mathbf{w}}_1^{it}     \right]
\left( 
\boldsymbol{\xi}_{1,k} ,  
    \boldsymbol{\xi}_{1,j}   \right),
    \quad
    j,k=1,\ldots,N_1^{\mathbf{w}},
    \end{equation}
 \end{subequations}
with 
 ${\mathbf{w}}_1^{it} =  \mathring{\mathbf{w}}_1^{it}+ \bs{\Psi}_{1,\mathbf{u}_{\rm in}}$.
  
  Problem \eqref{eq:SQP_hybrid_a} can be solved using the static condensation procedure outlined in \eqref{eq:SQP_iteration_sc_a} and 
\eqref{eq:SQP_iteration_sc_b}. Note that for 
\eqref{eq:SQP_hybrid_a} 
the least-square problem 
\eqref{eq:SQP_iteration_sc_a}$_1$ is of size $m$: the computational cost is hence independent of $N_1^{\mathbf{w}}$.
On the other hand, the cost to assemble
the least-square problem in 
\eqref{eq:SQP_iteration_sc_a}$_1$ is dominated by the cost of computing 
$(  {\mathbf{J}}_1^{{\rm hf},it})^{-1} \mathbf{E}_1^{\rm hf}$, which requires the solution to $m$ linear systems of size $N_1^{\mathbf{w}}$. We emphasize that the local models in 
 \eqref{eq:SQP_hybrid_a} only communicate through the vehicle of the control $\mathbf{s}$ (or equivalently through the generalized coordinates $\bs{\beta}$) and the matrices $\mathbf{A}_1^{\rm hf}, \mathbf{A}_2$ in the objective function: the implementation of the local models is hence agnostic to the discretization that is employed in the neighboring subdomain.

\section{Localized training and {adaptive}  enrichment}
\label{sec:loc_train}
In section \ref{sec:rom} we devised the CB-ROM based on the DD formulation \eqref{eq:lsq_discr_final_a}. The major limitation of the approach is the need for global  HF solves to generate the reduced spaces (cf. \eqref{eq:POD_application}). In this section, we propose a general strategy to adaptively construct the reduced space for state and control, for the model problem of section \ref{sec:model_pb}. First, in section \ref{sec:multi_component}, we present the general multi-component DD formulation and relevant quantities that are employed in the adaptive procedure. Then, in sections 
\ref{sec:pairwise_control} and \ref{sec:localized_state}, we present the localized training strategies for the control $\mathbf{s}$ and for the local states. Finally, in section 
\ref{sec:adaptive_enrichment} we present the adaptive enrichment strategy that allows the correction of the local approximations based on global reduced-order solves.

\subsection{Multi-component formulation}
\label{sec:multi_component}
Given the archetype components $\{  \widetilde{\Omega}^k \}_{k=1}^{N_{\rm c}}$ and the reference port $\widetilde{\Gamma}$, we introduce the instantiated system $\Omega\subset \mathbb{R}^2$ such that
$\overline{\Omega} = \bigcup_{i=1}^{N_{\rm dd}} \overline{\Omega}_i$ with 
$ \Omega_i= \Phi^{L_i}(\widetilde{\Omega}^{L_i},\mu_i)$ for $i=1,\ldots,N_{\rm dd}$ and ports
$\{  \Gamma_j \}_{j=1}^{N_{\rm f}}$  such that
$ \Gamma_j= \Psi_j(\widetilde{\Gamma})$, where 
$\mu_1,\ldots,\mu_{N_{\rm dd}}$ are geometric parameters associated with the elemental mapping
and $\Psi_1,\ldots,\Psi_{N_{\rm dd}}$ are the mappings associated with the ports; we further introduce the union of all ports
$\Gamma := \bigcup_{j=1}^{N_{\rm f}}    \Gamma_j$.
For $i=1,\ldots,N_{\rm dd}$, we denote by $\mathcal{I}_i^{\Gamma} \subset \{ 1,\ldots,N_{\rm f} \}$ the set of the indices of the ports that belong  to $\partial \Omega_i$. We further denote by $\mathbf{n}_j^+$ the positive normal to the port $\Gamma_j$.
We denote by $\widetilde{\bs{\Psi}}_{k,\mathbf{u}_{\rm in}}$ the HF solution to the Navier-Stokes equations in $\widetilde{\Omega}_k$ with inflow condition $u_0({\rm Re}_{\rm ref})$ for some ${\rm Re}_{\rm ref}>0$ and Neumann boundary conditions on the remaining ports; then, we introduce the parametric field
$\widetilde{\bs{\Psi}}_{k,\mathbf{u}_{\rm in}}({\rm Re}) =
\frac{\rm Re}{{\rm Re}_{\rm ref}} \,  \widetilde{\bs{\Psi}}_{k,\mathbf{u}_{\rm in}}$.

We introduce the FE spaces 
$\widetilde{\mathcal{X}}_k^{\rm hf}$
and
$\widetilde{\mathcal{X}}_{k,0}^{\rm hf}$
associated with the domain $\widetilde{\Omega}^k$ (cf. section \ref{sec:notation_discrete}) for $k=1,\ldots,N_{\rm c}$; furthermore, we introduce the reduced spaces
$\widetilde{\mathcal{Z}}_k \subset \widetilde{\mathcal{X}}_{k,0}^{\rm hf}$ and the affine spaces
$\widetilde{\mathcal{Z}}_k^{\rm dir}({\rm Re}):=
\widetilde{\bs{\Psi}}_{k,\mathbf{u}_{\rm in}}({\rm Re})
+
\widetilde{\mathcal{Z}}_k$
---to shorten notation, we omit the dependence of $\widetilde{\mathcal{Z}}_k^{\rm dir}$ on the Reynolds number.
The choice   $\widetilde{\mathcal{Z}}_k=\widetilde{\mathcal{X}}_{k,0}^{\rm hf}$ corresponds to considering the HF discretization in all components of type $k$.
Then, we define the global discontinuous approximation space over $\Omega$
\begin{equation}  
\label{eq:approx_space_multi}
\mathcal{X}^{\rm dd}:=\left\{
\mathbf{w}\in [L^2(\Omega)]^3 \,:\,
\mathbf{w}|_{\Omega_i} \circ 
 \Phi^{L_i}(\cdot,\mu_i) \in \widetilde{\mathcal{Z}}_{L_i}^{\rm dir}, \;\; i=1,\ldots,N_{\rm dd}
\right\}.
\end{equation}
We denote by $\llbracket   \mathbf{w} \rrbracket   \in [L^2(\Gamma)]^3$ the jump of the field $\mathbf{w}$ on the interfaces of the partition
\begin{equation}
\label{eq:jump_definition}
\llbracket   \mathbf{w} \rrbracket ( \mathbf{x})
=
 \mathbf{w}^+( \mathbf{x}) -  \mathbf{w}^-( \mathbf{x}) \;\; \forall \, \mathbf{x} \in \Gamma_j,
 \;\;
  \mathbf{w}^{\pm}( \mathbf{x}) 
  :=
  \lim_{\epsilon\to 0^+} 
   \mathbf{w}( \mathbf{x} \mp \epsilon \mathbf{n}_j^+(\mathbf{x} )), \quad
   j=1,\ldots,N_{\rm f}. 
\end{equation}
Given the port reduced space $\widetilde{\mathcal{W}}\subset [L^2(\widetilde{\Gamma})]^3$, we also introduce the global port space over $\Gamma$
\begin{equation}
\label{eq:approx_space_control_multi}
\mathcal{W}^{\rm dd}:=\left\{
\mathbf{s}\in [L^2(\Gamma)]^3 \,:\,
\mathbf{s}|_{\Gamma_j} \circ 
 \Psi_j \in \widetilde{\mathcal{W}}, \;\; j=1,\ldots,N_{\rm f}
\right\}.
\end{equation}

We handle geometry deformations using the \emph{discretize-then-map}  approach (cf.  \cite{taddei2021discretize}).
Given the FE field $\mathbf{w}\in \mathcal{X}_{i}^{\rm hf} $, we denote by 
$\widetilde{\mathbf{w}} \in \widetilde{\mathcal{X}}_{L_i}^{\rm hf}$ the corresponding field in the reference configuration; the two fields share the same FE vector.
We introduce norms in the reference components
\begin{equation}
\label{eq:reference_norm_X}
\|   \mathbf{w} = {\rm vec} \left(
  \mathbf{u}, p \right)  \|_{\widetilde{\mathcal{X}}_k}^2
  =
  \int_{\widetilde{\Omega}^k} \nabla \mathbf{u} : 
  \nabla \mathbf{u}  + |\mathbf{u}|^2 + p^2 \, dx,
  \quad
   \vertiii{ \mathbf{s} = {\rm vec} \left(
  \mathbf{g}, h \right) }_{\widetilde{\Gamma}}^2
=
\int_{\widetilde{\Gamma}}  
\big| \nabla_{\widetilde{\Gamma}}   \mathbf{g}\big| ^2
+  | \mathbf{g} |^2
  + h^2 \, dx,
\end{equation}
for $k=1,\ldots,N_{\rm c}$. Then, we define the corresponding norms for the instantiated components that are obtained by applying the prescribed deformation
\begin{equation}
\label{eq:instantiated_norm_X}
\|   \mathbf{w}    \|_{{\mathcal{X}}_i} 
  :=
  \| 
\mathbf{w}|_{\Omega_i} \circ 
 \Phi^{L_i}(\cdot,\mu_i)   
      \|_{  \widetilde{\mathcal{X}}_{L_i}   },
\;\; i=1,\ldots,N_{\rm dd},
\quad
\vertiii{   \mathbf{s}  }^2
=
\sum_{j=1}^{N_{\rm f}}       
   \vertiii{ \mathbf{s}
|_{\Gamma_j} \circ 
 \Psi_j      }_{\widetilde{\Gamma}}^2.   
\end{equation}
%\|   \mathbf{w}    \|_{{\mathcal{X}}^{\rm dd}} = \sum_{k=1}^{N_{\rm dd}}              \|   \mathbf{w}    \|_{{\mathcal{X}}_i}^2 
Note that the algebraic norms associated with  \eqref{eq:instantiated_norm_X} are independent of the geometric  parameters that enter in the mappings $\{  \Phi^{L_i}(\cdot,\mu_i)     \}_i$: there exist   indeed  $N_{\rm c}$ matrices
$\mathbf{X}_{1},\ldots,
\mathbf{X}_{N_{\rm c}}$ such that
$\|   \mathbf{w}    \|_{{\mathcal{X}}_i} 
=
\sqrt{\underline{\mathbf{w}}^\top \mathbf{X}_{L_i} \underline{\mathbf{w}}}$ for $i=1,\ldots,N_{\rm dd}$.
This observation simplifies the implementation of the dual residual norm used in the adaptive strategy (cf. \eqref{eq:local_error}).
Similarly, the variational forms 
associated with the PDE problem are defined for each archetype component and then mapped to obtain the variational forms for each instantiated component.
We define the forms
$\widetilde{\mathcal{R}}_k^{\rm hf}:\widetilde{\mathcal{X}}_k^{\rm hf} \times \widetilde{\mathcal{X}}_{k,0}^{\rm hf} 
\times \mathcal{P}_k \times \mathbb{R}_+
\to \mathbb{R}$ such that 
\begin{equation}
\label{eq:change_variable_calR}
\widetilde{\mathcal{R}}_{L_i}^{\rm hf}( \widetilde{\mathbf{w}}, \widetilde{\mathbf{z}};  \mu_i, {\rm Re} )
=
{\mathcal{R}}_i^{\rm hf}( {\mathbf{w}}, {\mathbf{z}}),
\quad
\forall \, 
\mathbf{w}\in \mathcal{X}_{i}^{\rm hf},
\;\;
\mathbf{z} \in \mathcal{X}_{i,0}^{\rm hf}.
\end{equation}
We further define the boundary form
\begin{equation}
\label{eq:change_variable_calE}
\widetilde{\mathcal{E}}_{L_i,\ell}^{\rm hf}( \widetilde{\bs{\eta}}, \widetilde{\mathbf{z}}, \mu_i )
=
\int_{\Gamma_{j_{i,\ell}  }}
\widetilde{\bs{\eta}} \circ \Psi_j^{-1} \cdot 
\mathbf{z}  \circ \Phi_{i}^{-1}  \, dx
\quad
{\rm where} \;
\Phi_{i} :=  \Phi^{L_i}(\cdot; \mu_i),
\quad
\forall \, 
\widetilde{\bs{\eta}} \in L^2(\widetilde{\Gamma}_j; \mathbb{R}^3),
\;\;
\mathbf{z} \in \mathcal{X}_{i,0}^{\rm hf},
\end{equation}
where $j_{i,\ell} \in \{1,\ldots,N_{\rm f}\}$ is the  index 
(in the global numbering) 
of the $\ell$-th port of the $i$-th component of the system.

We have now the elements to present the DD Galerkin formulation:
\begin{subequations}
\label{eq:DD_formulation}
\begin{equation}
\label{eq:DD_formulation_a}
\min_{\mathbf{w} \in \mathcal{X}^{\rm dd}, \mathbf{s}\in \mathcal{W}^{\rm dd} } 
\frac{1}{2}  \int_{\Gamma} | \llbracket   \mathbf{w} \rrbracket  |^2 \, dx
+\frac{\delta}{2} \vertiii{\mathbf{s}}^2 \; \; {\rm s.t.} 
\; \; 
\mathcal{R}_i^{\rm hf}( \mathbf{w} ,  \mathbf{z}   )
+\mathcal{E}_i^{\rm hf}(\mathbf{s}, \mathbf{z})
= 0 \;\;\forall \, \mathbf{z}\in \mathcal{Z}_i,
\quad
i=1,\ldots,N_{\rm dd};
\end{equation}
where $\mathcal{Z}_i = \{ \boldsymbol{\zeta}\in [H^1(\Omega_i)]^3 \, : \,   \boldsymbol{\zeta}\circ   
 \Phi^{L_i}(\cdot,\mu_i) \in \widetilde{\mathcal{Z}}_{L_i}\}$ and
 \begin{equation}
\label{eq:DD_formulation_b}
\mathcal{E}_i^{\rm hf}(\mathbf{s} ,  \mathbf{z})
=
\sum_{j\in \mathcal{I}_i^{\Gamma}} \int_{\Gamma_j} 
\mathbf{s}  \cdot  \mathbf{z} \, dx,
\end{equation}
 for $i=1,\ldots,N_{\rm dd}$.
Formulation \eqref{eq:DD_formulation_a} can be  adapted to cope with Petrov-Galerkin ROMs using the strategy outlined in   section \ref{sec:hybrid_solver}: we omit the details.
\end{subequations}

Given the estimate $(\mathbf{w}^\star,\mathbf{s}^\star )$ of the solution to \eqref{eq:DD_formulation_a}, we devise two error indicators to assess its accuracy;  the indicators are employed in section \ref{sec:adaptive_enrichment} to drive the enrichment strategy.
First, we define the local errors
\begin{equation}
\label{eq:local_error}
e_i:= \sup_{\mathbf{z}\in \mathcal{X}_{i,0}^{\rm hf}}
\frac{
\mathcal{R}_i^{\rm hf}( \mathbf{w}^\star ,  \mathbf{z}   )
+\mathcal{E}_i^{\rm hf}(\mathbf{s}^\star ,  \mathbf{z})
}{\| \mathbf{z}  \|_{{\mathcal{X}}_i} },
\quad
i=1,\ldots,N_{\rm dd}.
\end{equation}
The quantity $e_i$ measures the performance of the $i$-th ROM to approximate the solution to the Navier-Stokes equations for the control $\mathbf{s}^\star$. We further introduce the jump errors:
\begin{equation}
\label{eq:jump_error}
e_j^{\rm jump}:=
\sqrt{
\int_{\Gamma_j} | \llbracket   \mathbf{w} \rrbracket  |^2 \, dx},
\quad
j=1,\ldots,N_{\rm f}.
\end{equation}
The indicator \eqref{eq:jump_error} controls the jump of the state estimate at the interfaces: the value of 
$e_j^{\rm jump}$ can thus be interpreted as the measure of the ability of the control to nullify the jump at the $j$-th interface of the domain. 

\begin{remark}
In order to enhance the compressibility of the local state and control manifolds, following \cite{bui2023component}, in the numerical experiments, we   consider the approximation spaces
\begin{subequations}
\label{eq:rotated_space}
\begin{equation}
\label{eq:rotated_space_a}
\begin{array}{l}
\displaystyle{
\mathcal{X}^{\rm dd}:=\left\{
\mathbf{w}\in L^2(\Omega; \mathbb{R}^3) \,:\,
\mathbf{A}(\theta_i) \mathbf{w}|_{\Omega_i} \circ 
 \Phi^{L_i}(\cdot,\mu_i) \in \widetilde{\mathcal{Z}}_{L_i}^{\rm dir}, \;\; i=1,\ldots,N_{\rm dd}
\right\};
}
\\[3mm]
\displaystyle{
\mathcal{W}^{\rm dd}:=\left\{
\mathbf{s}\in L^2(\Gamma; \mathbb{R}^3) \,:\,
\mathbf{A}(\omega_j) \mathbf{s}|_{\Gamma_j} \circ 
 \Psi_j  \in \widetilde{\mathcal{W}}, \;\; j=1,\ldots,N_{\rm f}
\right\};
}
\\ 
\end{array}
\end{equation}
where $\theta_i$ (resp., $\omega_i$) is the   angle between  the inlet
port of the $i$-th deformed component 
$\Omega_i$
 (resp., the $j$-th port $\Gamma_j$)
and the $x_1$ axis, and
\begin{equation}
\label{eq:rotated_space_b}
\mathbf{A}(\theta)=\left[
\begin{array}{ccc}
\cos(\theta) & -\sin(\theta)  & 0 \\
\sin(\theta) &  \cos(\theta)  & 0 \\
0 & 0& 1 \\
\end{array}
\right].
\end{equation}
\end{subequations}
We remark that 
several authors have considered more sophisticated (Piola) transformations to improve  the compressibility of solution manifolds in internal flows, (e.g.  \cite{Lovgren_Maday_Ronquist_2006}): in this respect, our choice is a compromise between accuracy and simplicity of implementation.
\end{remark}

\subsection{Pairwise training for the control variables}
\label{sec:pairwise_control} 
Following 
\cite{Eftang_Patera_2013, benaceur2022port},
we pursue a pairwise-training approach to generate the port space $\widetilde{\mathcal{W}}$. We perform HF simulations for systems of two components that represent all possible connections (channel-channel, channel-junction, junction-junction, junction-channel) based on random Dirichlet boundary conditions at the inflow, random Neumann conditions at the outflow, and a random selection of the Reynolds number and the geometric parameters in prescribed parameter ranges (cf. Figure \ref{fig:port_training_example}).
The HF data for the ports are   retrieved and stored, and finally the port space $\widetilde{\mathcal{W}}$ is  constructed using  POD.
Recalling \eqref{eq:rotated_space}, the HF data are rotated using \eqref{eq:rotated_space_b} before applying the compression technique.

Similarly to \cite{Eftang_Patera_2013,hoang2021domain}, we consider the inlet velocity
\begin{equation}
\label{eq:dirichlet_random}
\mathbf{u}_{\rm in}(y)
=
- 
\frac{{\rm Re}}{{\rm Re}_{\rm ref}}
\left( u_0(y) + 
\delta_u \sum_{k=1}^R 
\frac{c_k}{k^2}
P_k(-1+2y)
 \right) \mathbf{n},
\end{equation}
where 
${\rm Re}\sim {\rm Uniform}({\rm Re}_{\rm min}, {\rm Re}_{\rm max})$, $\{ P_k \}_k$ are zero-flowrate weighted polynomials
(cf. \cite[section 3.1.1]{benaceur2022port})
$$
P_k(y)=\left\{
\begin{array}{ll}
(1-y^2)y, & \text{if } k=1, \\
(1-y^2)(5y^2-1), & \text{if } k=2, \\
(1-y^2)\mathfrak{L}_k(y), & \text{if } 3\leq k\leq R,
\end{array}
\right.
$$
and $\{ \mathfrak{L}_k \}_k$ are the  Legendre polynomials.
The coefficients of the expansion are sampled from a standard Gaussian distribution, $c_1,\ldots,c_R \overset{\rm iid}{\sim} \mathcal{N}(0,1)$,
$\mathbf{n}$ denotes the outward normal to $\Omega$ on the inlet boundary,
$y\in (0,1)$ is the curvilinear coordinate, 
$u_0(y)= 4 y(1-y)$ is the Poiseuille velocity profile,  
the coefficient $\delta_u$  is selected \emph{a posteriori} to  ensure  that the inflow is positive for all $y\in (0,1)$.
Similarly, we prescribe the outward flux as
\begin{equation}
\label{eq:neumann_random}
\mathbf{g}_{\rm out}(y)
=
\left( g_0 + \delta_g \sum_{k=1}^R c_k^{\rm out}  \mathfrak{L}_k(-1+2y)
 \right) \mathbf{n},
 \quad
 c_1^{\rm out}, \ldots, 
c_R^{\rm out} \overset{\rm iid}{\sim} \mathcal{N}(0,1),
\end{equation}
where $g_0 \sim {\rm Uniform}({ g_0}_{\rm min}, {g_0}_{\rm max})$, and we choose the coefficient $\delta_g$ to prevent reverse flow.

\begin{figure}[H]
\includegraphics[width=9cm]{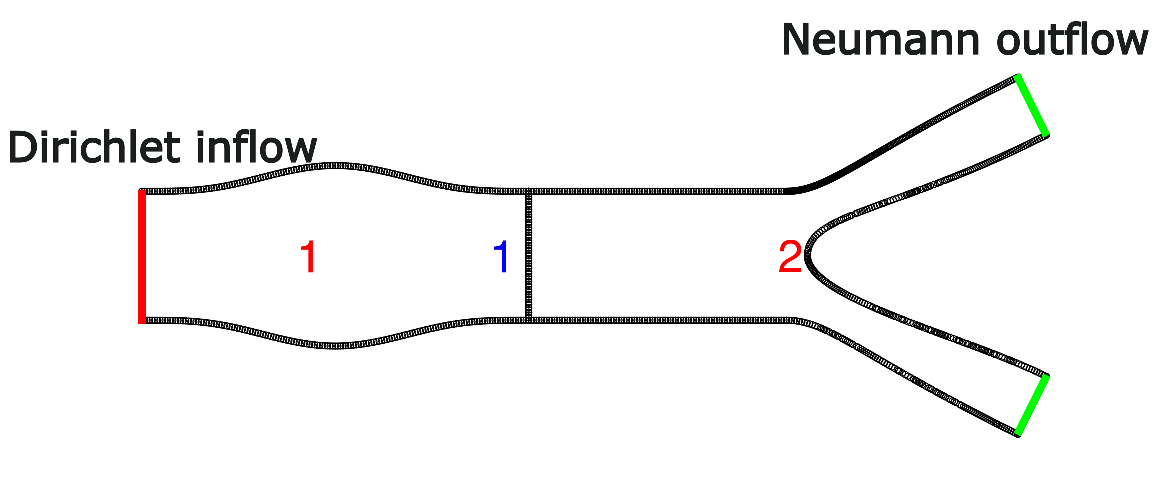}
\centering
\caption{channel-junction connection for the training of $\mathbf{s}$.}
\label{fig:port_training_example}
\end{figure}

\subsection{Localized training  for the state variables}
\label{sec:localized_state}
After having built the reduced space for the control, we repeatedly solve \eqref{eq:DD_formulation_a} for several random configurations and several parameter values to acquire datasets of simulations for each archetype component. 
Thanks to port reduction, the computational cost of the global problem is significantly reduced if compared with the full HF model; nevertheless, we choose to consider systems with a moderate number of components (up to four) to further reduce offline costs. 
The HF data for components of the same type are mapped in the reference configurations,
rotated through \eqref{eq:rotated_space_b}, 
 and are then used to build the local reduced spaces
$\widetilde{\mathcal{Z}}_1, \ldots,\widetilde{\mathcal{Z}}_{N_{\rm c}}$.

We observe that the training strategy is not fully local since it requires to assemble systems with up to four components. In our experience, the practical implementation of a fully localized training strategy 
for incompressible flows
is extremely challenging   due to the need to ensure that the fluid flows from left to right and that the prescribed Neumann conditions  lead to physical velocities.
The choice of considering   global training based on a reduced control space for systems of moderate dimension represents a trade-off between offline efficiency and accuracy. The adaptive strategy  presented in the next section provides a systematic way to improve the quality of the local reduced spaces.

\subsection{Adaptive enrichment}
\label{sec:adaptive_enrichment}

In Algorithm \ref{alg:online_enrichment}, we present the full adaptive strategy for the construction of the reduced spaces.
The procedure extends the method introduced in 
\cite{Smetana2023}; 
to clarify the presentation, we postpone two steps of the algorithm to sections \ref{sec:local_solves} and \ref{sec:enrichment_state_spaces}.

\begin{algorithm}[H]                      
\caption{{Adaptive} enrichment procedure.}     
\label{alg:online_enrichment}     

\begin{algorithmic}[1]
\State
Generate the reduced space $\widetilde{\mathcal{W}}$ for the control through pairwise training 
(cf. section \ref{sec:pairwise_control}).
\smallskip

\State
Generate the local spaces $\{  \widetilde{\mathcal{Z}}_k \}_{k=1}^{N_{\rm c}}$ for the state through global training 
(cf. section \ref{sec:localized_state}).
\smallskip

\State
Enrich the reduced spaces $\{  \widetilde{\mathcal{Z}}_k \}_{k=1}^{N_{\rm c}}$ based on the port space $\widetilde{\mathcal{W}}$
(cf. section \ref{sec:enrichment_state_spaces}).
\smallskip

\State
Sample $n_{\rm train}^{\rm glo}$ global configurations, $\mathscr{P}_{\rm train}:=\{\mu^{j}\}_{j=1}^{n_{\rm train}^{\rm glo}}$.
\smallskip

\State
(if LSPG projection is employed)
Build the empirical test space (cf. section \ref{sec:local_ROM})
\smallskip

\For {$it=1, \ldots, \texttt{maxit}$ }
\State
Initialize the datasets 
$\mathscr{D}_{(1)}=\ldots = \mathscr{D}_{(N_{\rm c})} =\emptyset$ and $\mathscr{D}_{{\mathbf{s}}}=\emptyset$.

\For {$\mu \in \mathscr{P}_{\rm train}$ }
\State
Compute the reduced solution using the CB-ROM solver (cf. \eqref{eq:DD_formulation_a}).
\smallskip

\State
Compute local residuals $\{ e_i \}_{i=1}^{N_{\rm dd}^\mu}$ (cf. \eqref{eq:local_error})
and the jumps  $\{ e_j^{\rm port}
\}_{j=1}^{N_{\rm f}^\mu}$ (cf. \eqref{eq:jump_error})
\smallskip

\State
Mark the $m_{\mathbf{w}}$ instantiated components with the largest residuals of each type $\{ \mathtt{I}_{\rm mark}^{\mu,(k)} \}_{k=1}^{N_{\rm c}}$.
\smallskip

\State
Mark the $m_{\mathbf{s}}$ instantiated ports with the largest port jumps of each type $\mathtt{I}_{\rm mark}^{\mu,\rm p}$.
\State
Update the datasets 
$\mathscr{D}_{(1)}, \ldots ,  \mathscr{D}_{(N_{\rm c})}$ and $\mathscr{D}_{{\mathbf{s}}}$
(cf. section \ref{sec:local_solves})
\EndFor

\State
Update the port  POD space $\widetilde{\mathcal{W}}=
\widetilde{\mathcal{W}} \oplus 
\text{POD}\left(\left\{ \Pi_{\widetilde{\mathcal{W}}^\perp}
\widetilde{\mathbf{s}}:  \widetilde{\mathbf{s}} \in \mathscr{D}_{{\mathbf{s}}}\right\},
\vertiii{\cdot}_{\widetilde{\Gamma}}, n^{\rm glo}\right)$.
\smallskip

\State
Update the reduced    spaces
 $\widetilde{\mathcal{Z}}_{k}=
\widetilde{\mathcal{Z}}_{k} \oplus
\text{POD}\left(\left\{ \Pi_{\widetilde{\mathcal{Z}}_k^\perp}
\widetilde{\mathbf{w}}:  \widetilde{\mathbf{w}} \in \mathscr{D}_{(k)}\right\},
\| \cdot \|_{\widetilde{\mathcal{X}_k}}, n^{\rm glo}\right)$, 
  $k=1,\ldots, N_{\rm c}$.
  \smallskip
  
\State
(Optional)
Enrich the reduced spaces $\{  \widetilde{\mathcal{Z}}_k \}_{k=1}^{N_{\rm c}}$ based on the port space
(cf. section \ref{sec:enrichment_state_spaces}).
\smallskip

\State
(if LSPG projection is employed)
Update the empirical test space (cf. section \ref{sec:local_ROM})

\EndFor
\end{algorithmic}
\end{algorithm}

As in 
\cite{Smetana2023}, we add $m_{\mathbf{w}}$ (resp., $m_{\mathbf{s}}$) snapshots to the
state (resp., control) datasets for each element of $\mu\in \mathcal{P}_{\rm train}$, instead of selecting the marked elements after having computed the local indicators for all configurations: this choice avoids the storage of all reduced global solutions and ultimately simplifies the implementation. In our experience, the enrichment of the state spaces is only needed for localized training (Line 3 of the Algorithm) but not after each update of the control space $\widetilde{\mathcal{W}}$ (Line 17 of the Algorithm): a possible explanation is that the enrichment step inherently couples the construction of the two spaces. Further numerical investigations are necessary to investigate this aspect.

Algorithm \ref{alg:online_enrichment} depends on several user-defined parameters.
The localized training of the control space depends on 
(i) the 
sampling distributions for the Dirichlet inflow boundary condition \eqref{eq:dirichlet_random} and for the Neumann outflow condition \eqref{eq:neumann_random};
(ii) the number $n_{\rm loc}^{\mathbf{s}}$ of samples; and 
(iii) the number $m_0$ of retained POD modes.
The localized training for the state variables depends on 
(i) the number $N_{\rm dd}$ components of the networks considered;
(ii) the number $n_{\rm loc}^{\mathbf{w}}$ of samples; and 
(iii) the number $n_0$ of retained POD modes for each archetype component.
The enrichment strategy depends on (i) the number $n'$ of added modes (cf. section \ref{sec:enrichment_state_spaces}).
The adaptive loop depends on
(i) the number $\texttt{maxit}$ of outlet loop iterations;
(ii) the number $n_{\rm train}^{\rm glo}$ of global configurations;
(iii) the numbers $m_{\mathbf{w}}$ and $m_{\mathbf{s}}$ of marked components and ports;
(iv) the number $n^{\rm glo},m^{\rm glo}$ of modes added at each iteration for state and control variables.
We envision that the selection of several  parameters can be automated: to provide a concrete reference, the parameters 
$n^{\rm glo},m^{\rm glo}$ can be updated based on a energy/projection criterion. Nevertheless, further investigations are necessary to provide actionable guidelines to select all the parameters.

\subsubsection{Computation of the local solutions}
\label{sec:local_solves}
Given the sampled port $\Gamma_j$, we solve the HF model with flux boundary conditions given by the control $\mathbf{s}^\star$ on the remaining port (cf. Figure \ref{fig:localsolve_example}) in the domain $\Omega^\star = \Omega_j^+\cup \Omega_j^-$ where 
$\Omega_j^+,  \Omega_j^-$ are the elements of the network that share $\Gamma_j$. 
Given the sampled component $\Omega_i$, we consider two separate strategies:
(i) we solve the global hybrid model in which we replace the local ROM with the local HF model in the sampled component, or 
(ii) we solve the HF model in the sampled component with boundary conditions prescribed by the control estimate $\mathbf{s}^\star$.
The first option is significantly less computationally expensive; however, we experienced some convergence issues for very inaccurate controls $\mathbf{s}^\star$. For this reason, in the numerical experiments,  we rely  on global hybrid solves for the first iteration of the algorithm and to fully local solves for the subsequent iterations. 

\begin{figure}[H]
\centering
\includegraphics[width=8.3cm]{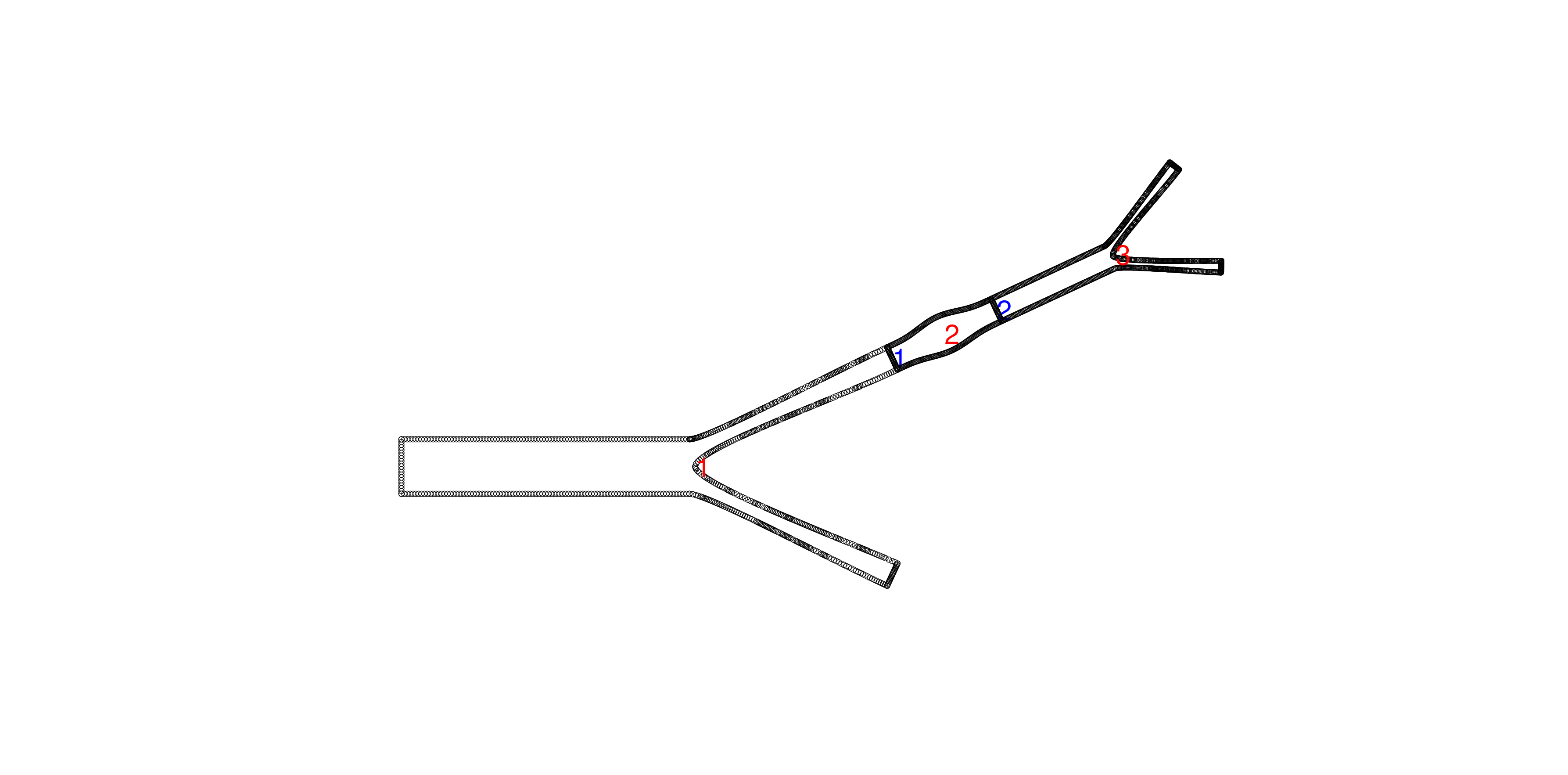}
 
\caption{computation of local solution.
Port update: if $\Gamma_1$ is the sampled port, we solve the HF model in the components $\Omega_1\cup \Omega_2$ with Neumann boundary conditions on the port $\Gamma_{2}$ given by the predicted control $\mathbf{s}^\star$ (cf. Line 9, Algorithm \ref{alg:online_enrichment}). State update: if $\Omega_2$ is the sampled component, we either solve the global problem using the HF discretization in $\Omega_2$ and the ROM discretization in $\Omega_3$ (option 1), or we solve the HF model in $\Omega_2$ with Neumann boundary conditions on the ports $\Gamma_1$ and $\Gamma_2$ given by  $\mathbf{s}^\star$ (option 2).}
\label{fig:localsolve_example}
\end{figure}

\subsubsection{Enrichment of the state spaces}
\label{sec:enrichment_state_spaces} 
It suffices to generalize the procedure of section \ref{sec:enrichment_basis}. We denote by $\{  \widetilde{\mathbf{w}}_\ell^k  \}_{\ell=1}^{n_{\rm train}^k}$ a dataset of snapshots associated with the $k$-th archetype component and the local parameters 
$\{  \mu_\ell^k  \}_{\ell=1}^{n_{\rm train}^k}$ and
$\{ {\rm Re}_\ell  \}_{\ell=1}^{n_{\rm train}^k}$.
The dataset  $\{  \widetilde{\mathbf{w}}_\ell^k  \}_{\ell=1}^{n_{\rm train}^k}$ is extracted by the global simulations performed in the internal loop (cf. Lines $8-14$) of Algorithm \ref{alg:online_enrichment} or from the simulations performed to generate the initial local space $\widetilde{\mathcal{Z}}_k$ (cf.  Line $2$).
 We denote by $\{ \boldsymbol{\eta}_{j}'  \}_{j=1}^m$ the newly-added modes of the port space;
we further recall the definitions of the 
local residuals and 
\eqref{eq:change_variable_calR}
boundary forms  
\eqref{eq:change_variable_calE}.
%  we denote by $\widetilde{\Gamma}_q$.
Then, we define $\widetilde{\mathbf{w}}_{\ell,j,q}^k$ such that
(compare with 
\eqref{eq:enriching_modes})
$$
\widetilde{\mathcal{R}}_{k}^{\rm hf}(\widetilde{\mathbf{w}}_\ell^k,  \mathbf{z};
\mu_\ell^k, 
{\rm Re}_\ell )
  +
\frac{\partial \widetilde{\mathcal{R}}_{k}^{\rm hf  }}{\partial \widetilde{\mathbf{w}}_{k}}\left[ 
\widetilde{\mathbf{w}}_\ell^k, \mu_\ell^k, 
{\rm Re}_\ell 
  \right]
\left( 
\widetilde{\mathbf{w}}_{\ell,j,q}^k
+ 
\widetilde{\bs{\Psi}}_{k,\mathbf{u}_{\rm in}}({\rm Re}_{\ell}) 
- 
\widetilde{\mathbf{w}}_\ell^k
\;  ,   \;    \mathbf{z}     \right)
  + \widetilde{\mathcal{E}}_{k,q}^{\rm hf}( \boldsymbol{\eta}_{j}' 
  \;  ,   \;       
  \mathbf{z}   )
  =
  0
  \quad
  \forall \, \mathbf{z}   \in \widetilde{\mathcal{X}}_{k,0}^{\rm hf},
$$
for 
$\ell=1, \ldots, n_{\rm train}^k$, 
$j=1,\ldots,m$, and $q=1,\ldots,N_{\rm port}^k$
($N_{\rm port}^k=2$ for the channel component, and $N_{\rm port}^k=3$ for the junction component). 
After having computed the snapshots $\{ \widetilde{\mathbf{w}}_{\ell,j,q}^k  \}_{\ell,j,q}$, we update the reduced space 
$\widetilde{\mathcal{Z}}_k$ 
with $n'$ modes 
using POD (cf.
\eqref{eq:hierarchicalPOD}).

\section{Numerical results}
\label{sec:numerical_res}
 We present  numerical results of the proposed method for the parameterized incompressible flow of 
 section \ref{sec:model_pb}. The  parameters are the Reynolds number and the geometric parameters $\alpha$ and $h_c$ introduced for each instantiated component.
 We consider  a \texttt{P}$2$ FE discretization with $1281$ degrees of freedom for the channel, and $3329$ degrees of freedom for the junction. The regularization constant $\delta$ is set equal  to $10^{-8}$.

\subsection{HF solver}
\label{sec:hf_results}
We present the HF results for the Reynolds number $\text{Re}=100$ and the geometric configuration shown in Figure \ref{fig:instantiation_example}(b). 
In Figure \ref{fig:hf_new_control}(a)-(b)-(c), we show the solution to the global HF problem (i.e., without domain decomposition) for the x-direction velocity, y-direction velocity, and the pressure, respectively. Figures \ref{fig:hf_new_control}(d)-(e)-(f) illustrate the difference between the solution to the global problem and the solution
to the (multi-component generalization of the) DD formulation  \eqref{eq:lsq_discr_final_a}. 
Our new formulation  exhibits high accuracy, with a pointwise error of the order of $10^{-6}$ for the three variables. Here, we employ the SQP method introduced in section \ref{sec:FOM_sqp}; GNM (cf. section \ref{sec:FOM_gnm}) does not converge for this value of  the  Reynolds number.
For the solution to the DD problem, the  global prediction  at the interfaces is obtained by averaging   the solution in the two neighboring sub-domains.

\begin{figure}[H]
  \centering
  
\subfloat[$u_x^{\rm fe}$]{\includegraphics[width=0.33\textwidth]{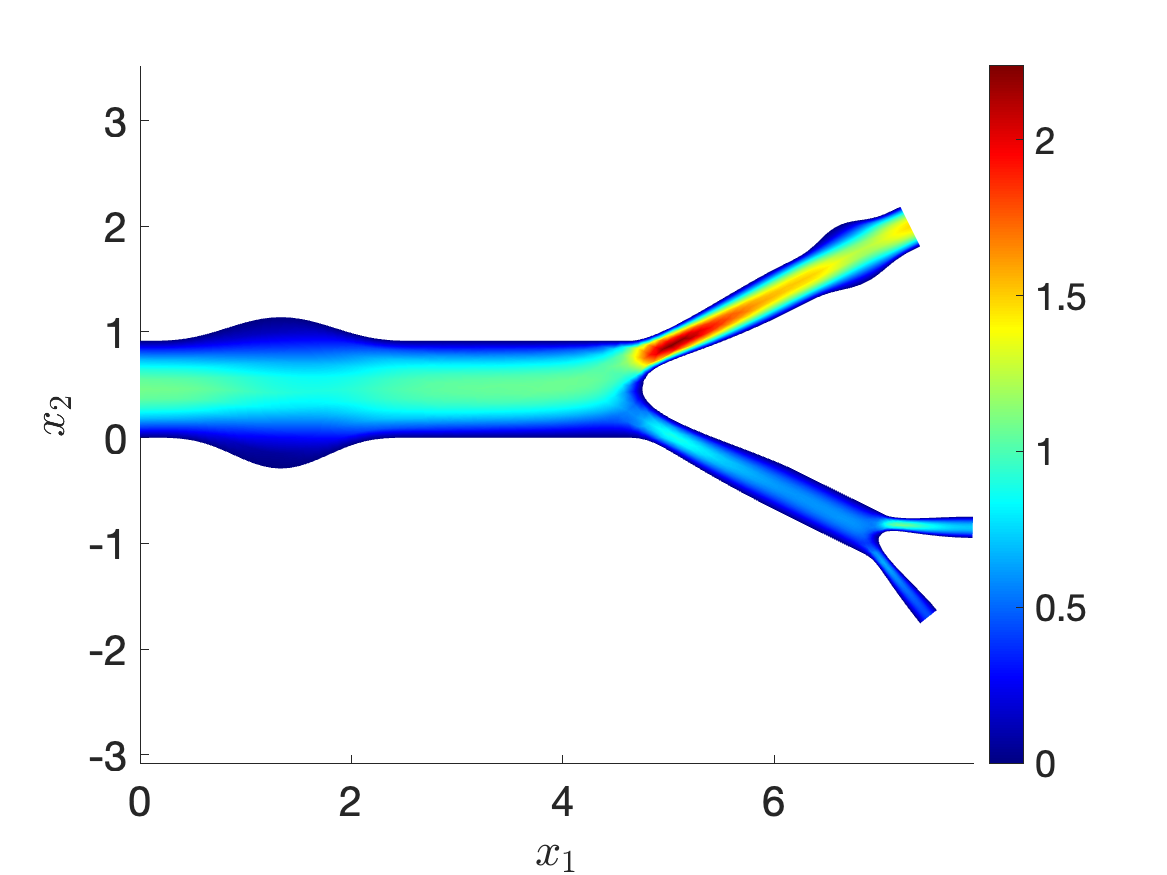}}
~~
\subfloat[$u_y^{\rm fe}$]{\includegraphics[width=0.33\textwidth]{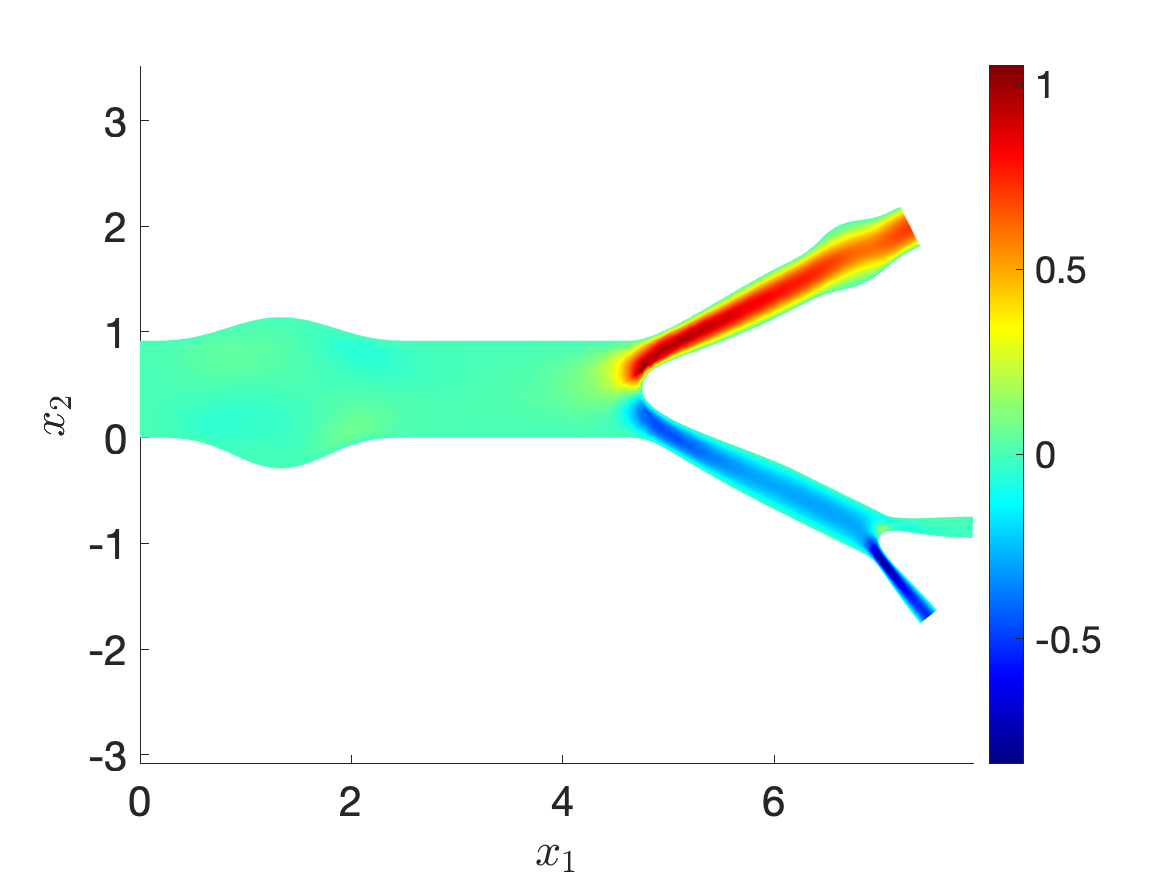}}
~~
\subfloat[$p^{\rm fe}$]{\includegraphics[width=0.33\textwidth]{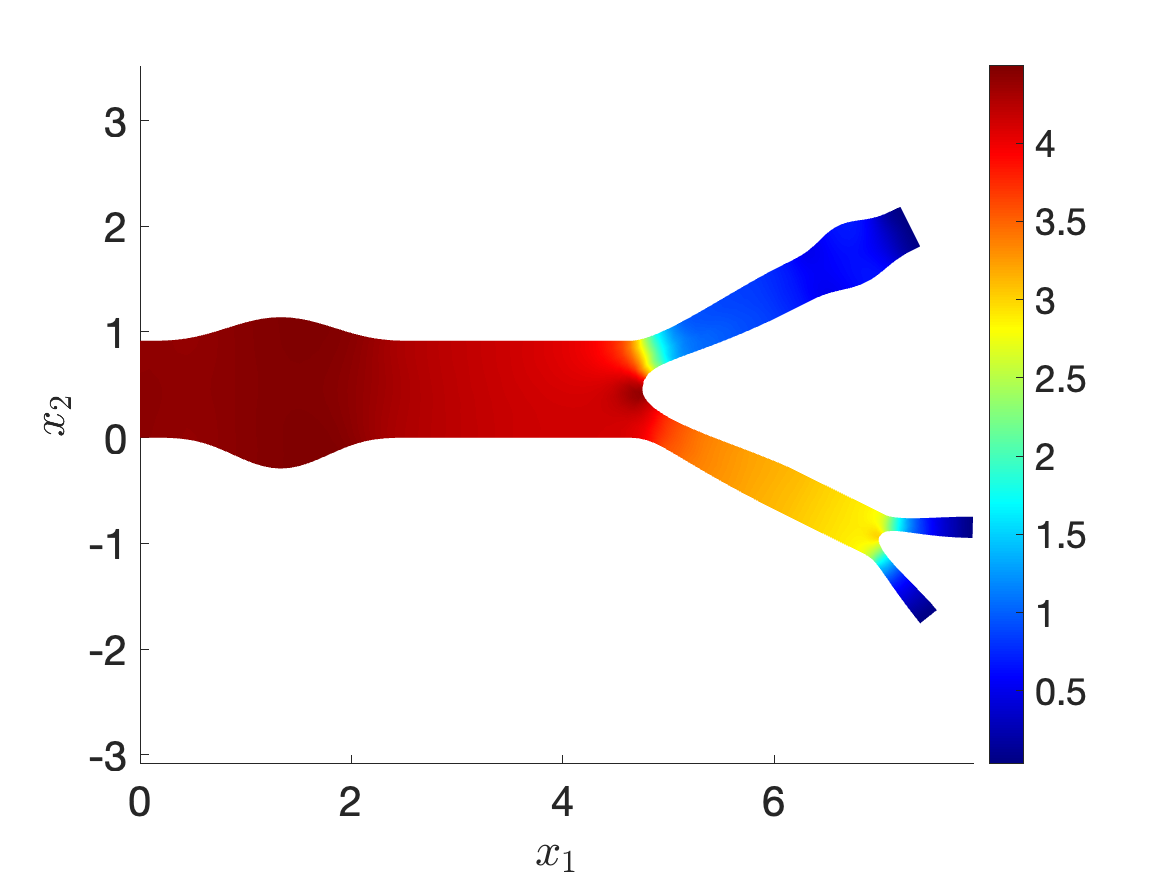}}

\subfloat[$u_x^{\rm fe}-u_x^{\rm dd}$]{\includegraphics[width=0.33\textwidth]{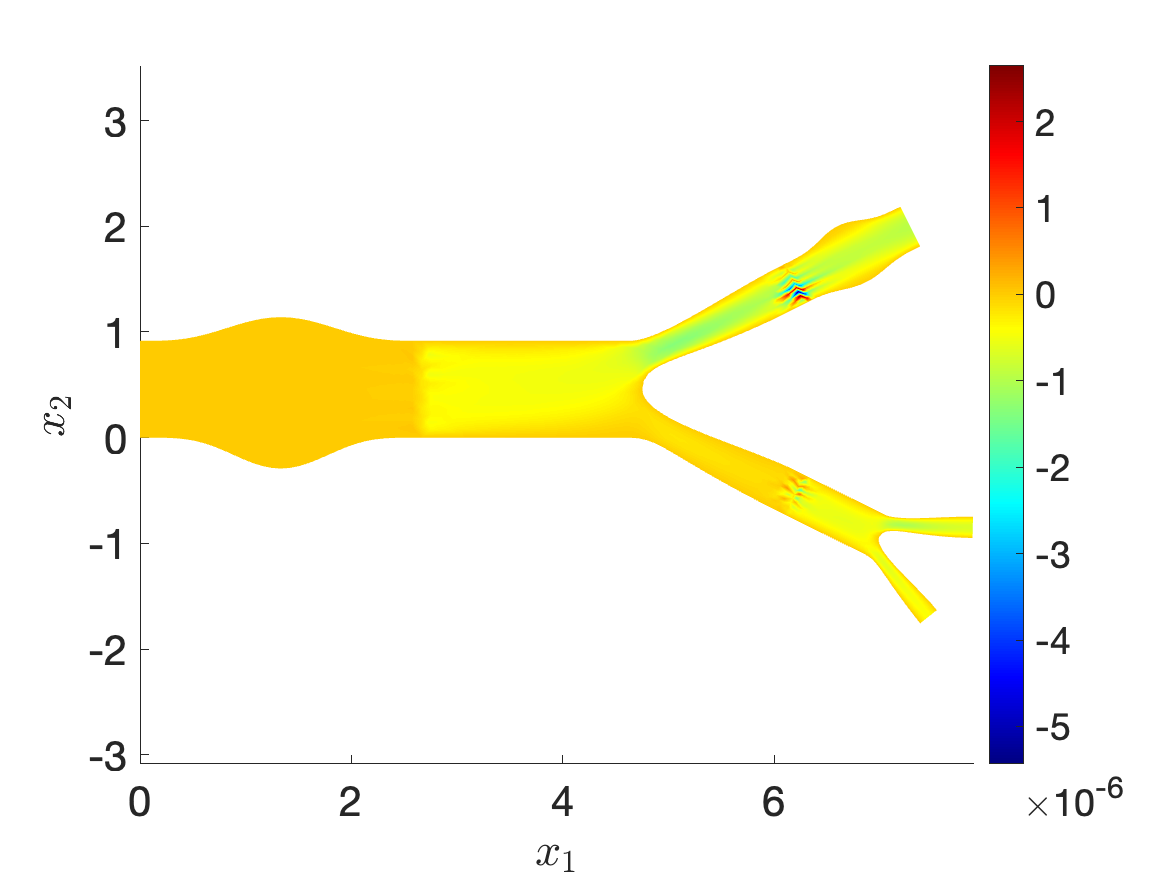}}
~~
\subfloat[$u_y^{\rm fe}-u_y^{\rm dd}$]{\includegraphics[width=0.33\textwidth]{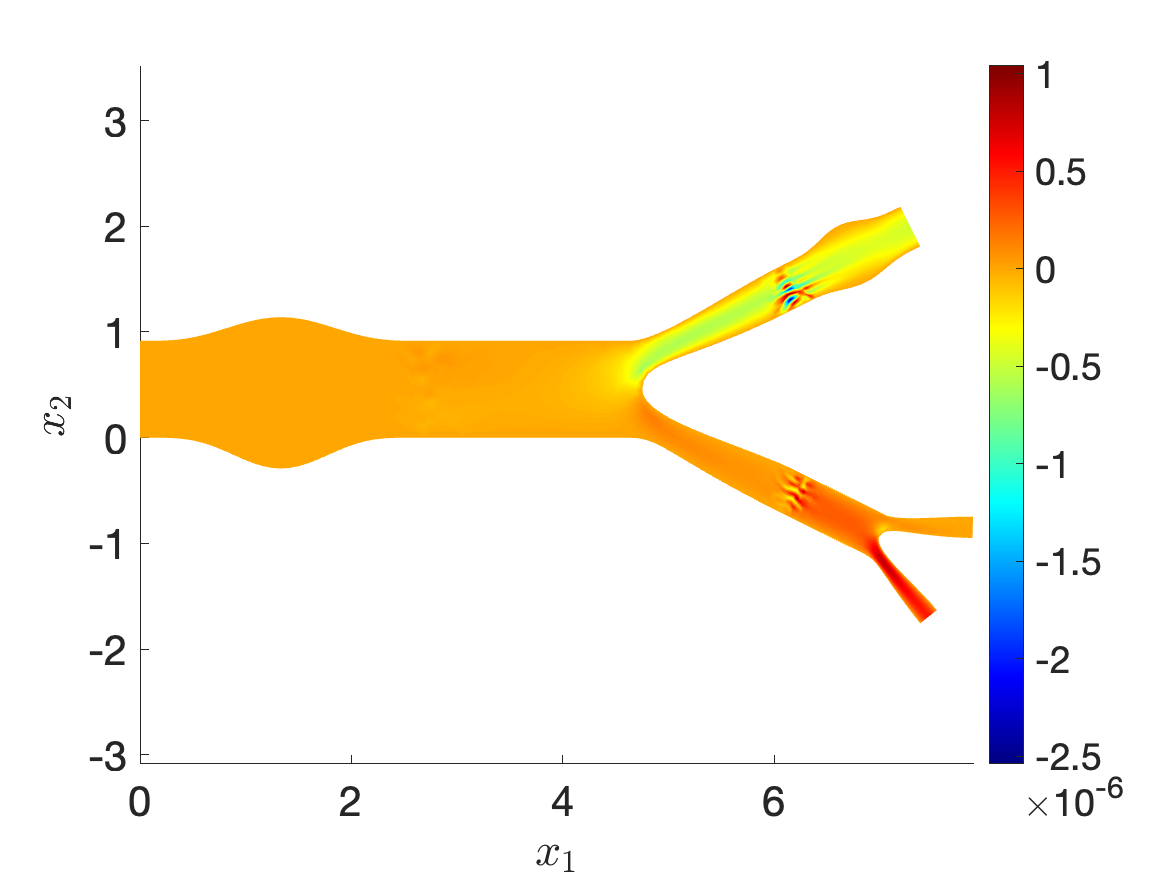}}
  ~~
\subfloat[$p^{\rm fe}-p^{\rm dd}$]{\includegraphics[width=0.33\textwidth]{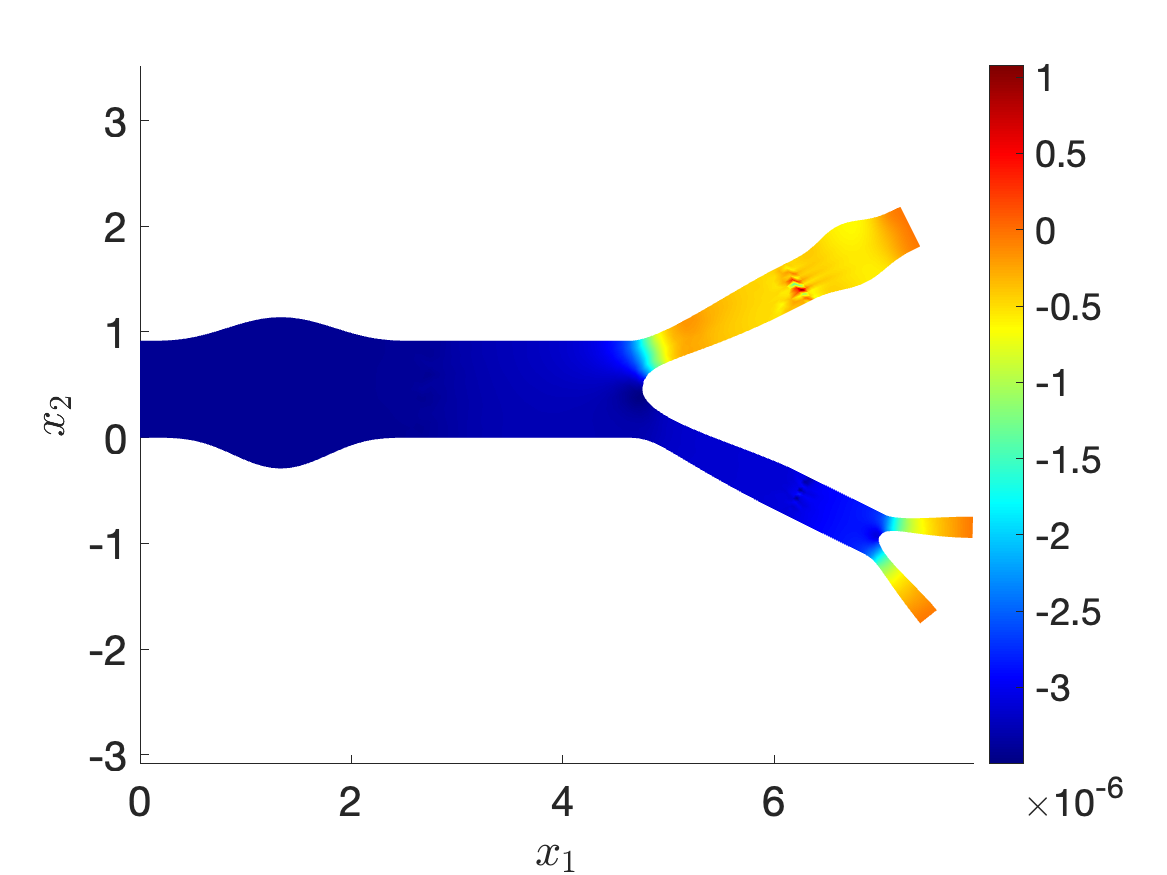}}
  
  \caption{HF formulation.
(a)-(b)-(c) Behavior of the solution to the monolithic FE problem.
(d)-(e)-(f) difference between the monolithic FE solution and the DD solution based on \eqref{eq:lsq_discr_final_a}.}
  \label{fig:hf_new_control}
\end{figure}

In Figure \ref{fig:hf_old_control}, we present the comparison between the monolithic FE solution and the solution  to the DD formulation \eqref{eq:lsq_discr_standard}. The results 
of Figure \ref{fig:hf_old_control} show much larger pointwise errors for both velocity and pressure --- the error for the pressure is $\mathcal{O}(10^{-2})$ as opposed to $\mathcal{O}(10^{-6})$.
This result justifies the addition of the control $h$ for the continuity equation.

In Figure \ref{fig:hf_jump}, we present the variable jump across the interfaces for the new formulation
\eqref{eq:lsq_discr_final_a} and the standard formulation \eqref{eq:lsq_discr_standard}.
For \eqref{eq:lsq_discr_standard}, the jump of the velocity field is modest, but it is  significant ($\mathcal{O}(10^{-1})$) for the pressure.
In contrast, 
for \eqref{eq:lsq_discr_final_a}, 
the jump of both velocity and pressure is extremely modest. These results further corroborate the introduction of the control   $h$ for the continuity equation.

\begin{figure}[H]
  \centering

  \subfloat[$u_x^{\rm fe}-u_x^{\rm dd}$]{\includegraphics[width=0.33\textwidth]{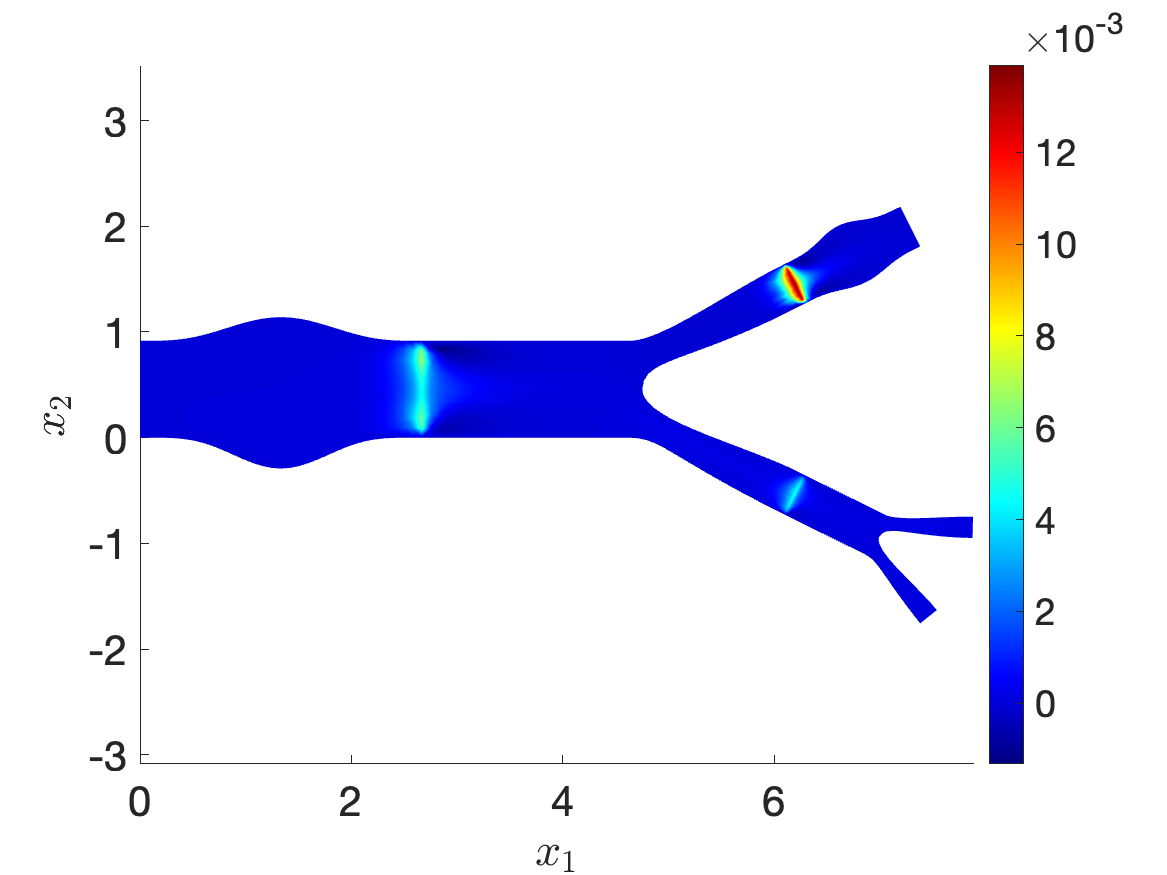}}
~~
  \subfloat[$u_y^{\rm fe}-u_y^{\rm dd}$]{\includegraphics[width=0.33\textwidth]{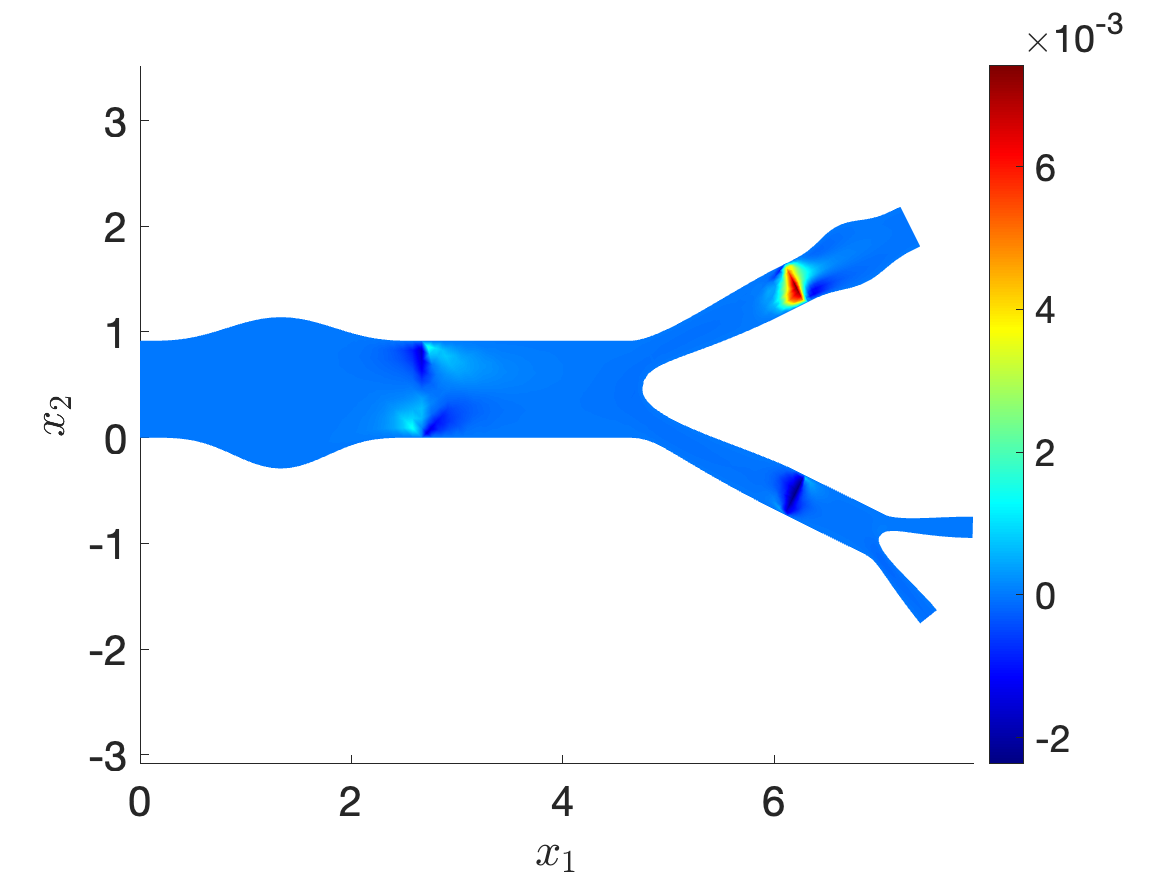}}
  ~~
  \subfloat[$p_x^{\rm fe}-p_x^{\rm dd}$]{\includegraphics[width=0.33\textwidth]{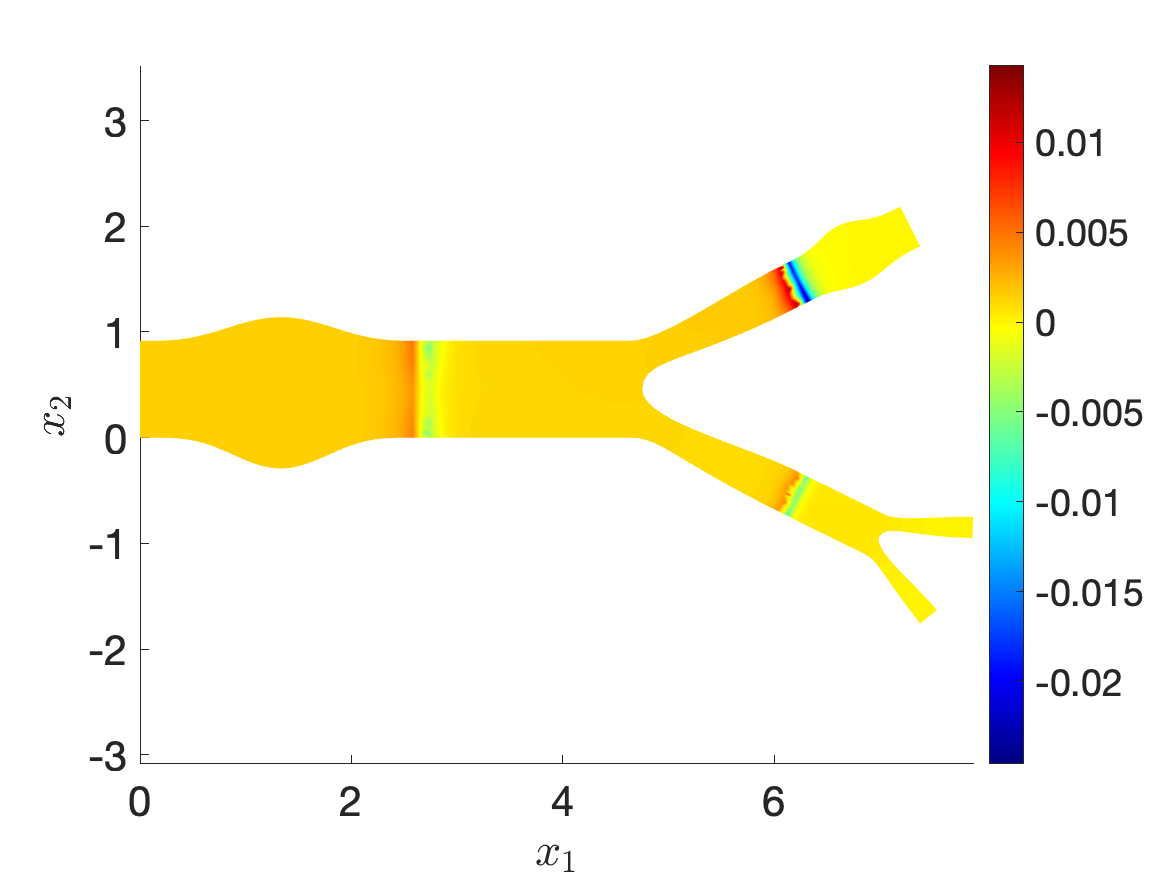}}
  
  \caption{HF formulation.
Difference between the monolithic FE solution and the DD solution based on   
\eqref{eq:lsq_discr_standard}.}
  \label{fig:hf_old_control}
\end{figure}

\begin{figure}[H]
  \centering
  
\subfloat[$\llbracket  u_x \rrbracket$]{\includegraphics[width=0.33\textwidth]{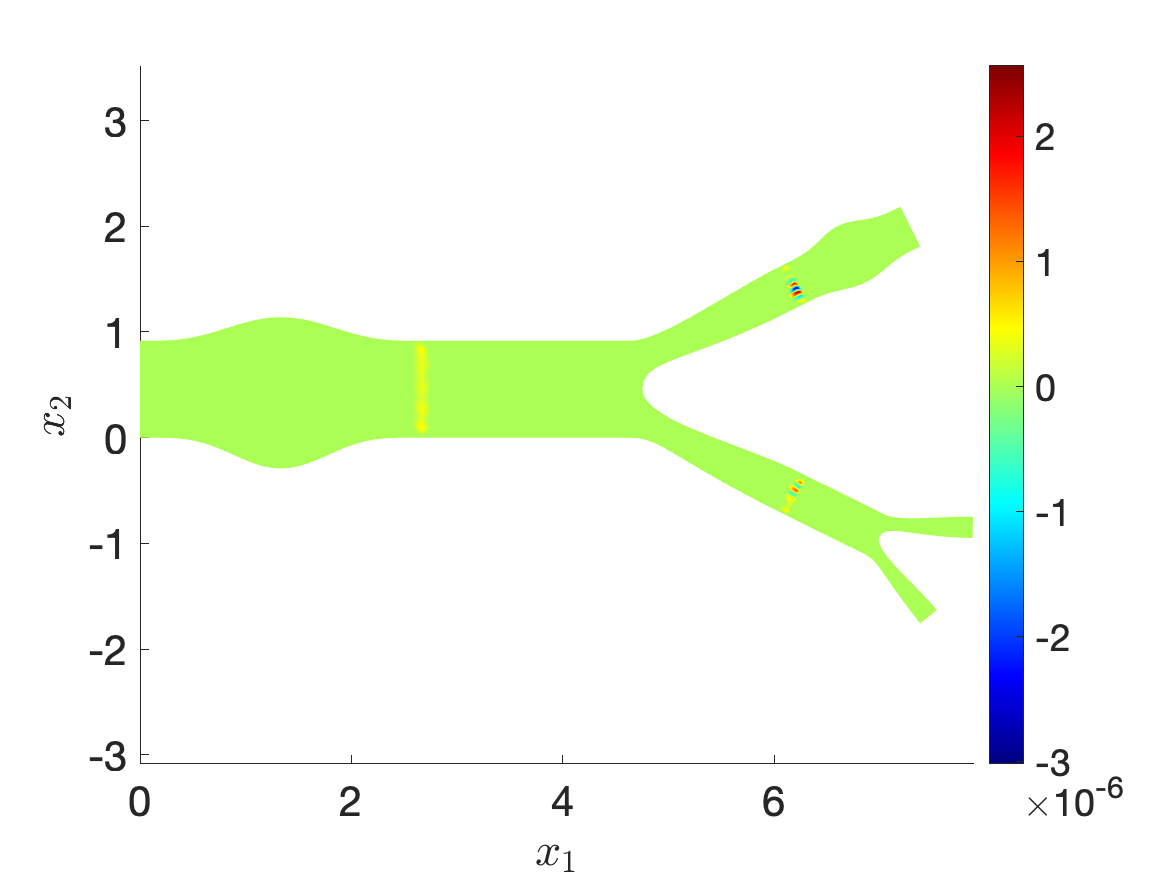}\label{fig:subfigA}}
~~
\subfloat[$\llbracket  u_y \rrbracket$]{\includegraphics[width=0.33\textwidth]{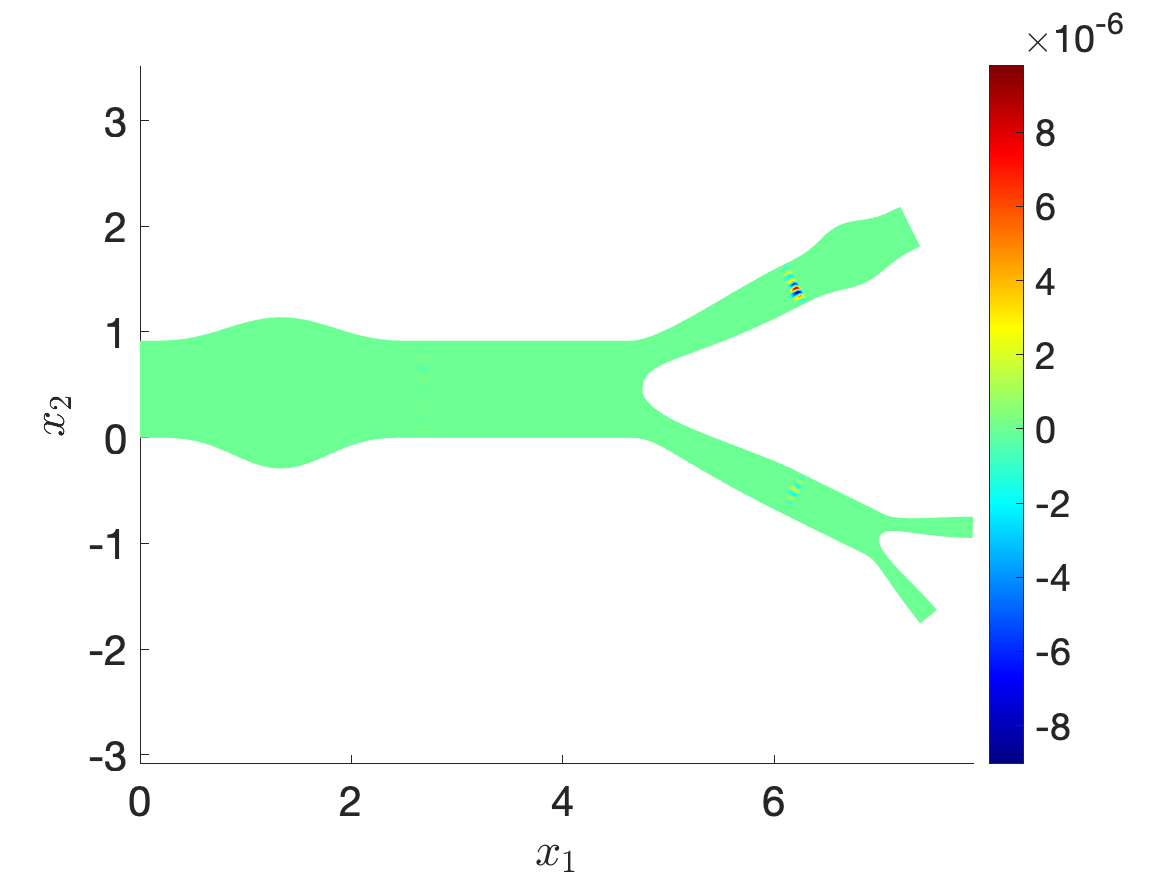}}
 ~~
\subfloat[$\llbracket  p \rrbracket$]{\includegraphics[width=0.33\textwidth]{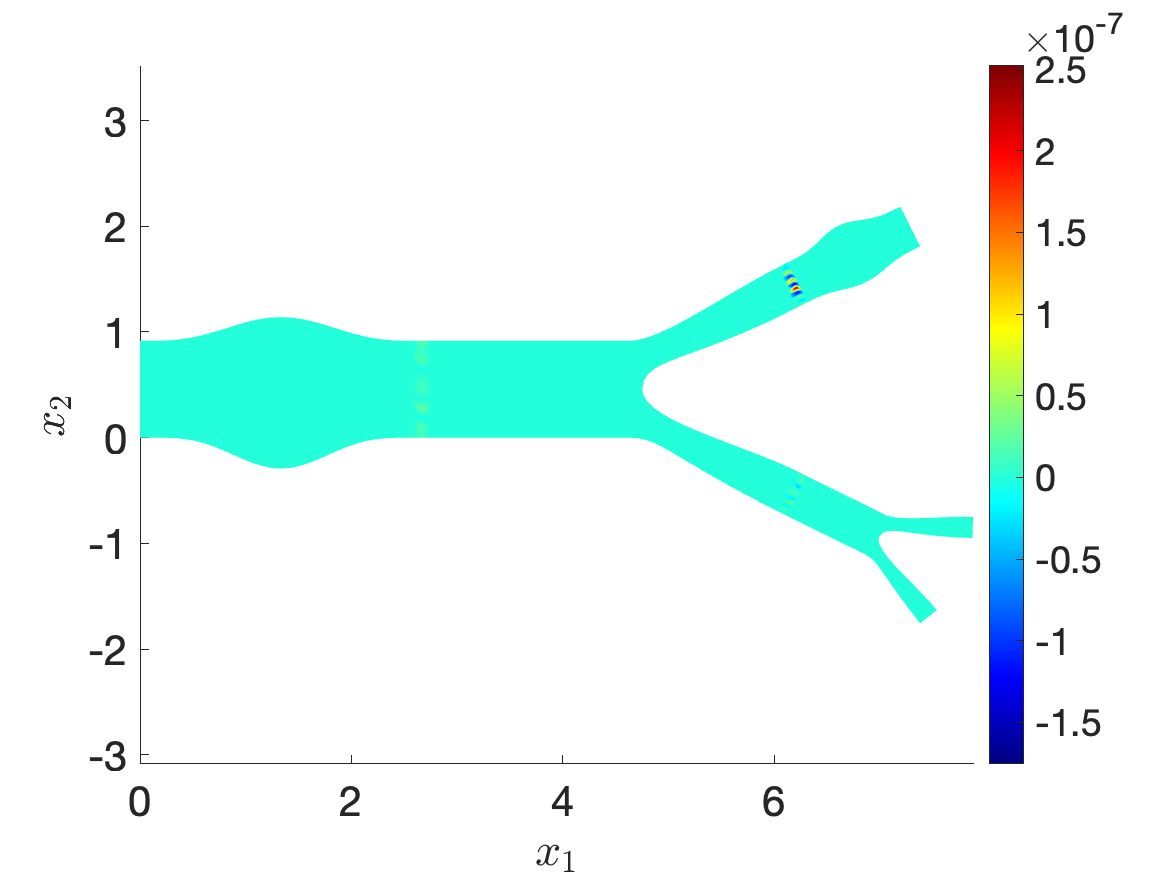}}
  
\subfloat[$\llbracket  u_x \rrbracket$]{\includegraphics[width=0.33\textwidth]{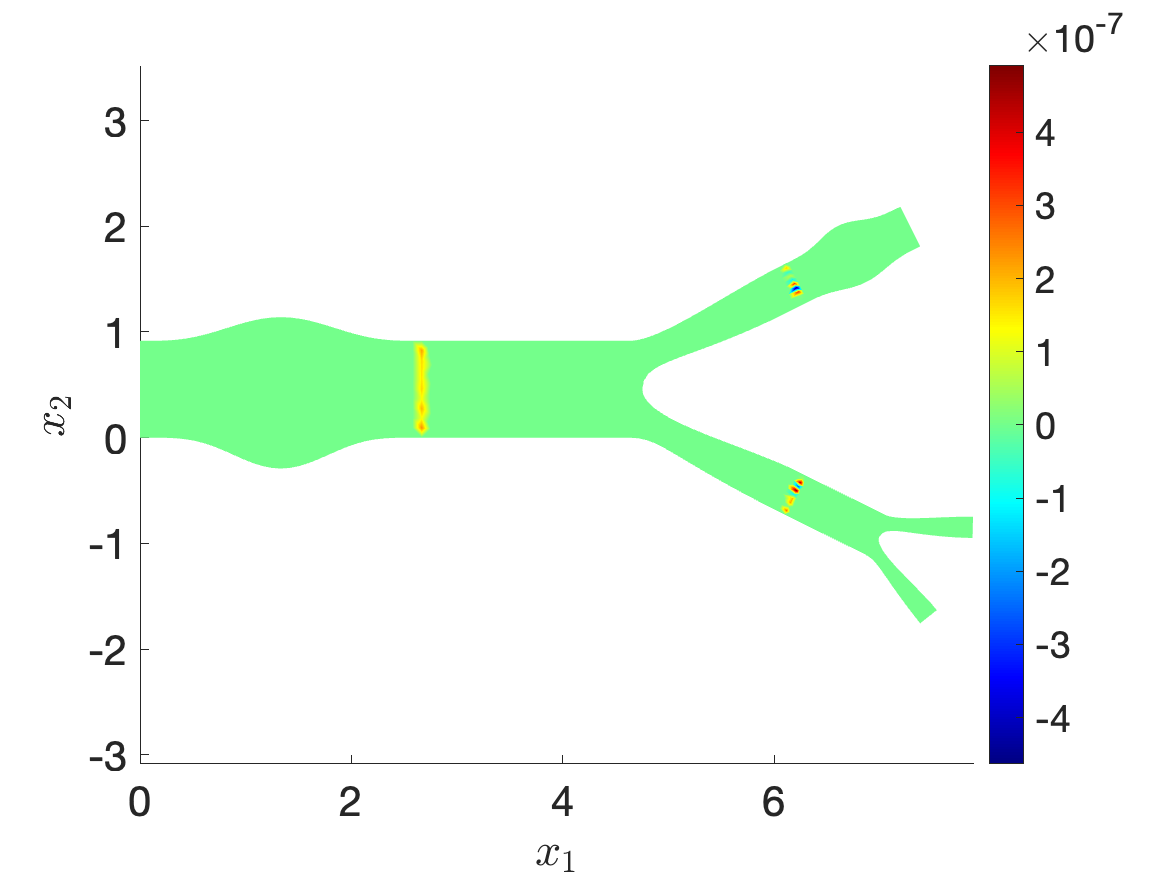}}
~~
\subfloat[$\llbracket  u_y \rrbracket$]{\includegraphics[width=0.33\textwidth]{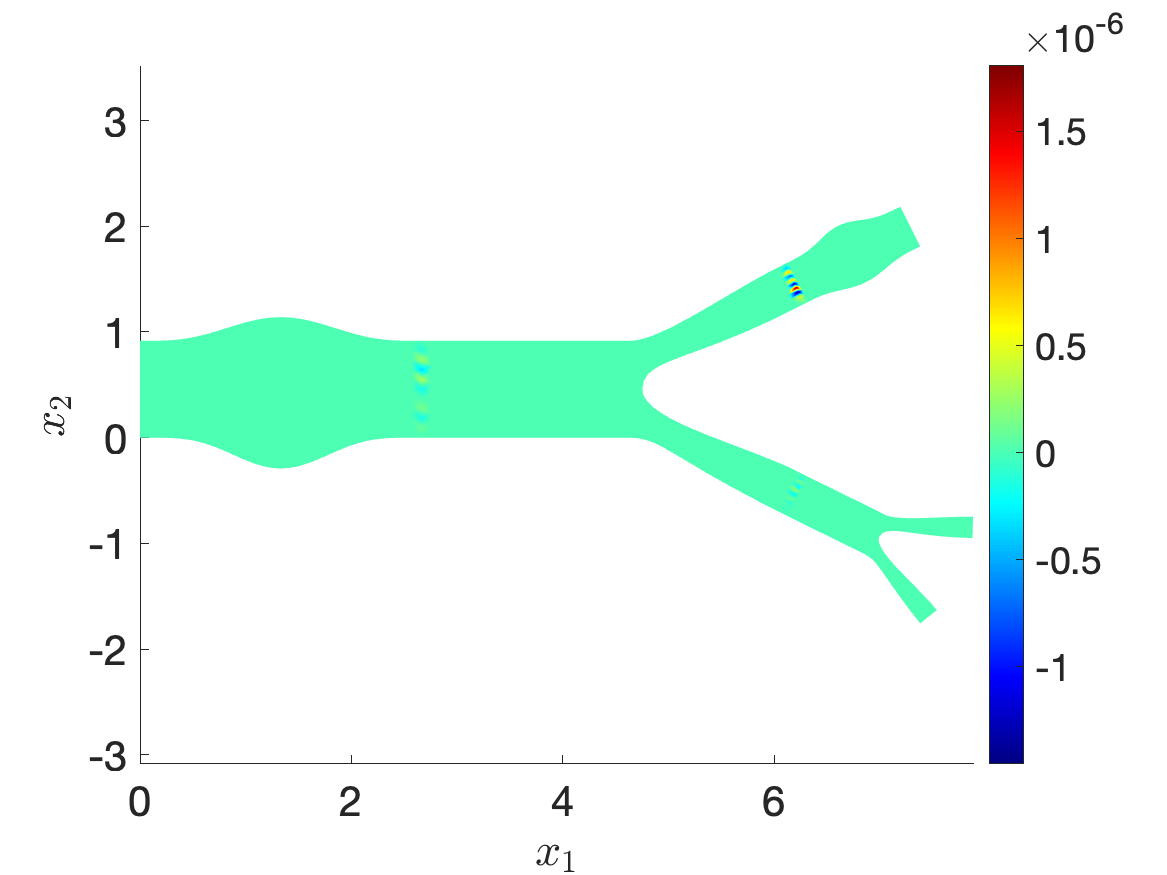}}
~~
\subfloat[$\llbracket p \rrbracket$]{\includegraphics[width=0.33\textwidth]{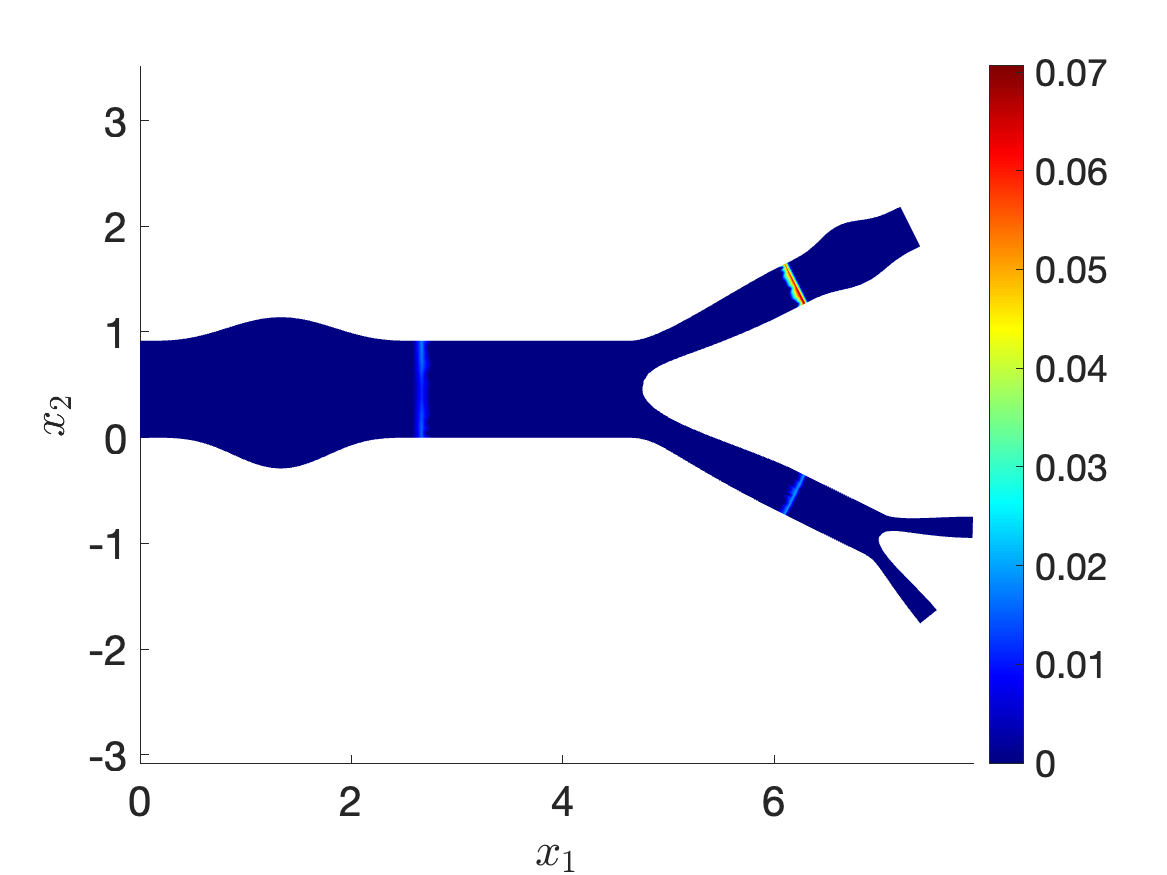}}
  
\caption{HF formulation.
(a)-(b)-(c)  interface jump of the solution to 
\eqref{eq:lsq_discr_final_a}.
(d)-(e)-(f)  interface jump of the solution to   
\eqref{eq:lsq_discr_standard}.}
\label{fig:hf_jump}
\end{figure}

Figure \ref{fig:g_H1norm} investigates the effect of the choice of the penalization norm for the control. In more detail, we compare the behavior of the horizontal control $g_x$ for the first port $\Gamma_1$
in Figure \ref{fig:instantiation_example}(b) for both $L^2$ regularization and $H^1$ regularization. We observe that the use of the $H^1$ regularization dramatically reduces the spurious oscillations in the proximity of the boundaries of the domain.
We further observe that, since
$\mathbf{n}={\rm vec}(1,\,0)$, 
 the control $g_x$ should  equal the viscous flux $- p + \nu \frac{\partial u_x}{\partial x}$; provided that $p \gg \big| \nu \frac{\partial u_x}{\partial x} \big|$, we hence find that $g_x \approx -p$.

\begin{figure}[H]
\includegraphics[width=9cm]{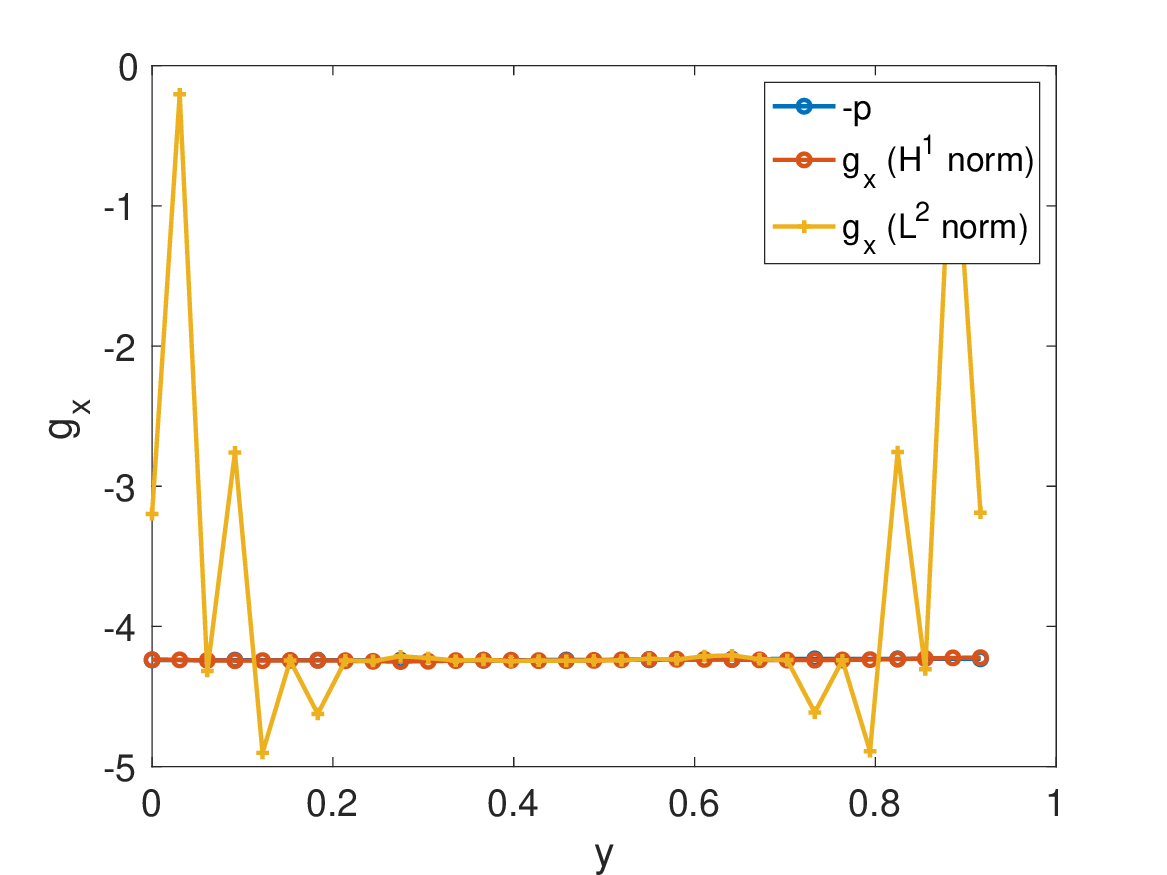}
\centering
\caption{Comparison of the $H^1$ norm and the $L^2$ norm for the regularization term.}
\label{fig:g_H1norm}
\end{figure}

\subsection{MOR procedure for networks of moderate size}
\label{sec:MOR_simple}
We now evaluate the performance of the ROM  introduced in section \ref{sec:rom} for the system 
configuration shown in Figure \ref{fig:instantiation_example}.
Since the total number of degrees of freedom is relatively modest, we can afford to solve the  multi-component generalization of  \eqref{eq:lsq_discr_final_a}  with HF local models and HF control. This enables a rigorous assessment of the results.
For the test cases presented in this section and in section \ref{sec:ROM_res_loc}, we choose the dimension of the original ROB (i.e., without ROB enrichment) for the state $n$ to be equal to the dimension of the ROB for the control $m$.

\subsubsection{Performance for a fixed geometry}
We freeze the value of the geometric parameters and we let   the Reynolds number vary in the domain 
$ \mathcal{P}= [50,150]$. 
We train the local ROMs based on $n_{\rm train}=60$ snapshots with equi-spaced parameters in $\mathcal{P}$, and we assess the performance of the resulting CB-ROM based on $n_{\rm test}=10$ randomly-selected out-of-sample parameters.
We measure performance of the ROMs in terms of the average out-of-sample relative prediction error for the four components:
\begin{equation}
\label{eq:Eavg}
E_{{\rm avg},\,i}:=\frac{1}{n_{\rm test}}\sum\limits_{\mu \in \mathcal{P}_{\rm test}}\frac{\Vert {\mathbf{w}}^{\rm hf}_i(\mu )-
\widehat{\mathbf{w}}_i(\mu ) \Vert_{\mathcal{X}_i}}{\Vert {\mathbf{w}}^{\rm hf}_i(\mu ) \Vert_{  \mathcal{X}_i  }},\quad i=1,\cdots,N_{\rm dd}=4,
\end{equation}
and the three ports:
\begin{equation}
\label{eq:Eavg_port}
E_{{\rm avg},\,j}^{\rm port}:=\frac{1}{n_{\rm test}}\sum\limits_{\mu \in \mathcal{P}_{\rm test}}\frac{\vertiii{ \mathbf{s}_j^{\rm hf}(\mu ) 
-\widehat{\mathbf{s}}_j(\mu ) 
}_{\Gamma_i}}
{\vertiii{\mathbf{s}_j^{\rm hf}}_{\Gamma_i}
}, \quad j=1,\cdots,N_{\rm f}=3.
\end{equation}

Figure \ref{fig:rom1_res} shows the prediction error $E_{{\rm avg},\,i}$ for the state $\mathbf{w}$ associated with three different local ROMs,
 Galerkin, Petrov-Galerkin, and minimum residual,
 for the four components of the network;
 Figure   \ref{fig:rom1_res_gh} shows the prediction error $E_{{\rm avg},\,j}^{\rm port}$ for the control on the three ports for the same choice of the local ROM.
 In this test, we do not perform  the enrichment of the state spaces described in section \ref{sec:enrichment_basis}.
 The Galerkin method exhibits stability issues, while both the minimal residual and the Petrov-Galerkin methods perform equally well in terms of accuracy with a relative error of the order of $10^{-5}$ for  $n=m=20$. 
The prediction  of the control variables is far less accurate: for $n=m=20$, the  relative error is $\mathcal{O}(10^{-1})$ for port $3$ and $\mathcal{O}(10^{-2})$ for the other two ports. 
Nevertheless, we envision that the  results can still be considered satisfactory, as illustrated by the  profiles of the control $\mathbf{g}$ at port $3$ shown in Figure \ref{fig:rom1_res_g_profile}, where $\xi$ is the local coordinate along the port $3$.
 
 %\label{fig:rom1_res}
 \begin{figure}[H]
  \centering
  
  \subfloat[domain 1]{\includegraphics[width=0.4\textwidth]{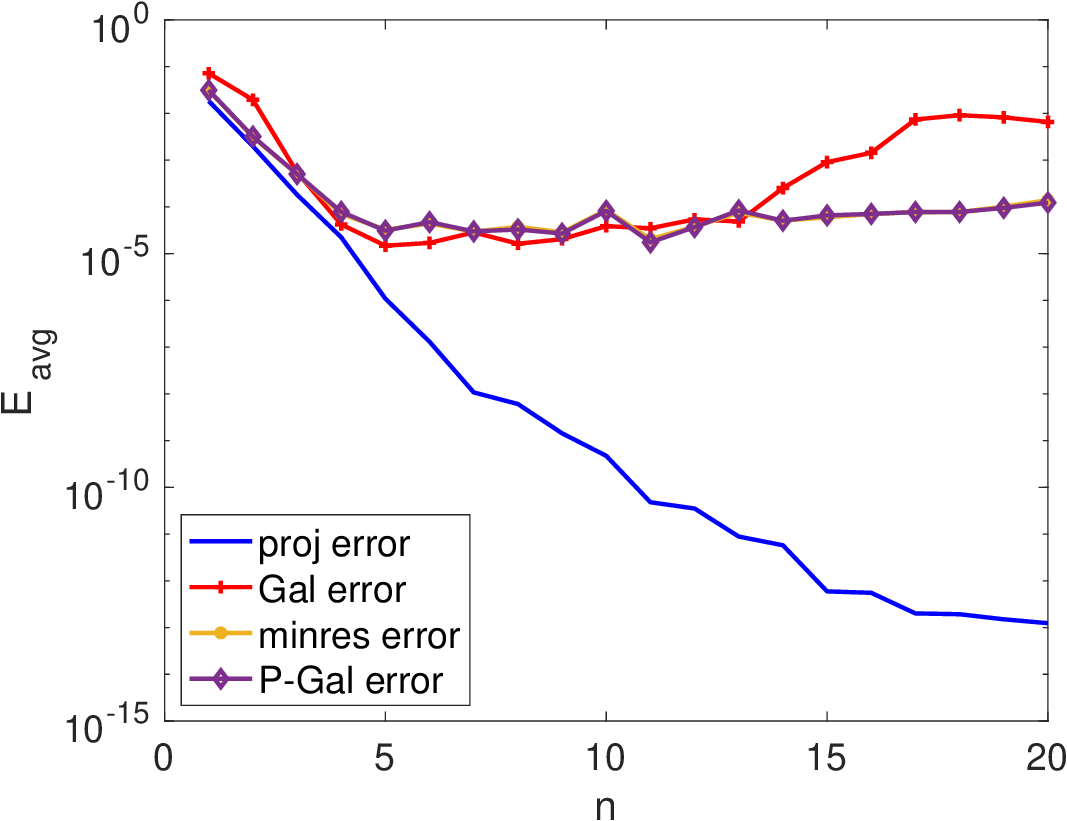}}
  \hfill
  \subfloat[domain 2]{\includegraphics[width=0.4\textwidth]{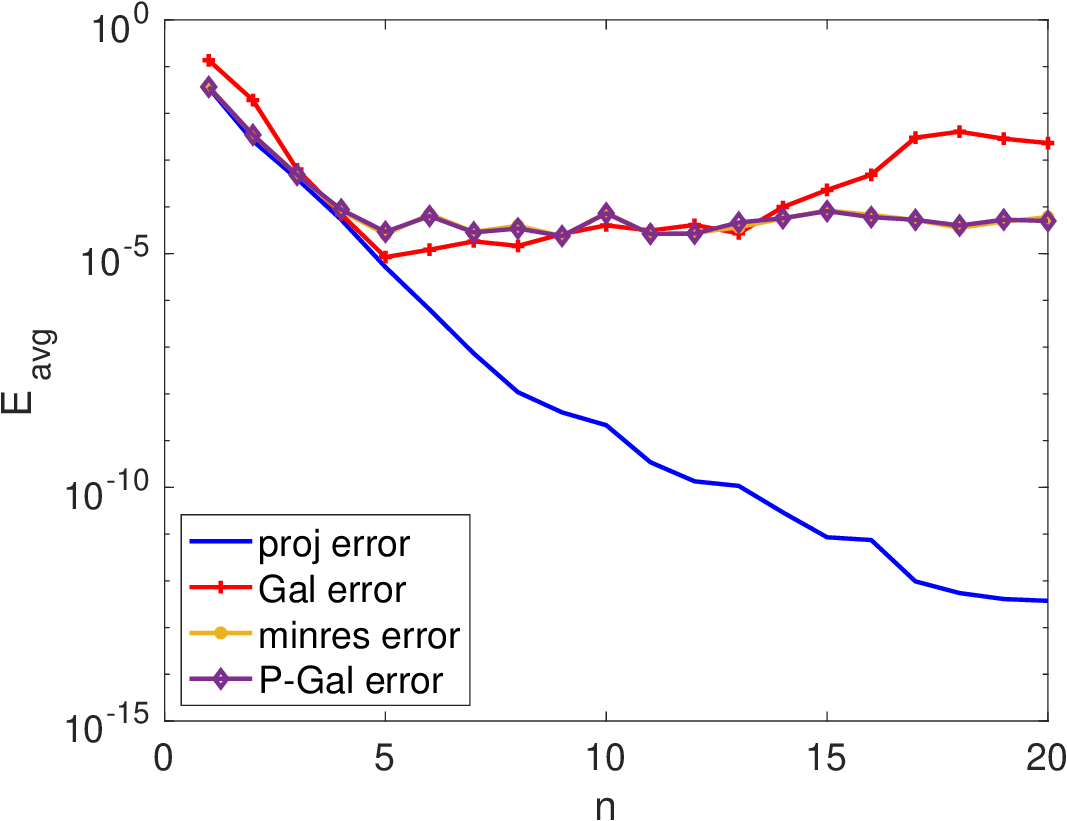}}

  \subfloat[domain 3]{\includegraphics[width=0.4\textwidth]{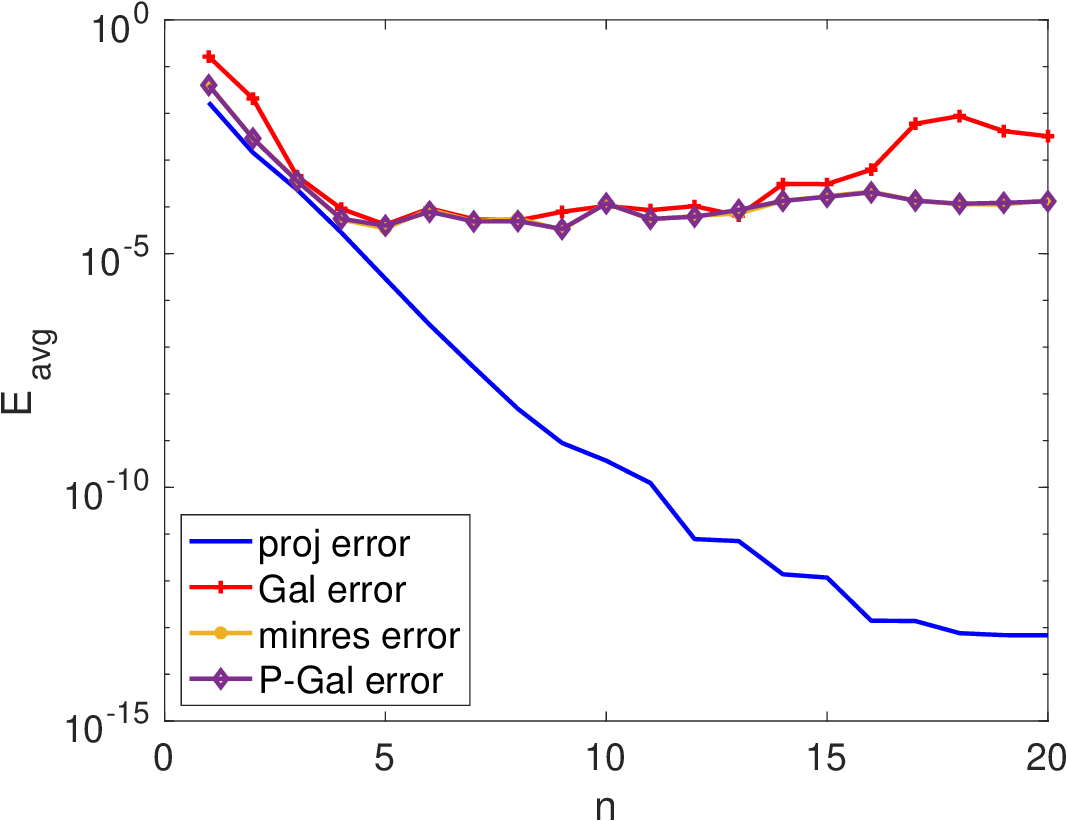}}
  \hfill
  \subfloat[domain 4]{\includegraphics[width=0.4\textwidth]{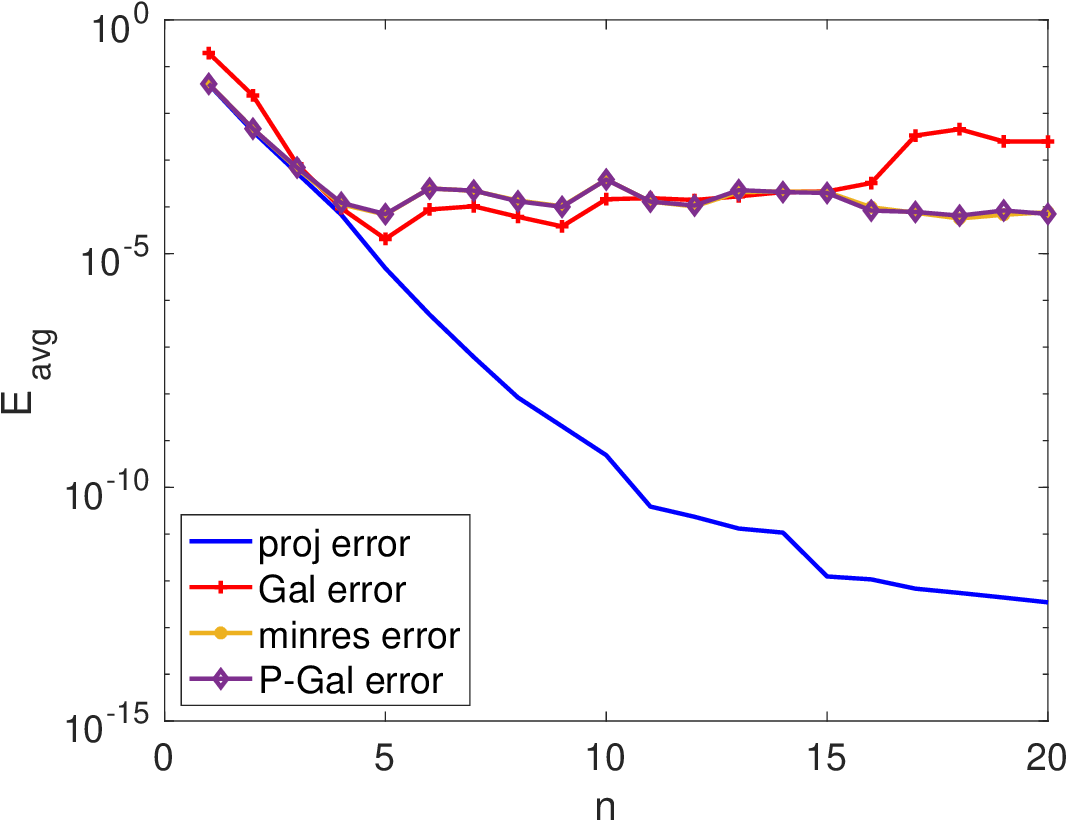}}
  
\caption{performance for a fixed geometry. Behavior of the error \eqref{eq:Eavg} for the subdomains (no enrichment).}
\label{fig:rom1_res}
\end{figure}

% \label{fig:rom1_res_gh}
\begin{figure}[H]
  \centering
  
  \subfloat[port 1]{\includegraphics[width=0.33\textwidth]{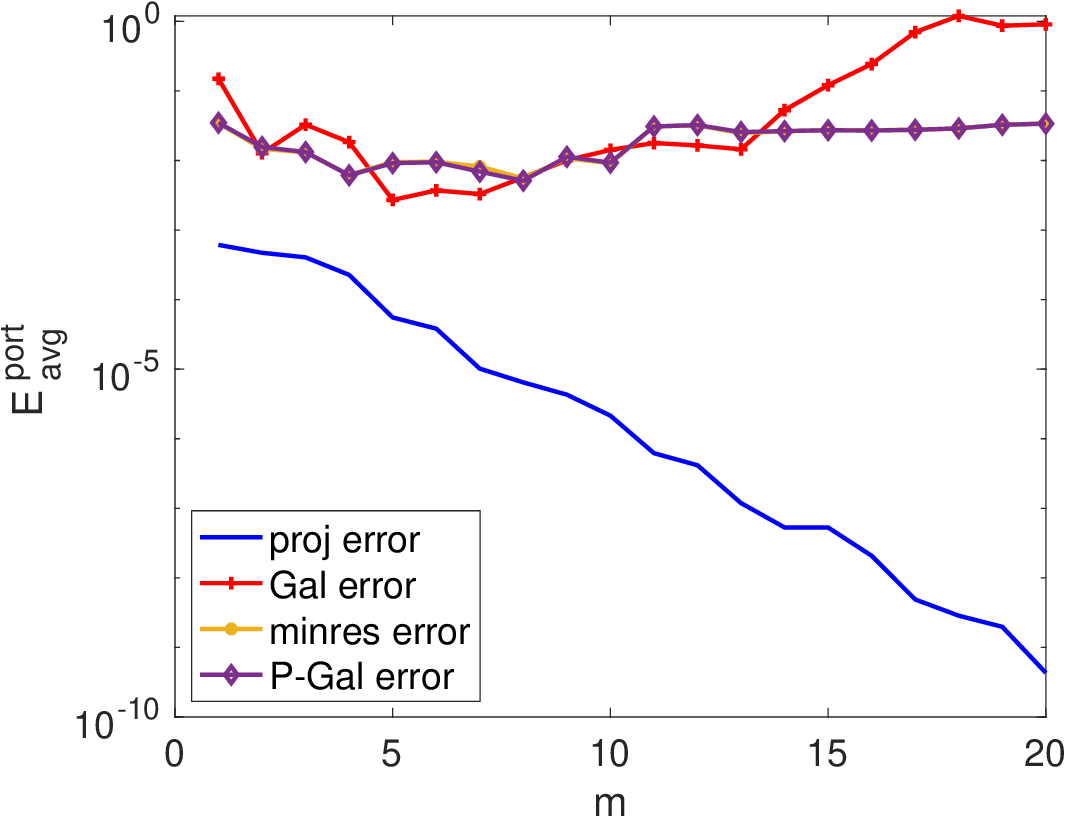}}
  ~~
  \subfloat[port 2]{\includegraphics[width=0.33\textwidth]{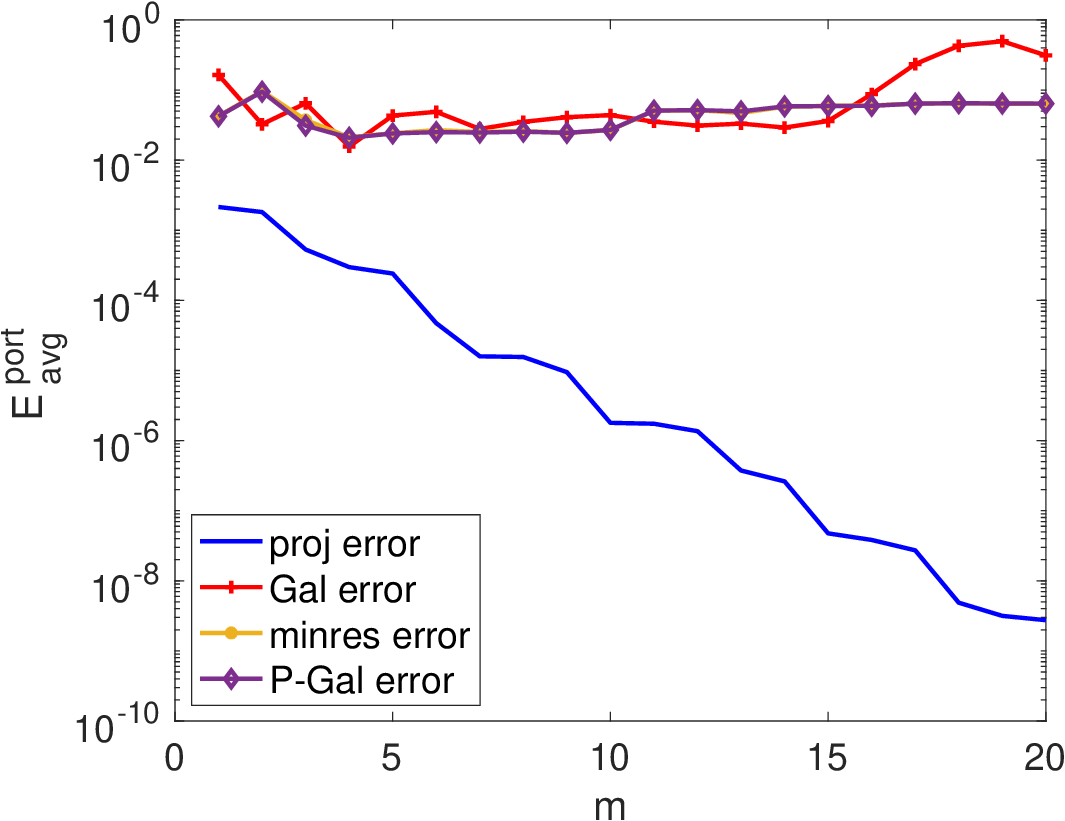}}
~~
  \subfloat[port 3]{\includegraphics[width=0.33\textwidth]{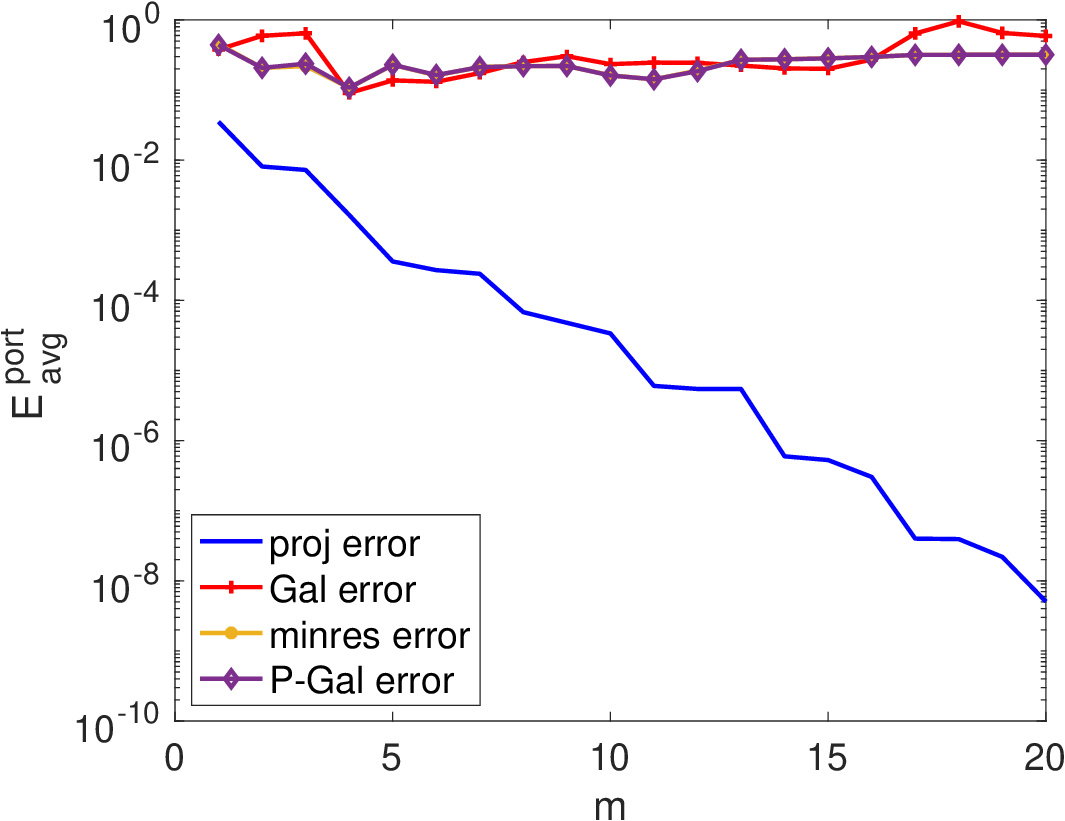}}
  
 \caption{performance for a fixed geometry.
  Behavior of the error \eqref{eq:Eavg_port} for  the  three ports (no enrichment).}
  \label{fig:rom1_res_gh}
\end{figure}

%  \label{fig:rom1_res_g_profile}
\begin{figure}[H]
  \centering
  
  \subfloat[$g_x$]{\includegraphics[width=0.4\textwidth]{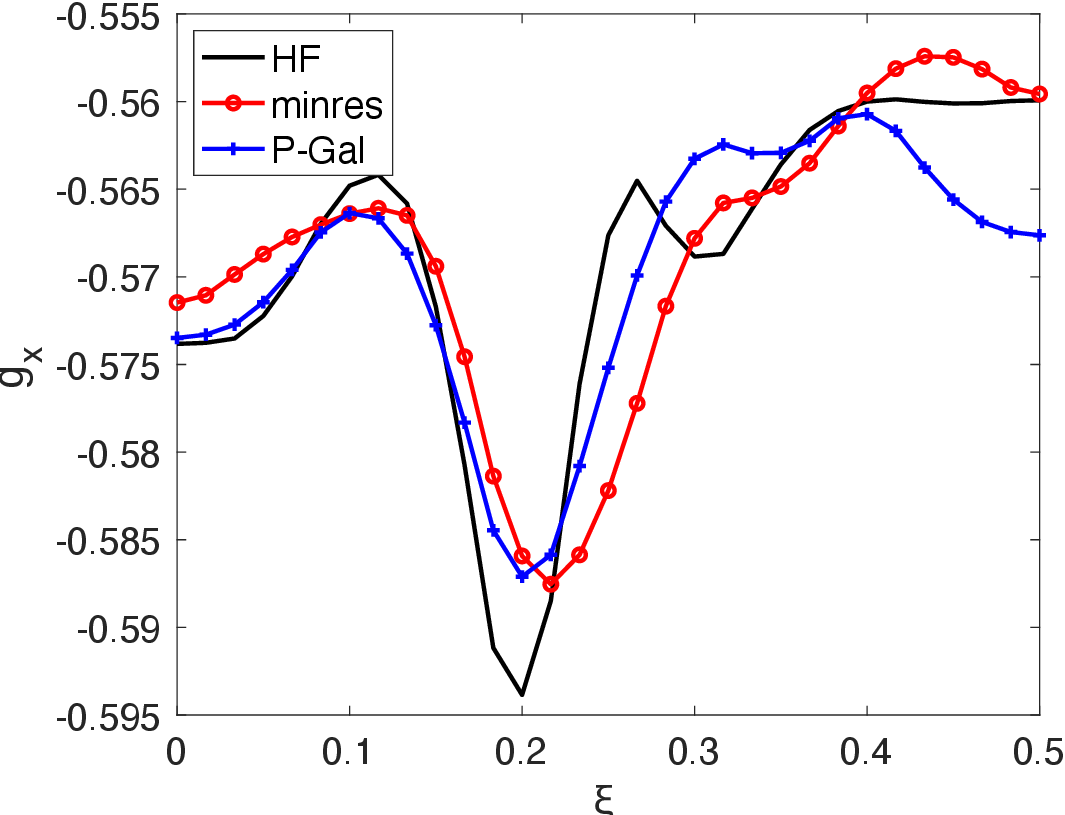}}
  \hfill
  \subfloat[$g_y$]{\includegraphics[width=0.4\textwidth]{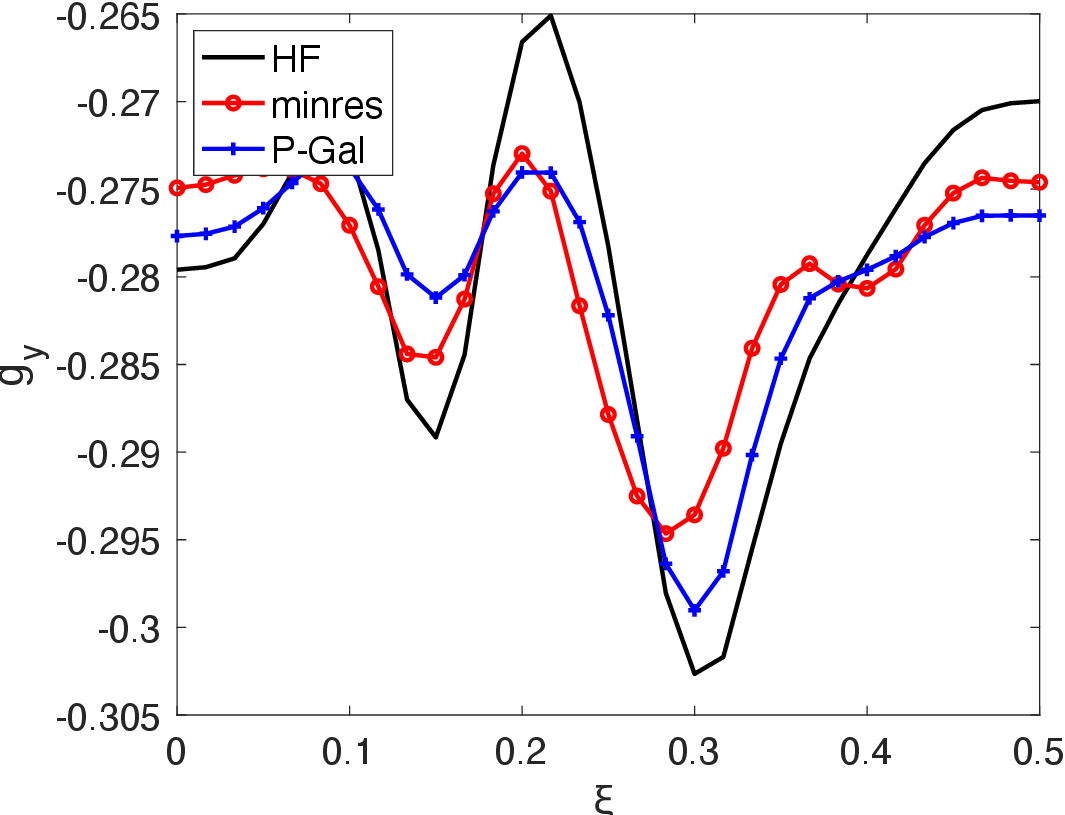}}
    
  \caption{performance for a fixed geometry. 
  Profile of the two components of the control $\mathbf{g}$ for one representative parameter value along the port $3$ (no enrichment).}
  \label{fig:rom1_res_g_profile}
\end{figure}

In Figure \ref{fig:rom1_compr_Fgh}, we illustrate the performance of the ROM when we employ the enrichment strategy discussed in section \ref{sec:enrichment_basis}.
To facilitate comparison, we include dashed lines representing the results obtained without employing ROB enrichment, which corresponds to the data presented in Figure \ref{fig:rom1_res_gh}.
Here, the number of additional modes $n'$ (cf. section \ref{sec:enrichment_basis}) is chosen to be equal to the dimension of the ROB of the control,  $m$. 
The ROB enrichment strategy significantly reduces the prediction error for the control; the state prediction 
achieved with ROB enrichment
is comparable to the case without ROB enrichment and is not provided below.
We further remark that the enrichment 
does not contribute to increase the number of SQP iterations: to provide a concrete reference, for $m=10$,  SQP converges in six iterations  for both cases.

% \label{fig:rom1_compr_Fgh}
\begin{figure}[H]
  \centering
  
  \subfloat[port 1]{\includegraphics[width=0.33\textwidth]{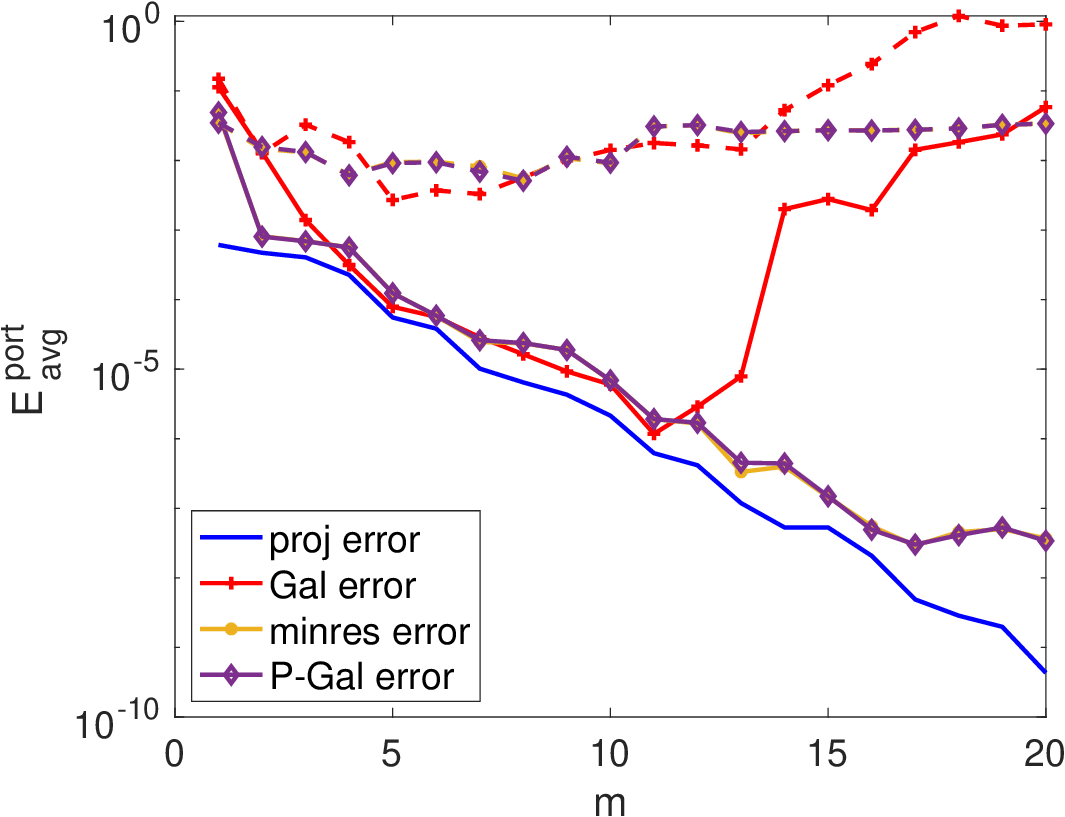}}
~~
  \subfloat[port 2]{\includegraphics[width=0.33\textwidth]{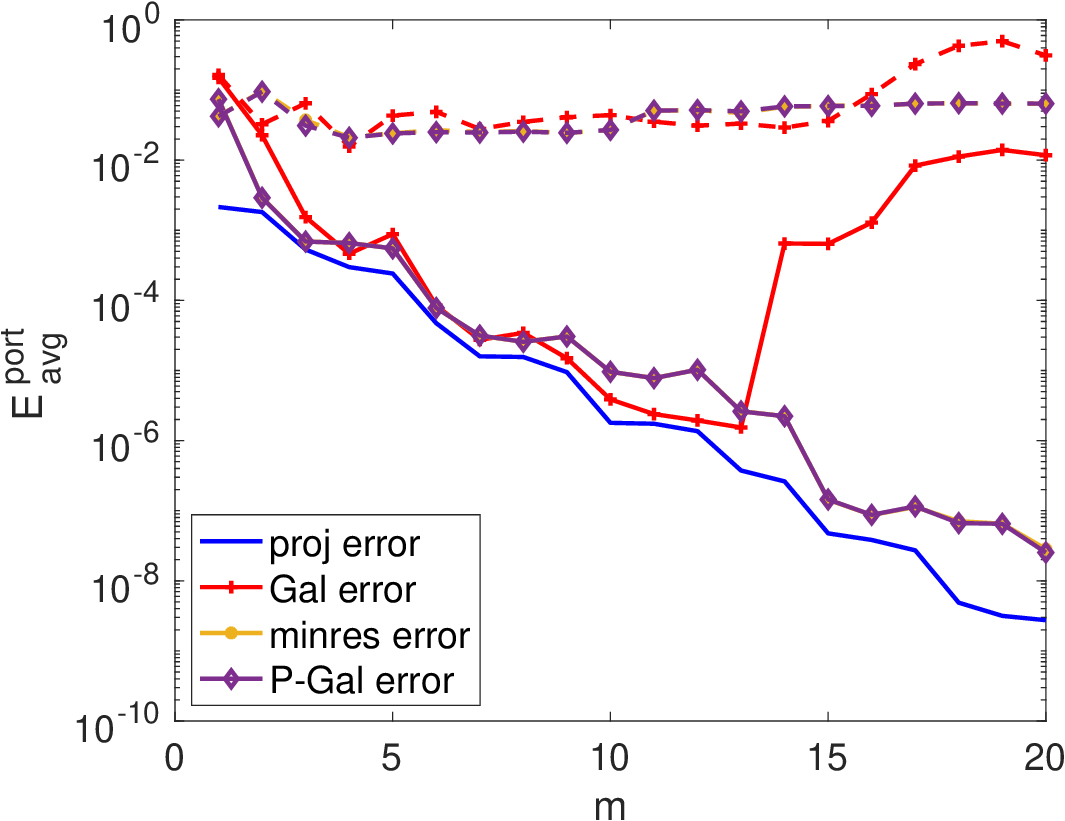}}
~~
  \subfloat[port 3]{\includegraphics[width=0.33\textwidth]{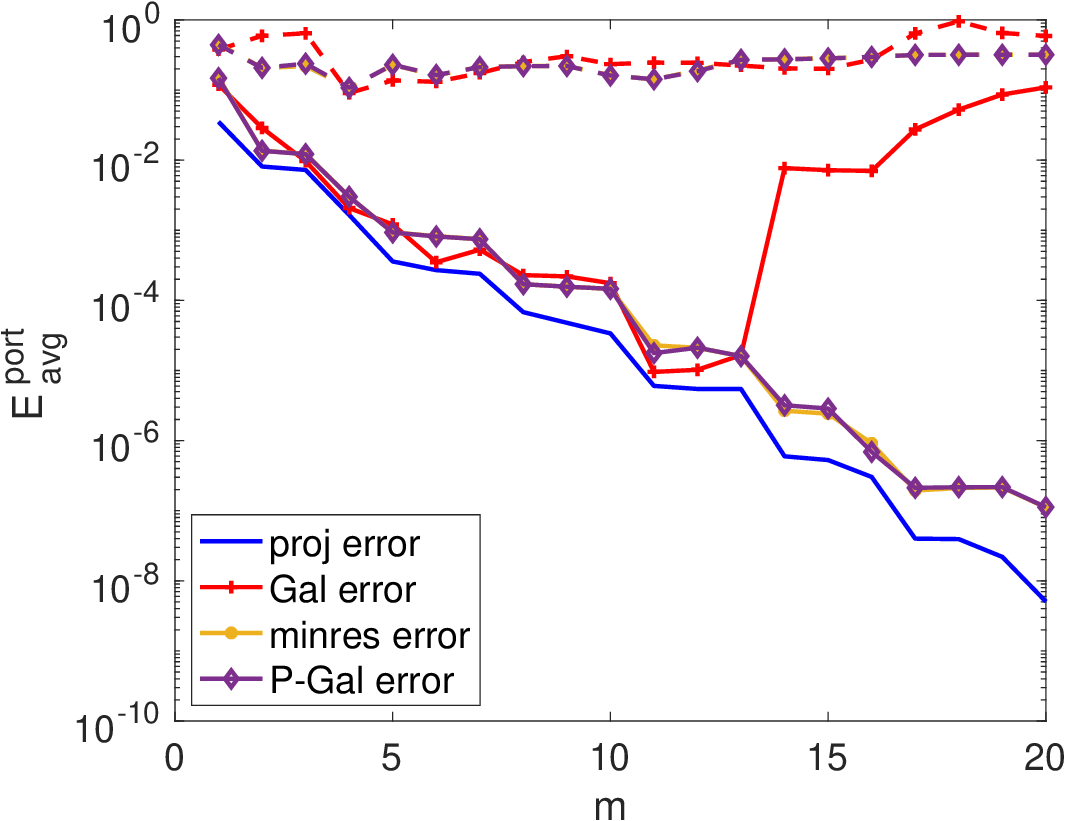}}
  
\caption{performance for a fixed geometry. Behavior of the error \eqref{eq:Eavg_port} for the  three ports (with  enrichment).}
 \label{fig:rom1_compr_Fgh}
\end{figure}

\subsubsection{Performance for a parametric  geometry}
\label{sec:ROM_results2}
We incorporate the geometric parameters described in section \ref{sec:model_pb}, along with the Reynolds number. 
For each junction component in the network, 
we set $\alpha \in  [\frac{\pi}{8},\,\frac{\pi}{6}]$; for each channel component, we set  $h_c\in [0.1,\,0.3]$; finally, we consider
$\text{Re}\in [50,150]$ with $\text{Re}_{\rm ref}=100$.
We train the ROMs based on $n_{\rm train}=120$ snapshots and assess performance based on $n_{\rm test}=10$ randomly-selected out-of-sample parameters.
As for  the previous test, we analyze the prediction error for both $\mathbf{w}$ and $\mathbf{s}$ associated with the different ROMs. Figure \ref{fig:rom2_res} illustrates the prediction error $E_{{\rm avg},\,i}$ for the four components, while Figure \ref{fig:rom2_res_gh} shows the prediction error $E_{{\rm avg},\,k}^{\rm port}$ for the three ports. Interestingly,   the Galerkin method is   as effective as   the minimal residual and the Petrov-Galerkin methods. All three ROMs yield a  state prediction relative error of approximately $\mathcal{O}(10^{-4})$ for  $n=20$; on the other hand,   the  control prediction error
is roughly $\mathcal{O}(10^{-1})$ for the third port, and $\mathcal{O}(10^{-2})$ for the other two ports, for  $n=20$.
In Figure \ref{fig:rom2_compr_Fgh}, we perform a comparison of ROM errors associated to the three ports, with and without the ROB enrichment strategy outlined in section \ref{sec:enrichment_basis}. The dashed lines represent the results obtained in the absence of ROB enrichment, which correspond  to the data shown in Figure \ref{fig:rom2_res_gh}.  
As for the previous test, the ROB enrichment strategy significantly improves the   accuracy of the control prediction. 
Here, the number of additional modes $n'$ (cf. section \ref{sec:enrichment_basis}) is chosen to be twice as large as the dimension of the ROB for the ports $m$.

\begin{figure}[H]
  \centering
  
  \subfloat[domain 1]{\includegraphics[width=0.4\textwidth]{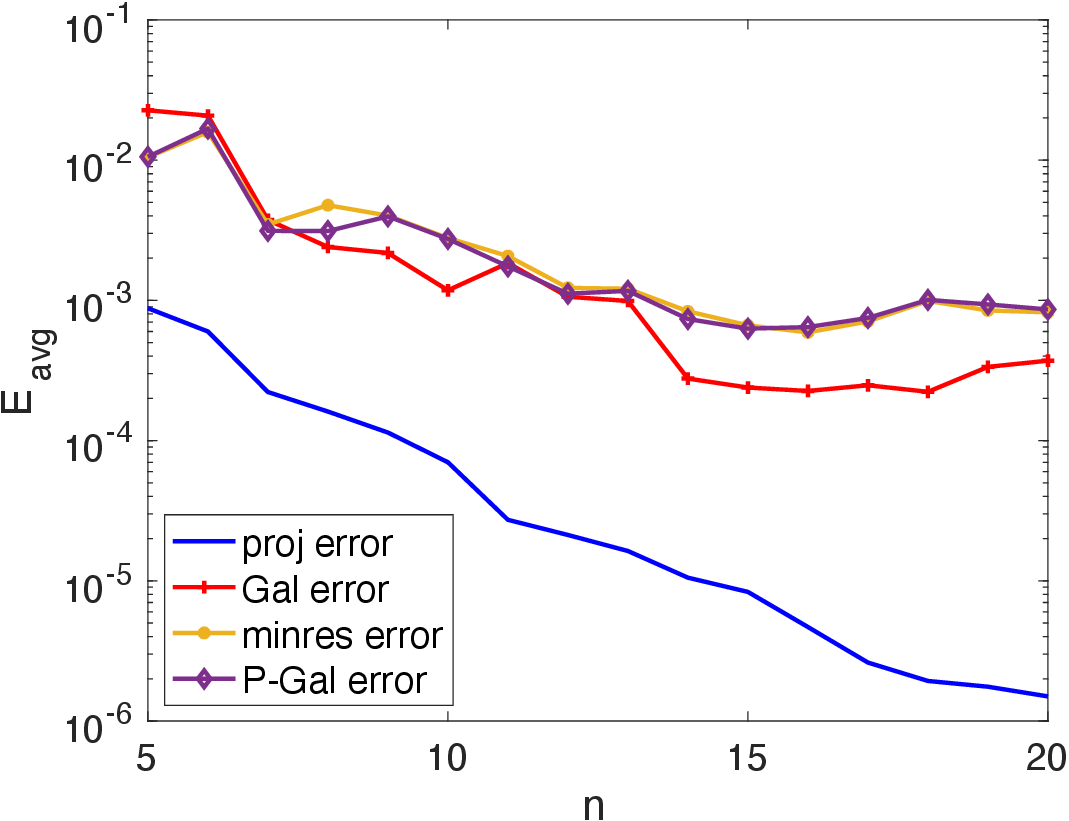}}
  \hfill
  \subfloat[domain 2]{\includegraphics[width=0.4\textwidth]{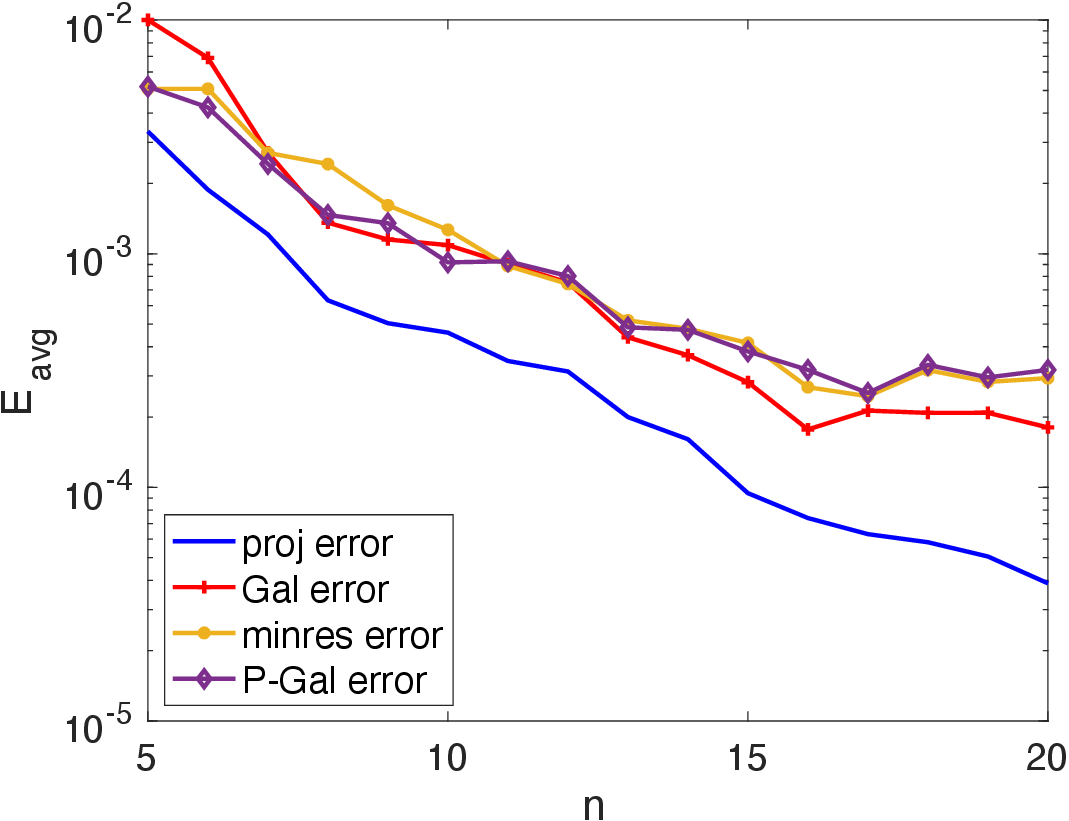}}

  \subfloat[domain 3]{\includegraphics[width=0.4\textwidth]{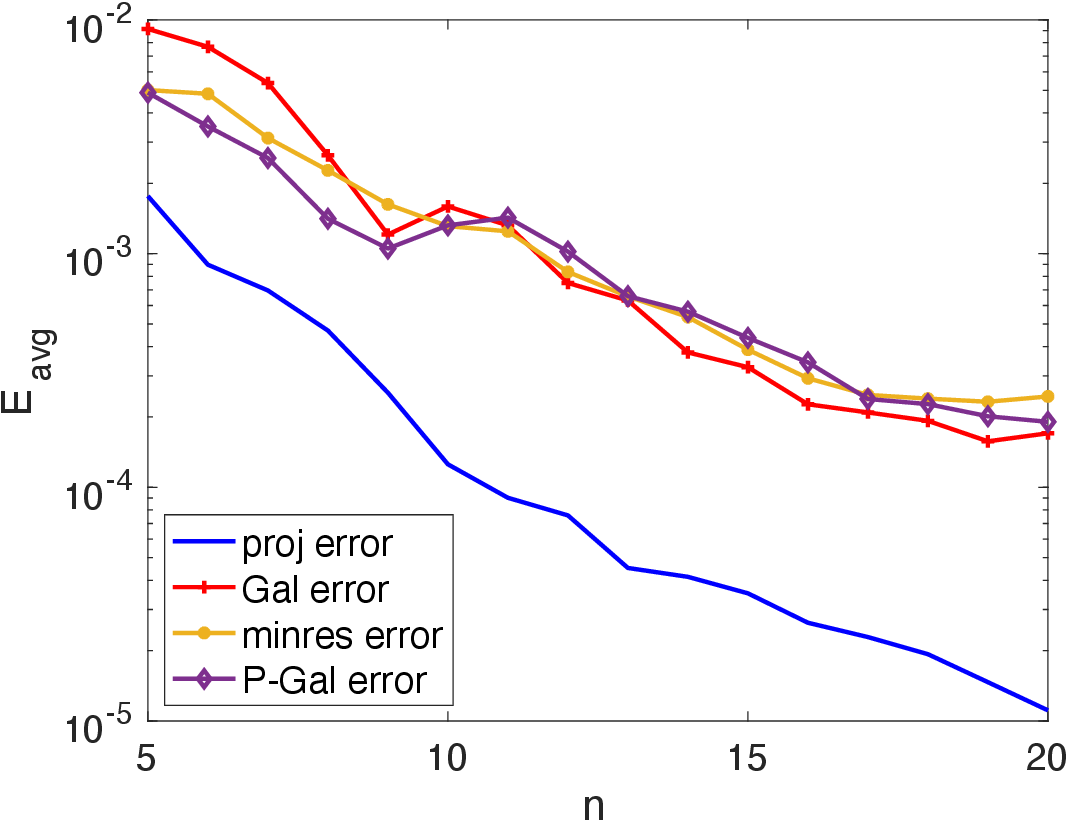}}
  \hfill
  \subfloat[domain 4]{\includegraphics[width=0.4\textwidth]{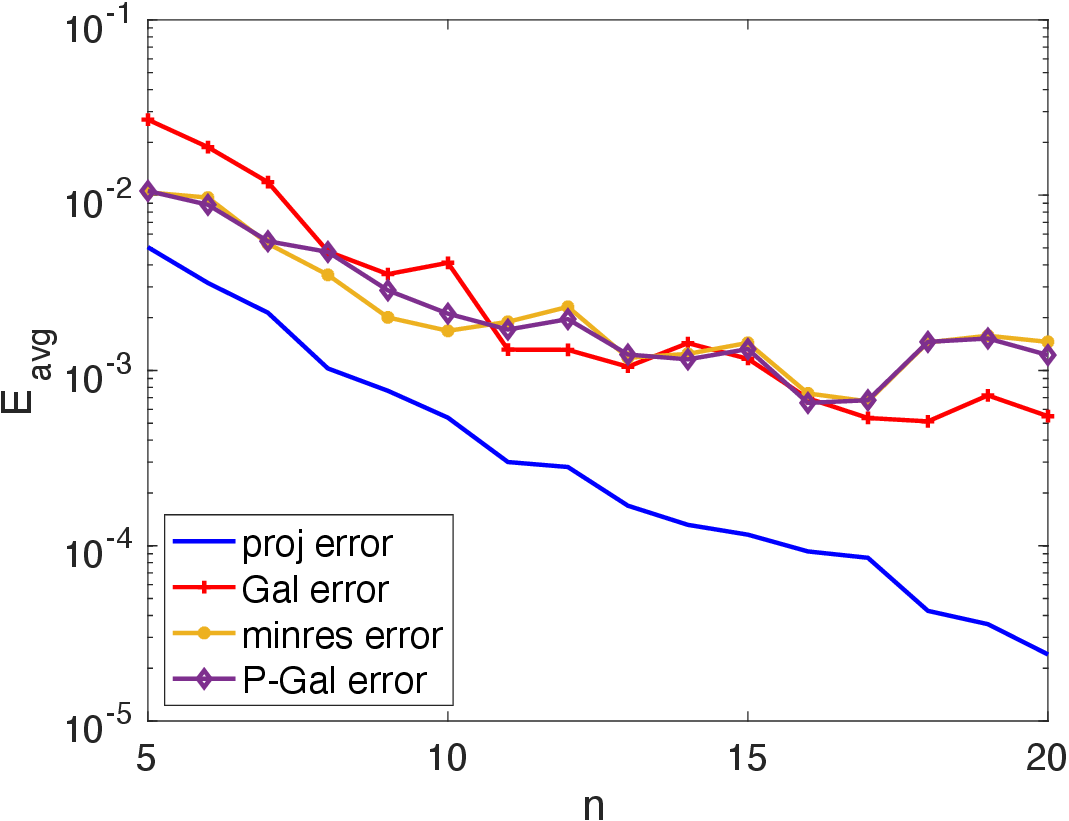}}
  
\caption{performance for a parametric geometry. 
Behavior of the error \eqref{eq:Eavg} for the subdomains (no enrichment).}
  \label{fig:rom2_res}
\end{figure}

\begin{figure}[H]
  \centering
  
  \subfloat[port 1]{\includegraphics[width=0.33\textwidth]{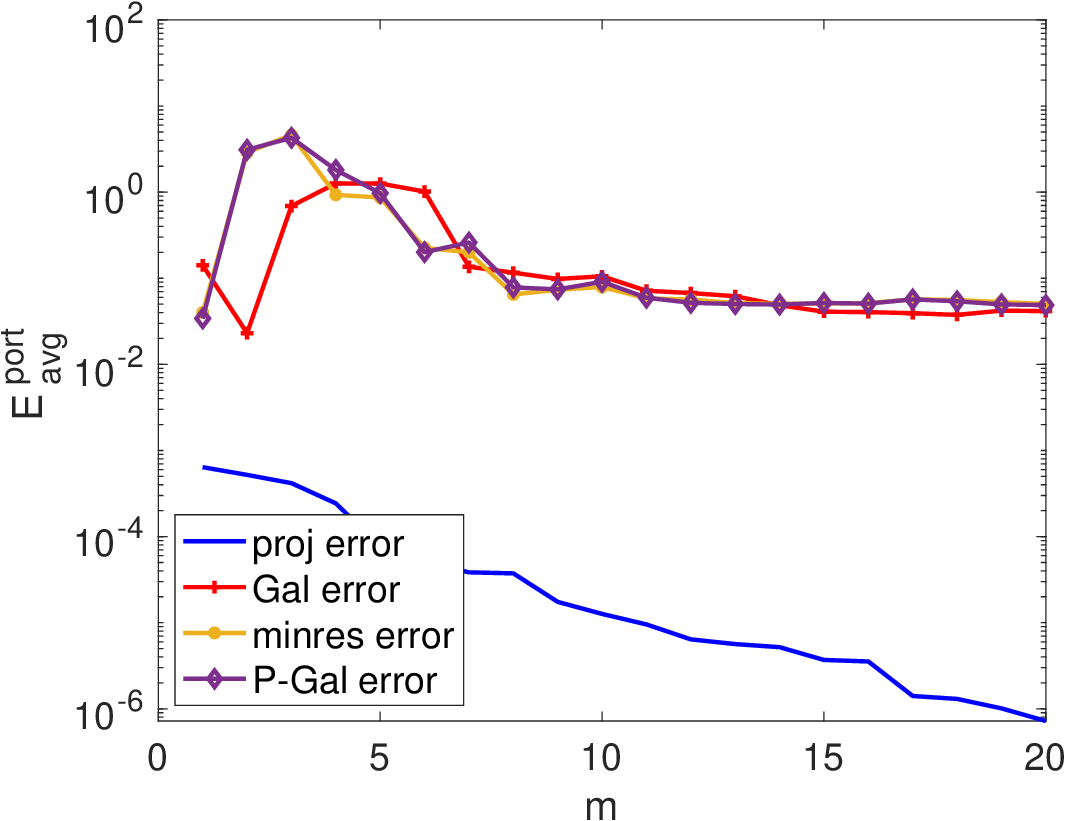}}
~~
  \subfloat[port 2]{\includegraphics[width=0.33\textwidth]{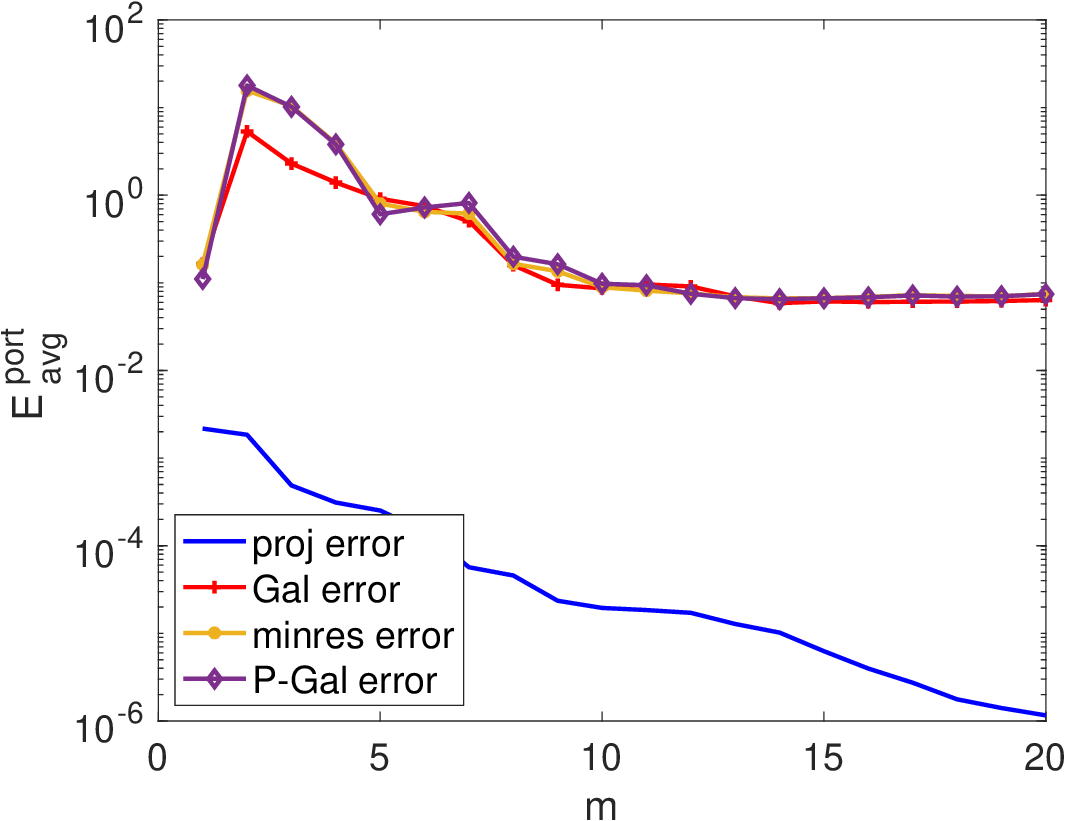}}
~~
  \subfloat[port 3]{\includegraphics[width=0.33\textwidth]{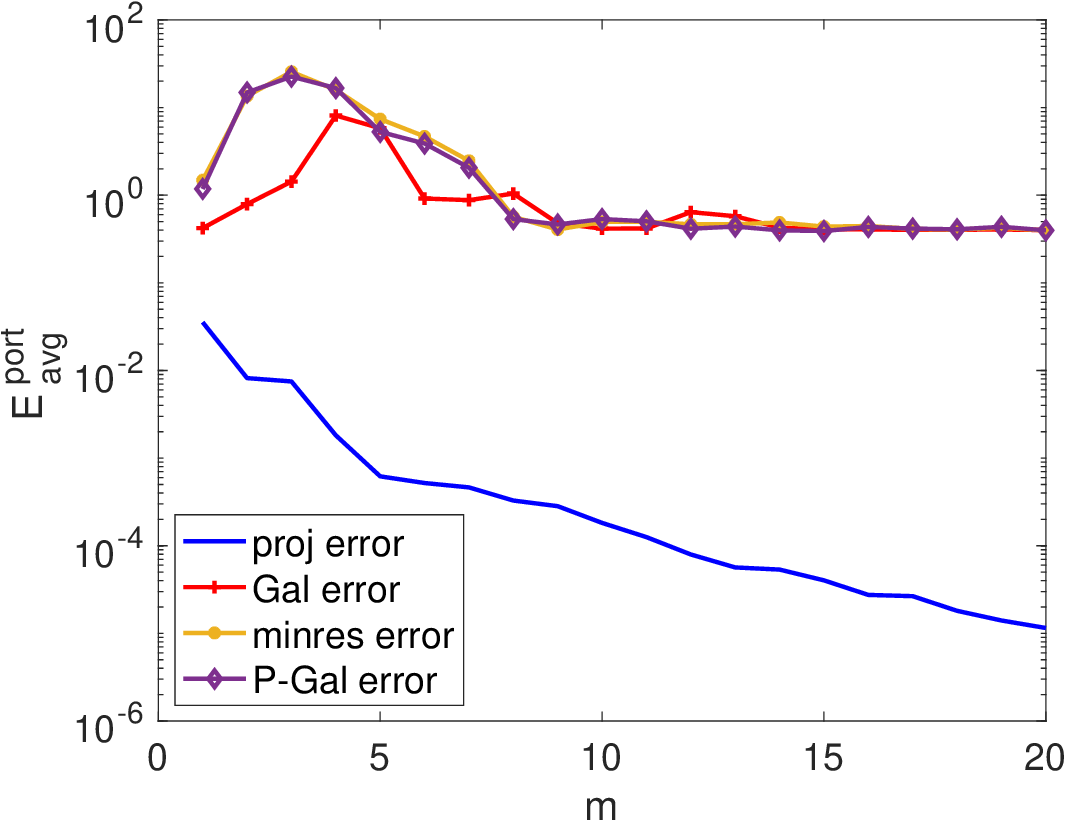}}

  \caption{performance for a parametric geometry.  
Behavior of the error \eqref{eq:Eavg_port} for the three ports (no enrichment).}
  \label{fig:rom2_res_gh}
\end{figure}

\begin{figure}[H]
  \centering
  
  \subfloat[port 1]{\includegraphics[width=0.33\textwidth]{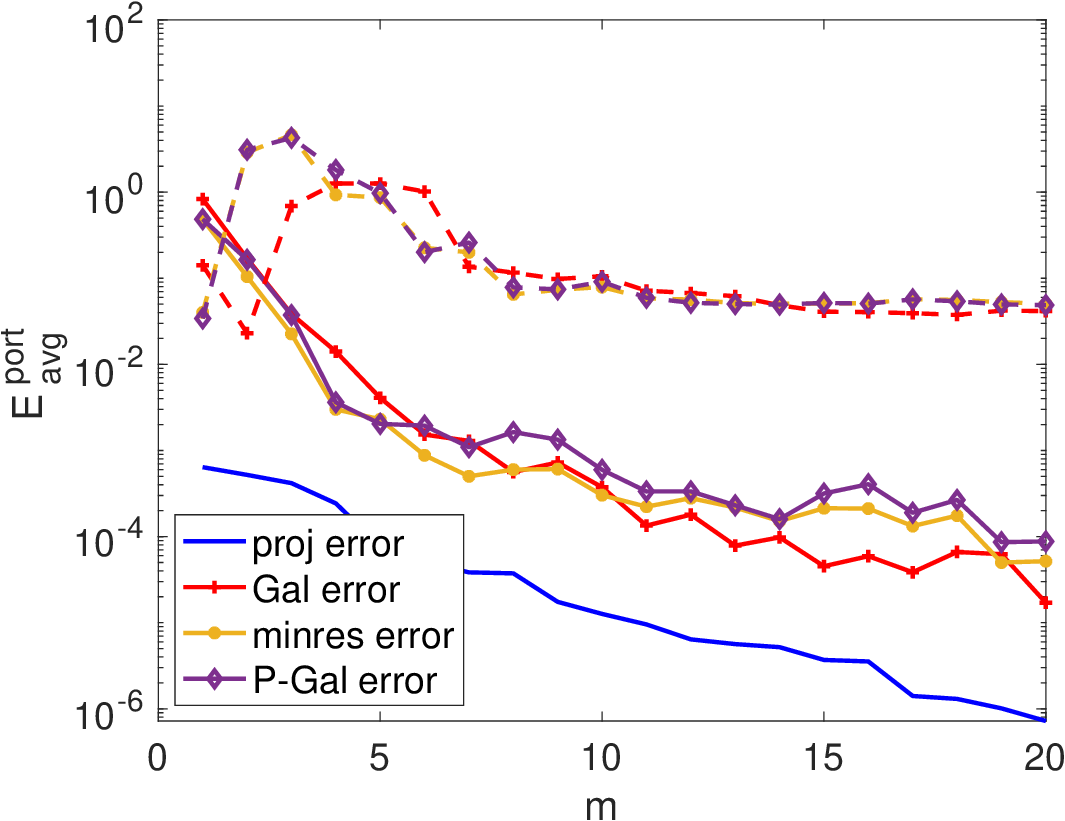}}
~~
  \subfloat[port 2]{\includegraphics[width=0.33\textwidth]{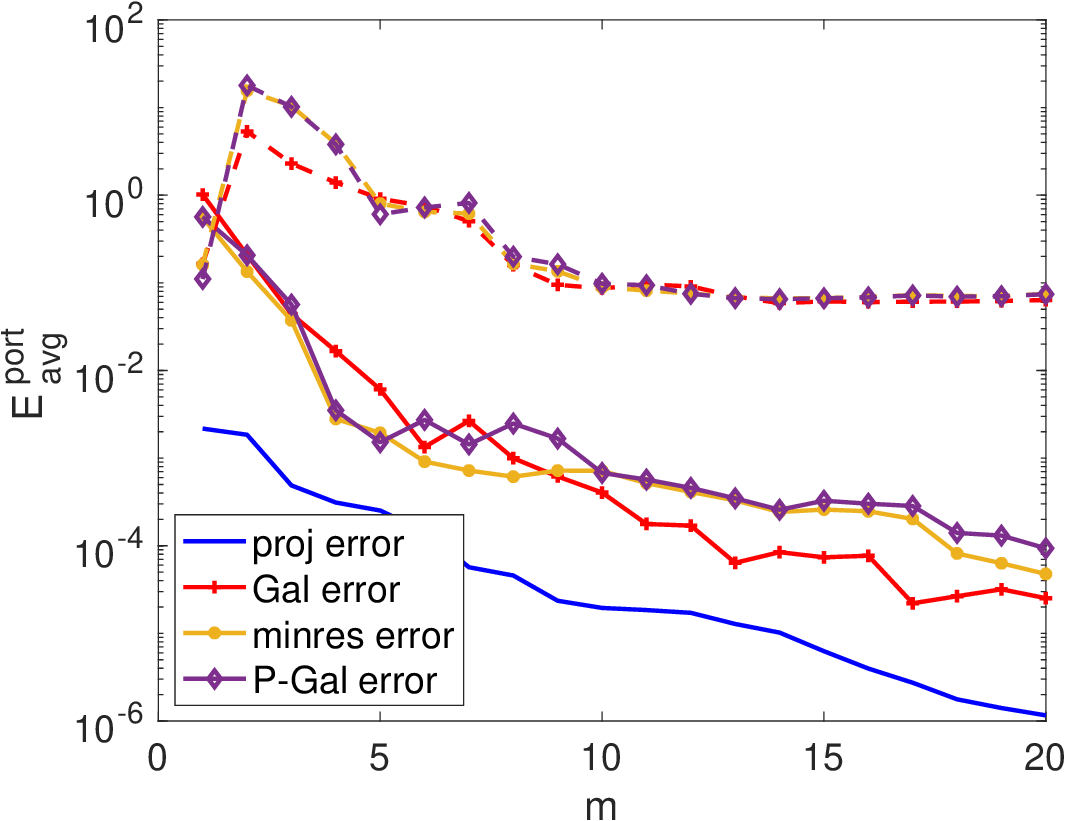}}
~~
  \subfloat[port 3]{\includegraphics[width=0.33\textwidth]{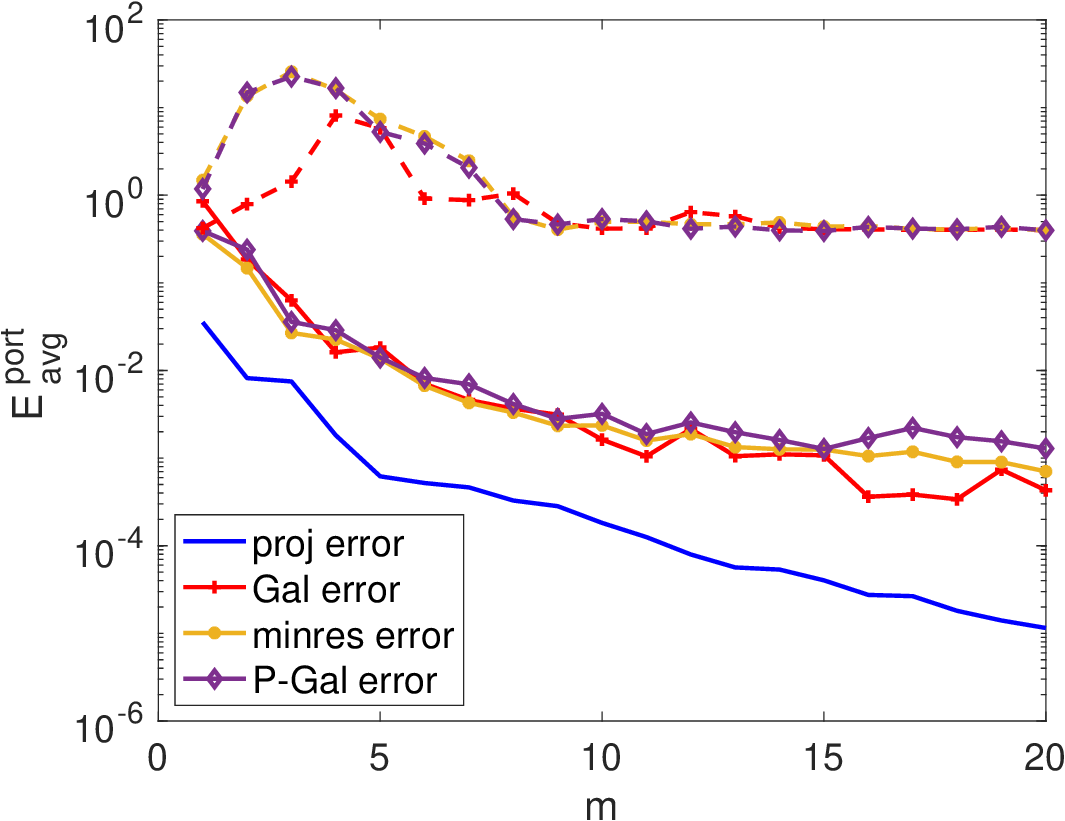}}
  
\caption{performance for a parametric geometry.  
Behavior of the error \eqref{eq:Eavg_port} for the three ports (with enrichment).}
  \label{fig:rom2_compr_Fgh}
\end{figure}

\subsection{Localized training and adaptive enrichment}
\label{sec:ROM_res_loc}
In the previous   test cases, a distinct reduced space is employed for each instantiated component: the same configuration is used for both training and  assessment.
This approach is computationally demanding when dealing with systems that comprise a large number of components; 
it is also unfeasible in the presence of topology changes. 
To address this issue, we apply  the localized training and adaptive enrichment algorithms developed in section \ref{sec:loc_train}.
%Here, we follow the localized training algorithms outlined in
%sections \ref{sec:pairwise_control}-\ref{sec:localized_state} to generate a shared ROB for each of the two archetype component, as well as a common ROB for the port.

\subsubsection{Application to networks with four components}
\label{sec:ROM_results_loc_small}
We apply the localized training strategy of sections 
\ref{sec:pairwise_control} and 
\ref{sec:localized_state}, for the same  test set of section \ref{sec:ROM_results2}. 
In order to build the reduced space for the control,
we consider
$60$ randomly selected boundary conditions for each connection described in section \ref{sec:pairwise_control};
on the other hand, 
we generate the reduced space for the state using 
$20$ randomly-sampled networks with four components and the reduced space for the control.

Figure \ref{fig:rom3_res} presents the prediction error $E_{{\rm avg},\,i}$ for the four components, while Figure \ref{fig:rom3_res_gh} shows the prediction error $E_{{\rm avg},\,k}^{\rm port}$ for the three ports; we do not rely on the enrichment of the state space (cf. section \ref{sec:enrichment_basis}).
The results are comparable to those obtained in section \ref{sec:ROM_results2} with slight deterioration in accuracy.
Figure \ref{fig:rom3_compr_Fgh} displays
the ROM errors for the three ports using ROB enrichment ($n'=2m$), as represented by the solid line. The results exhibit significant improvement when compared to those obtained without the use of the state space enrichment, as illustrated  by the dashed lines, which correspond to the data shown in Figure \ref{fig:rom3_res_gh}.

\begin{figure}[H]
\centering
  
  \subfloat[domain 1]{\includegraphics[width=0.4\textwidth]{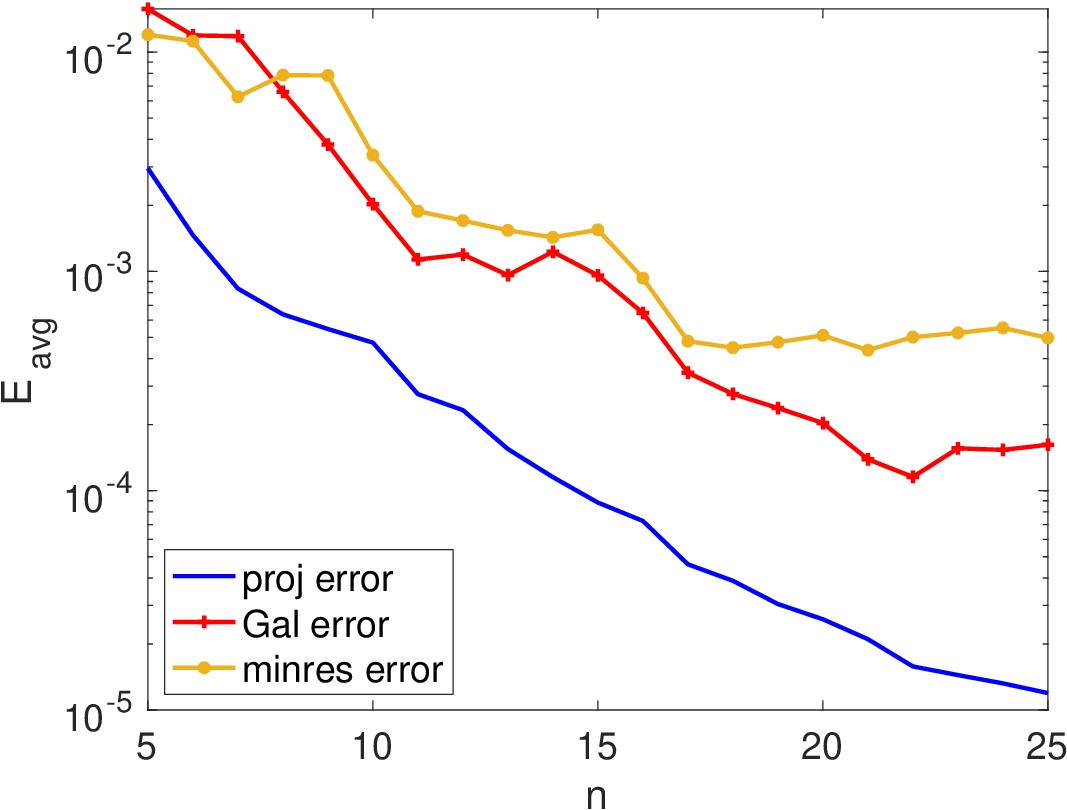}}
  \hfill
  \subfloat[domain 2]{\includegraphics[width=0.4\textwidth]{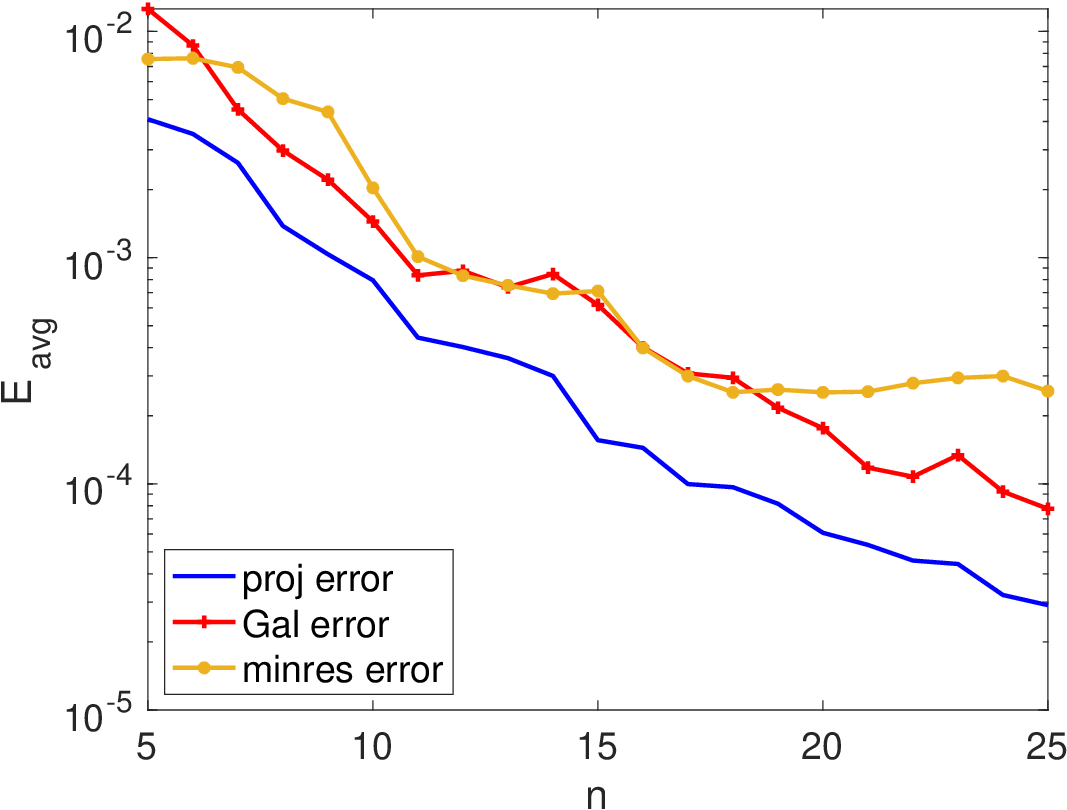}}

  \subfloat[domain 3]{\includegraphics[width=0.4\textwidth]{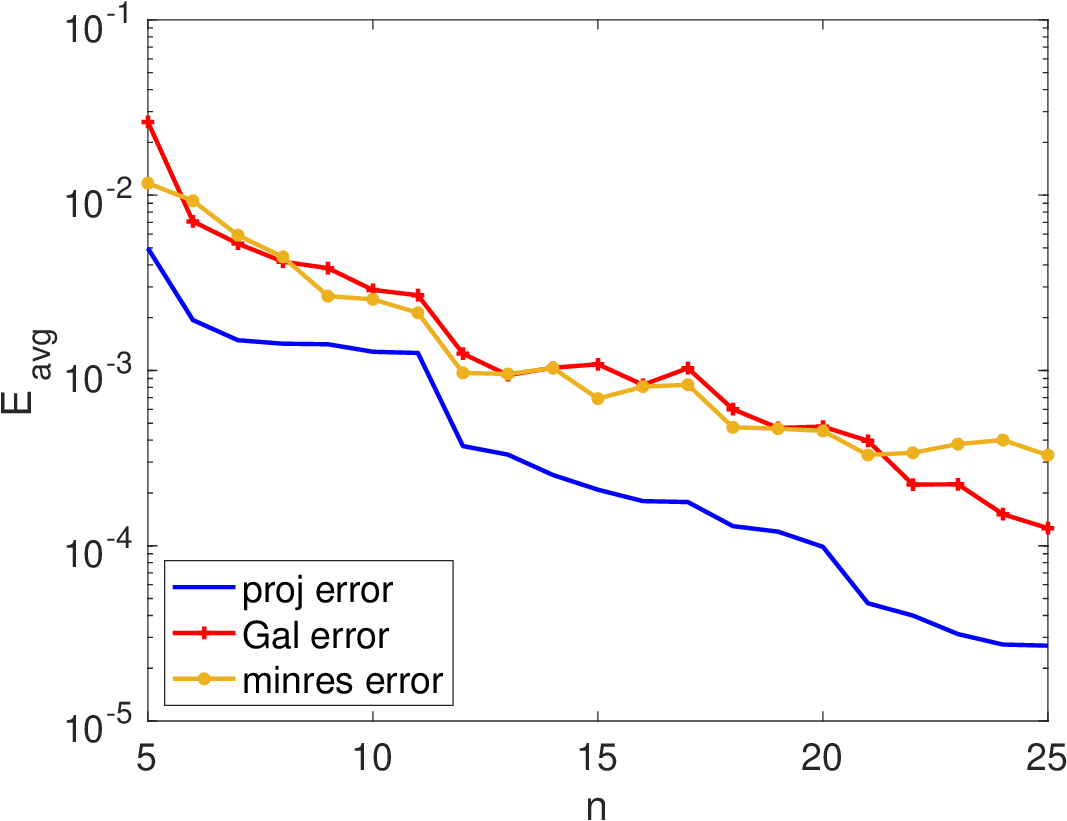}}
  \hfill
  \subfloat[domain 4]{\includegraphics[width=0.4\textwidth]{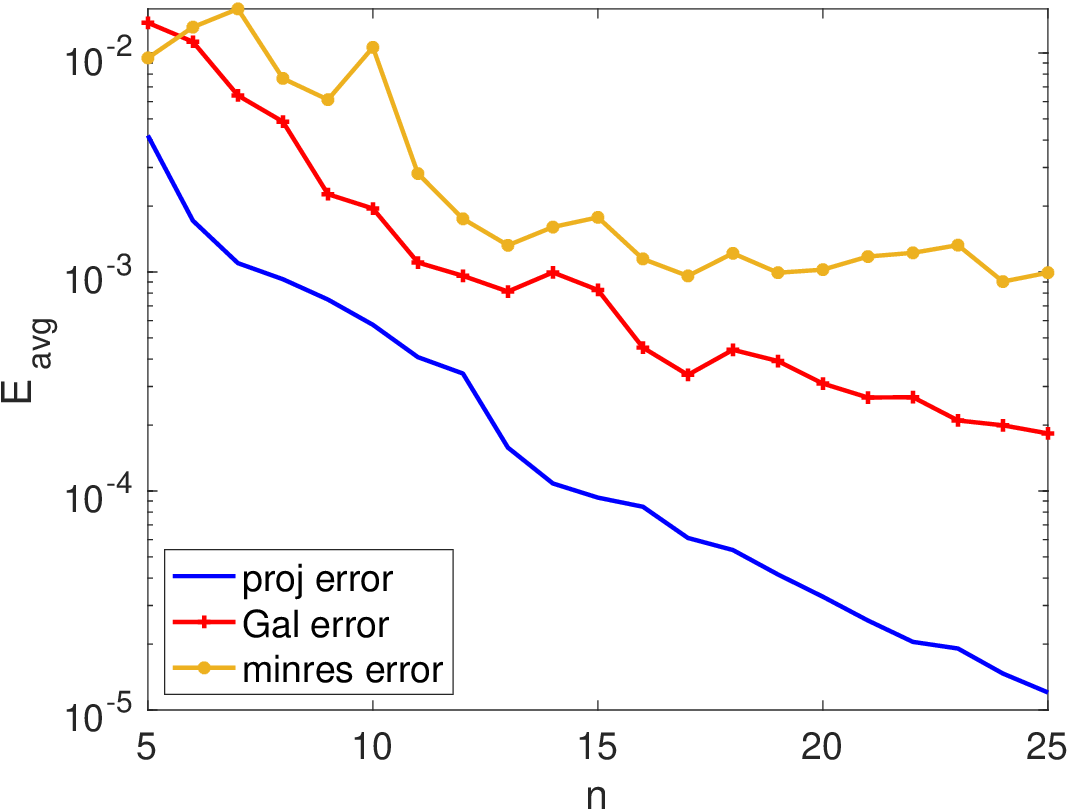}}
  
\caption{localized training for networks with four components. State prediction error for the four sub-domains.}
  \label{fig:rom3_res}
\end{figure}

\begin{figure}[H]
  \centering
  
  \subfloat[port 1]{\includegraphics[width=0.33\textwidth]{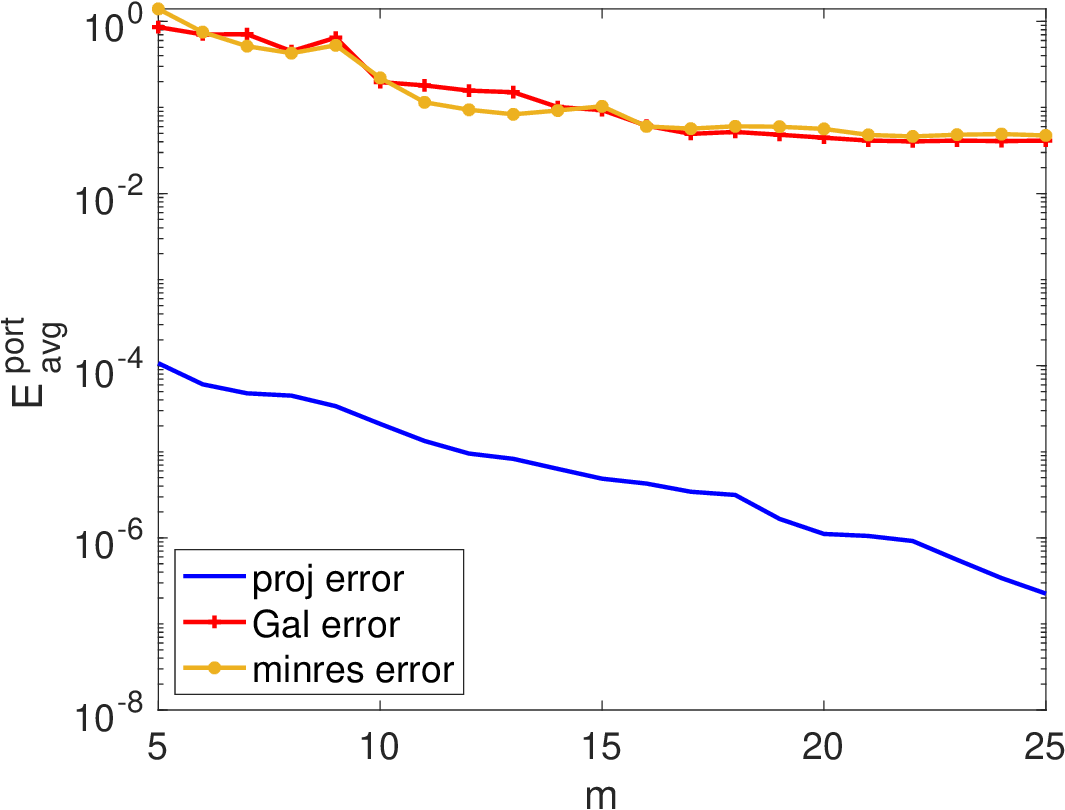}}
~~
  \subfloat[port 2]{\includegraphics[width=0.33\textwidth]{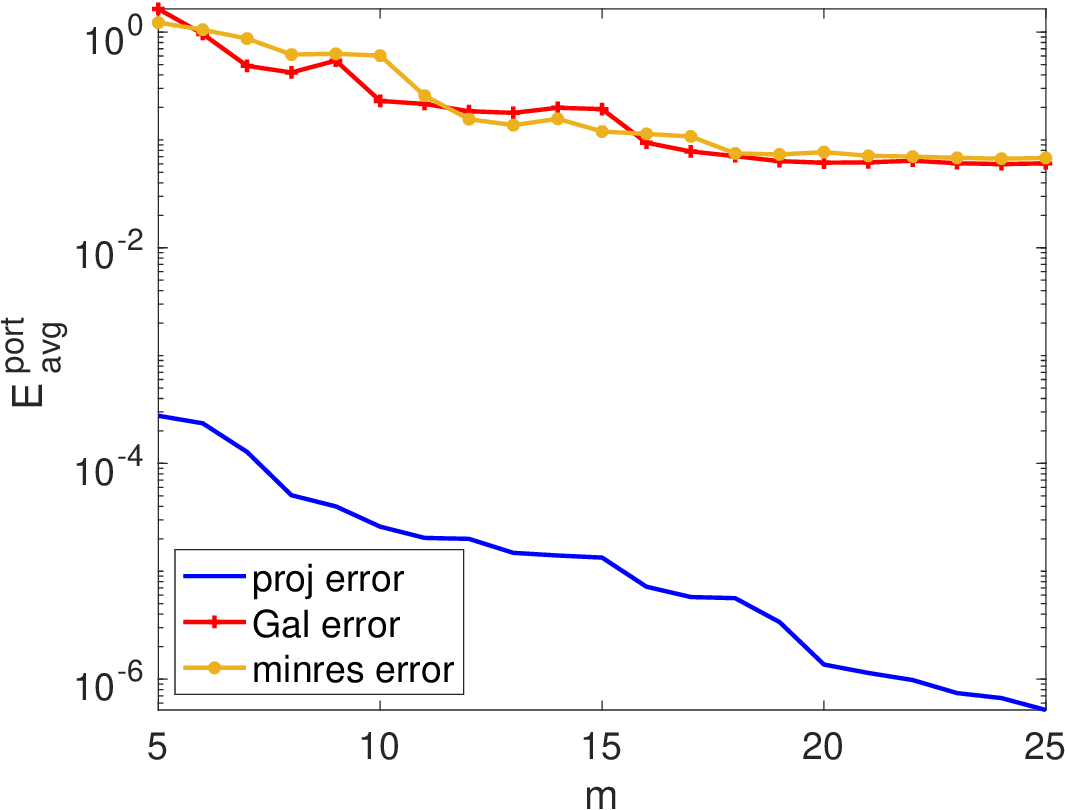}}
~~
  \subfloat[port 3]{\includegraphics[width=0.33\textwidth]{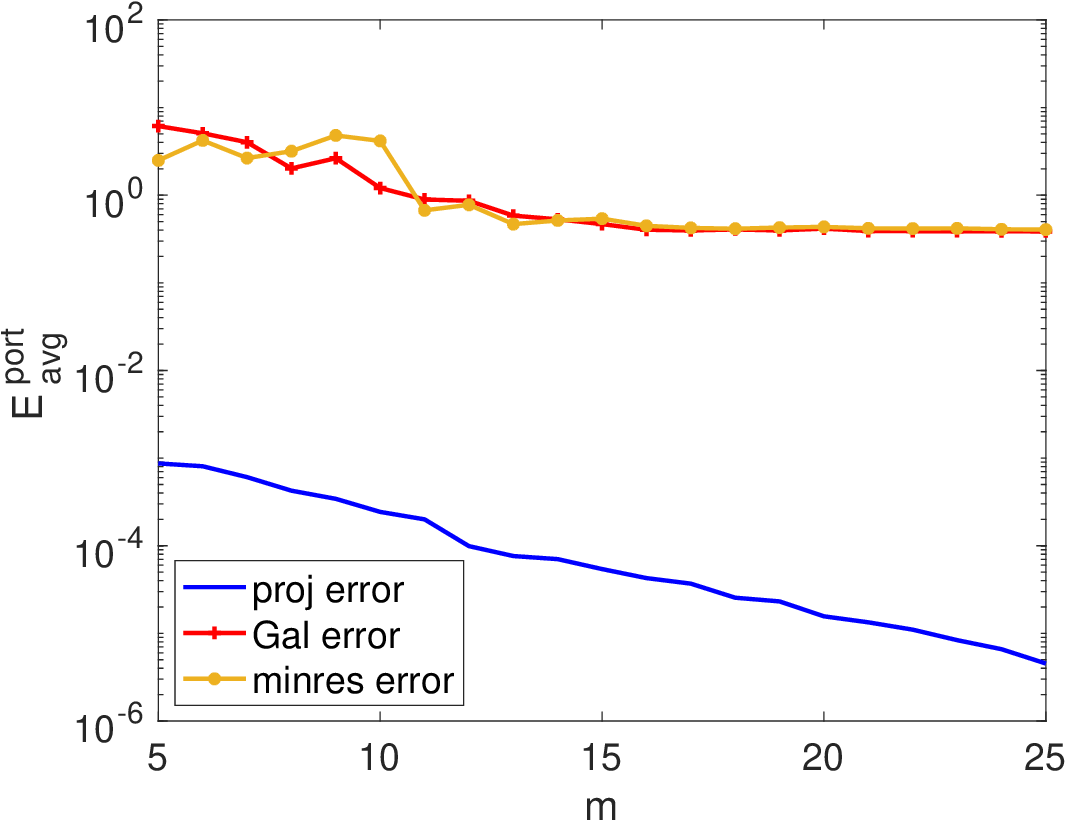}}
  
\caption{localized training for networks with four components. Control prediction error for the three ports (without enrichment).}
\label{fig:rom3_res_gh}
\end{figure}

%aux1=load('data_vascflow/vascflow_res1_loc_coarse_4comps_config3_geo.mat');
%aux2=load('data_vascflow/vascflow_res1_loc_coarse_4comps_config3_Fgh_geo_enrich.mat');

\begin{figure}[H]
  \centering
  
  \subfloat[port 1]{\includegraphics[width=0.33\textwidth]{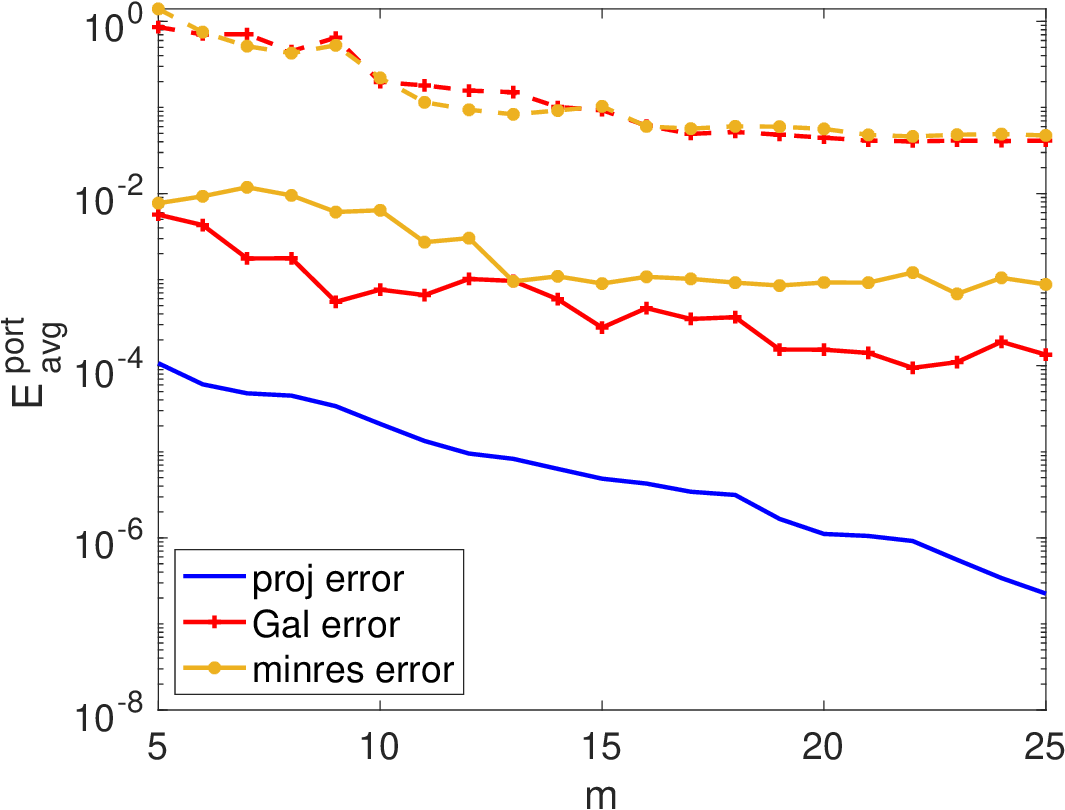}}
~~
  \subfloat[port 2]{\includegraphics[width=0.33\textwidth]{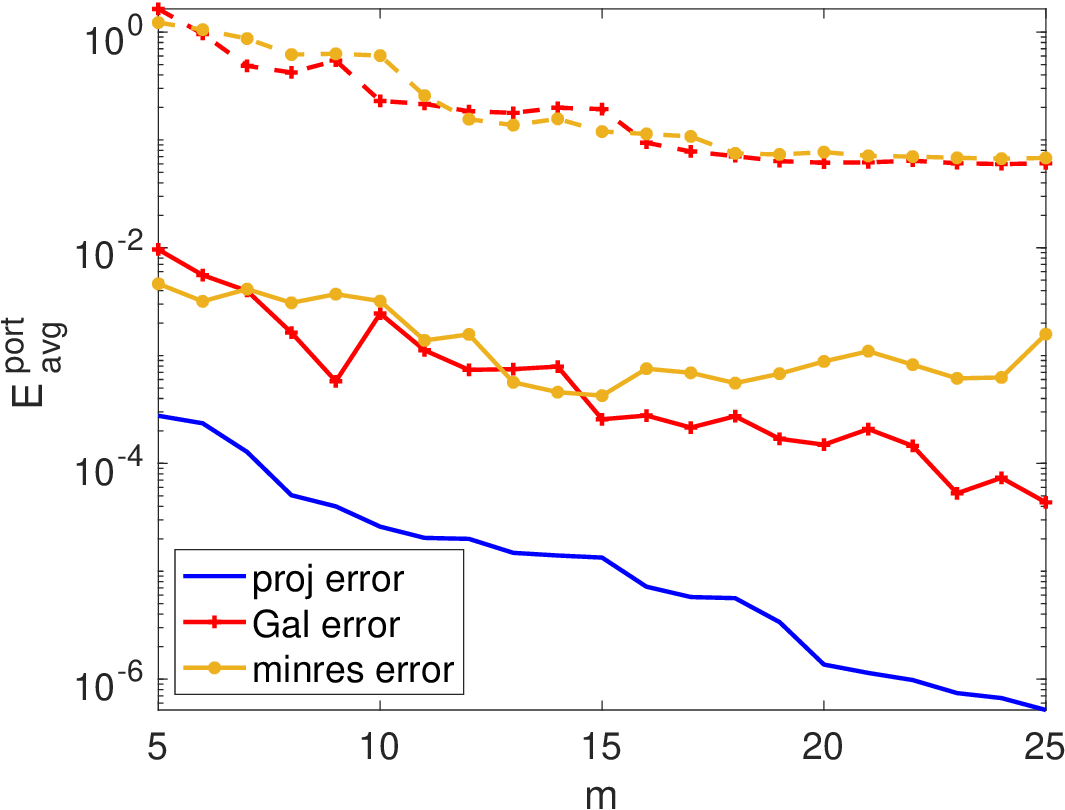}}
~~
  \subfloat[port 3]{\includegraphics[width=0.33\textwidth]{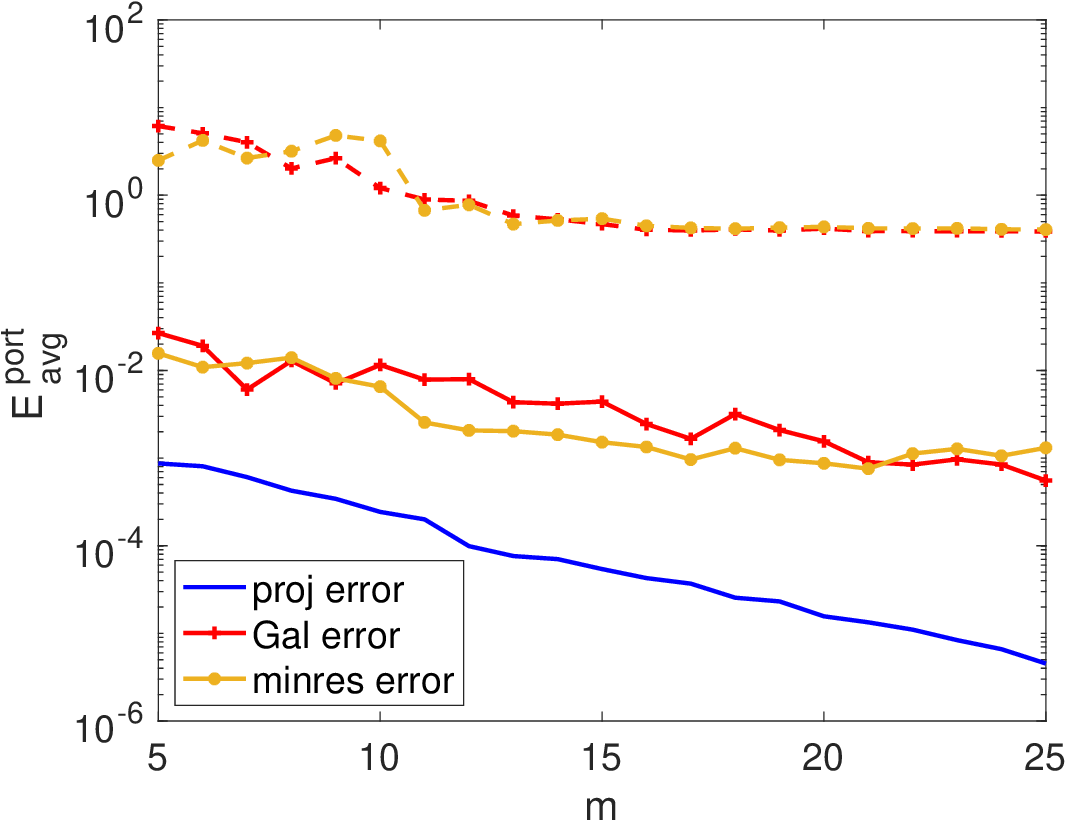}}
  
\caption{localized training for networks with four components. 
Control prediction error for the three ports (with enrichment).}
  \label{fig:rom3_compr_Fgh}
\end{figure}

%\subsection{Validation of localized training of $(g,\,h)$}

\subsubsection{Application to networks with ten components}
We apply the full training procedure described in Algorithm \ref{alg:online_enrichment} to 
$n_{\rm test}=10$ randomly selected configurations  with ten components. 
As for the previous test case, we consider independent geometric variations for each instantiated component and we consider ${\rm Re}\in [50,150]$. We only present results for local Galerkin ROMs: the results obtained using minimum residual projection are comparable and are hence omitted.

Figure \ref{fig:boxplot_no_enrichment} shows the local relative error  for the state and for the control, over the test set for the CB-ROM based on   localized training: we use the same dataset considered in 
section \ref{sec:ROM_results_loc_small} with port-based enrichment (cf. section \ref{sec:enrichment_basis}). We observe that the error is roughly $10\%$ for both state and control and does not decrease as we increase the number of modes.

Figure \ref{fig:online_enrichment_boxplot} 
shows the results for the full application of Algorithm \ref{alg:online_enrichment}. We initialize the algorithm with a ROB  of size  $m_0=10$ for the control using localized training;
  we apply the strategy of section \ref{sec:localized_state}, together with 
  port-based enrichment, to find reduced spaces for the state of size $n_0=10+10$ for each component.
 Then, we apply adaptive enrichment:
 we consider $n_{\rm train}^{\rm glo}=50$ global
 randomly-selected  configurations with ten components;
 we mark $m_{\mathbf{s}}=1$ port, and $m_{\mathbf{w}} = 1-3$ components
of each type (specifically, we mark $1$ component with the largest error of each type, along with the $2$ adjacent components of the marked port).
 Then, we augment the bases for state and control with $n^{\rm glo} = m^{\rm glo} =10$ modes.
 We do not  apply the port-based enrichment strategy 
after each iteration of the  adaptive loop (cf. Line 17).
If Figure \ref{fig:online_enrichment_boxplot}, iteration $it$ corresponds to local ROBs of size $m=10 (it+1)$ and $n=10(it+2)$.
% the ROBs  are hence of the same size of the corresponding ROBs considered in Figure \ref{fig:boxplot_no_enrichment}.

We observe that the enrichment strategy clearly enhances the performance of the CB-ROM.
This result empirically demonstrates the importance of adaptive enrichment when dealing with nonlinear PDEs.

\begin{figure}[H]
  \centering
  
  \subfloat[Components.]{\includegraphics[width=0.5\textwidth]{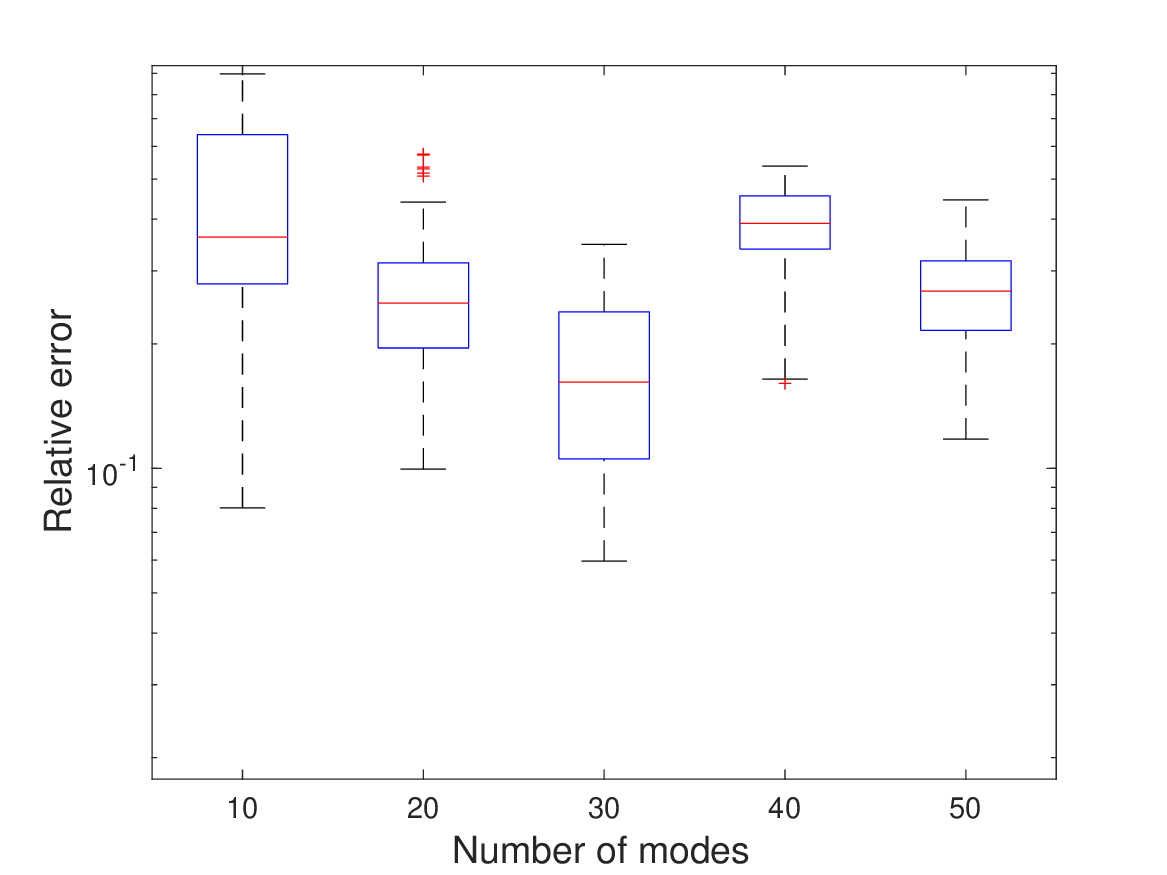}}
  \hfill
  \subfloat[Ports]{\includegraphics[width=0.5\textwidth]{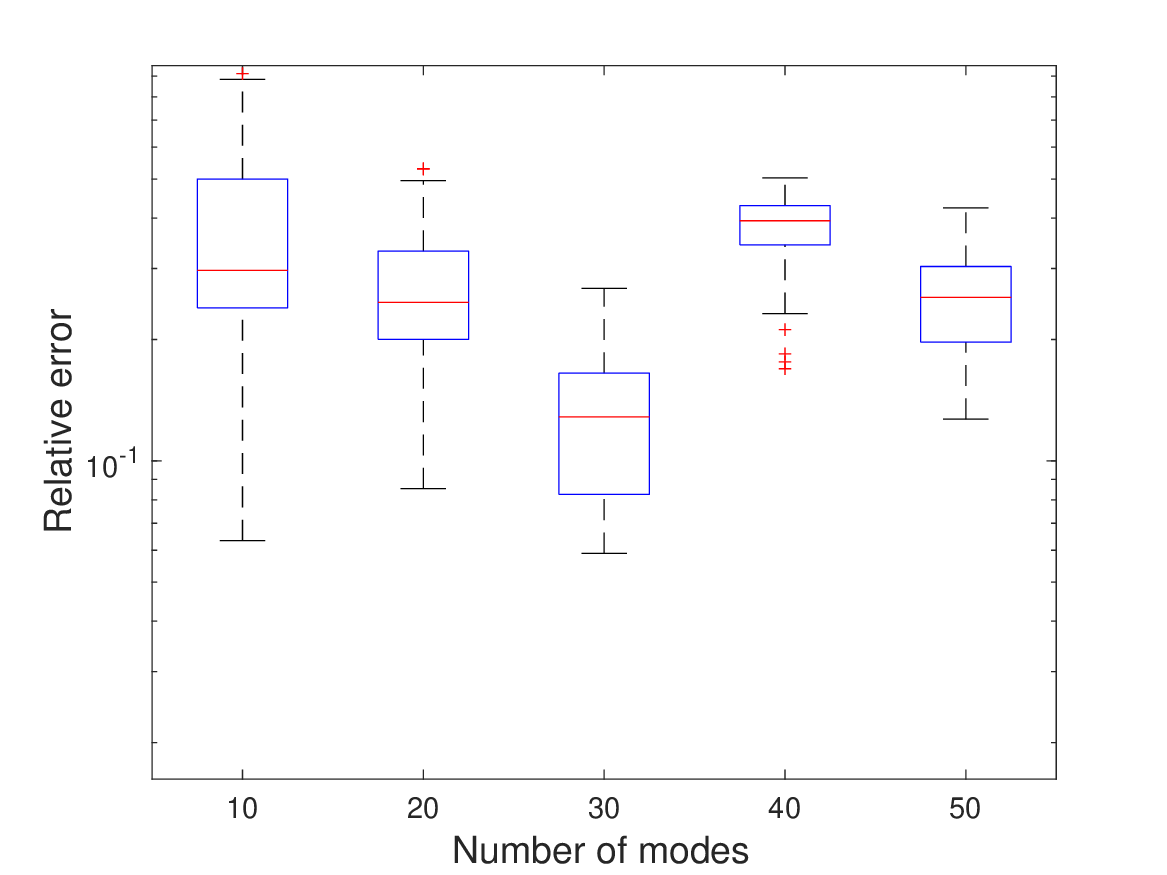}}  
  
\caption{application to networks with ten components. 
Boxplots of the out-of-sample  error
for reduced spaces of several sizes obtained using   localized training  (without adaptive enrichment).}
\label{fig:boxplot_no_enrichment}
\end{figure}

\begin{figure}[H]
  \centering
  
  \subfloat[Components.]{\includegraphics[width=0.5\textwidth]{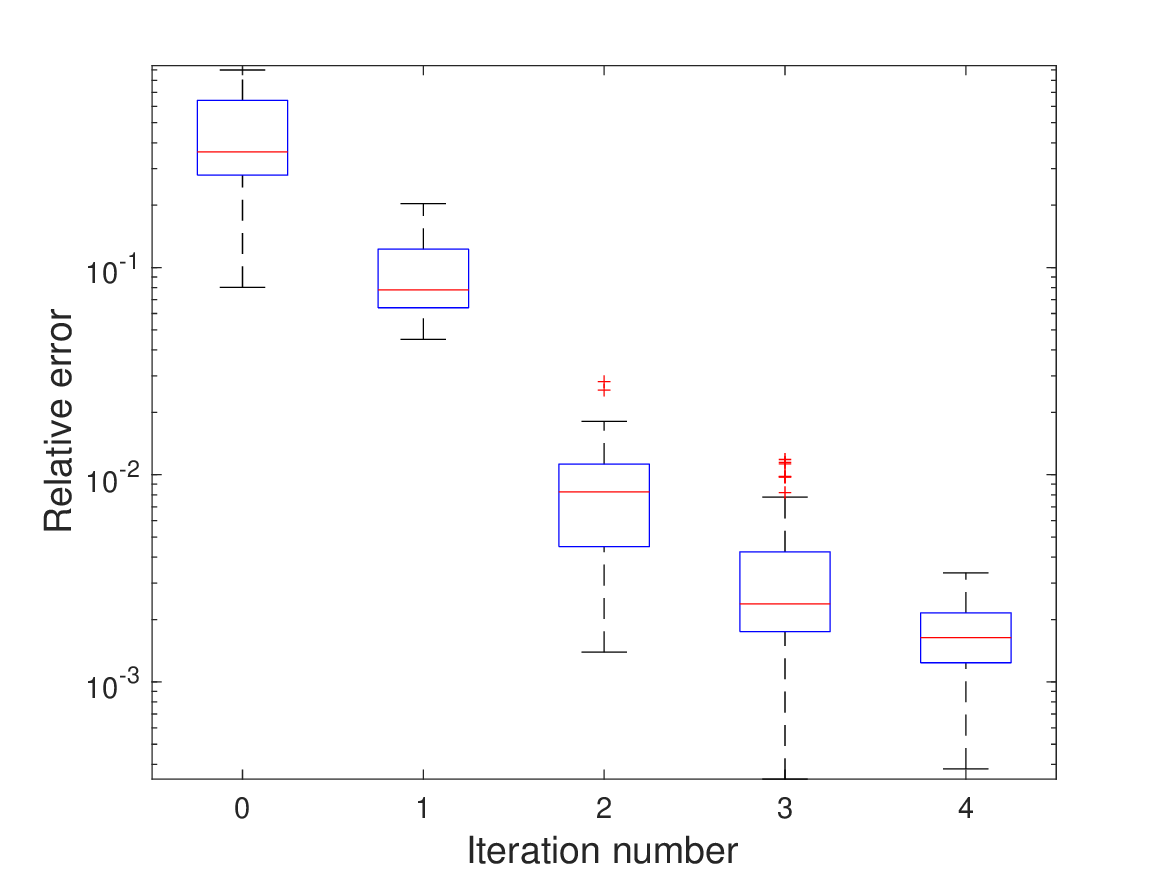}}
  \hfill
  \subfloat[Ports]{\includegraphics[width=0.5\textwidth]{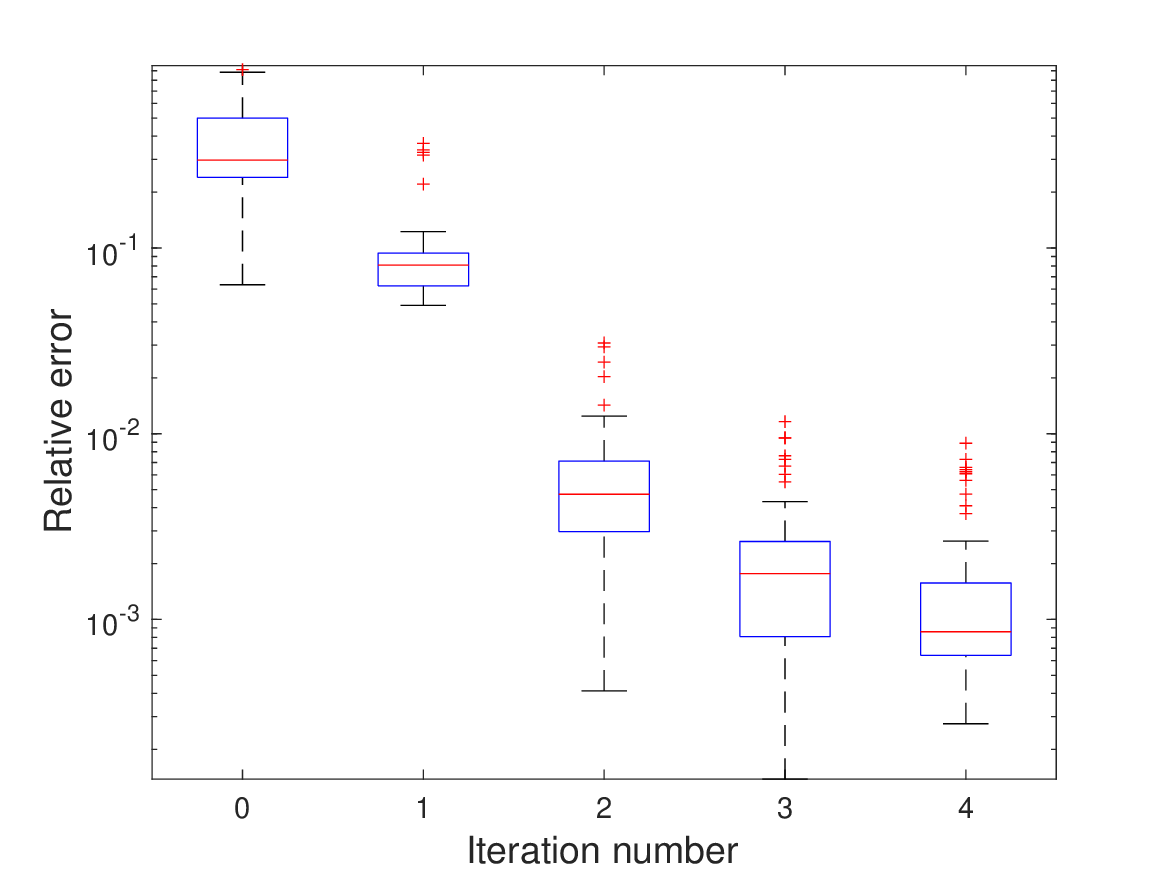}}  
  
\caption{application to networks with ten components. Boxplots of the  out-of-sample  error for several  iterations of Algorithm  \ref{alg:online_enrichment}.}
\label{fig:online_enrichment_boxplot}
\end{figure}

\section{Conclusions}
\label{sec:conclusions}
We developed and numerically validated a component-based model order reduction procedure for incompressible  flows governed by the Navier-Stokes equations. 
Our point of departure is the optimization-based formulation 
of \cite{Gunzburger_Peterson_Kwon_1999}:
we included an additional control variable $h$ for the continuity equation that weakly enforces the continuity of pressure at   interfaces; furthermore, we modified the regularization term to damp spurious oscillations of the control.
We relied on sequential quadratic programming to solve the nonlinear  optimization problem: at each iteration of the procedure, we relied on static condensation of the local degrees of freedom to enable trivial parallelism of the local solves and avoid the introduction of Lagrange multipliers.
We relied on projection-based (Galerkin and Petrov-Galerkin) ROMs to speed up the solution to the local subproblems and we exploited port reduction to reduce the cost of the global problem.
Finally, we adapted the localized training and adaptive enrichment strategy of \cite{Smetana2023} to build the local approximation spaces without the need for expensive global HF solves.

We illustrated the many pieces of our methodology 
for a parametric steady Navier-Stokes problem at moderate ($\mathcal{O}(10^2)$) Reynolds number.
The new DD formulation enables much tighter control of the discrepancy between the FE monolithic solver and the DD solution.
LSPG projection is superior to Galerkin projection in the absence of geometric variability; interestingly, Galerkin and LSPG projection show comparable performance for all the test cases that involve varying geometries.
The port-based enrichment of the state space (cf. section \ref{sec:enrichment_basis}) is key to adequately approximate the control variables.
The localized training strategy discussed in this paper 
leads to poor   reconstructions of the state; adaptive enrichment driven by local error indicators is hence necessary to achieve accurate reconstructions.

In the future, we plan to extend our method to a broader class of problems including multi-physics (fluid-structure interaction) and unsteady problems, and to more challenging (higher-Reynolds, three-dimensional) test cases.
Towards this end, it is of paramount importance to devise effective 
hyper-reduction techniques to speed up local solves and also the assembly of the objective function.
We also plan to combine first-principles models with data-fitted models to enhance the flexibility of the method.

\section*{Acknowledgements}
The work of Lei Zhang is supported by the Fundamental Research Funds for the Central Universities of Tongji University.
\appendix

 \section{Stabilized FE formulation}
\label{sec:stablized_FE}

For completeness, we review the stabilized finite element formulation employed in the numerical results;
we refer to 
\cite{Tezduyar1991,Tezduyar2000} for a thorough  review of stabilized FE methods for incompressible flows.
We denote by $\{  \texttt{D}_k^i \}_k$ the elements of the mesh of $\Omega_i$; 
we further denote by $h_{k,i}$    the size of the $k$-th element of the mesh, and by $r$  the degree of the polynomials.

We consider the residual:
\begin{subequations}
\label{eq:local_problems_stab}
\begin{equation}
\label{eq:local_problems_stab_a}
\mathcal{R}_i^{\rm hf}({\mathbf{u}}_i,\,{{p}}_i, \mathbf{v},\,q )=
\mathcal{R}_i({\mathbf{u}}_i,\,{{p}}_i, \mathbf{v},\,q )
+
\mathcal{R}_i^{\rm supg}({\mathbf{u}}_i,\,{{p}}_i, \mathbf{v} )
+
\mathcal{R}_i^{\rm pspg}({\mathbf{u}}_i,\,{{p}}_i, \mathbf{v} )
+
\mathcal{R}_i^{\rm lsic}({\mathbf{u}}_i, \mathbf{v} ),
\quad
\forall  \, (\mathbf{v},q)\in \mathcal{X}_{i,0}.
\end{equation}
The form $\mathcal{R}_i$ corresponds to the local residual introduced in 
\eqref{eq:local_problems_d}, while the other three terms are designed to improve the stability of the discrete problem.
The form $\mathcal{R}_i^{\rm supg}$ corresponds to the
Streamline upwind Petrov-Galerkin (SUPG,
\cite{brooks1982streamline}) stabilization, which  is designed to handle advection-dominated flows, 
\begin{equation}
\label{eq:local_problems_stab_supg}
\mathcal{R}_i^{\rm supg}({\mathbf{u}},\,{{p}}, \mathbf{v} )=
\sum_k 
\int_{\texttt{D}_k^i}
\left(
\left(
\mathbf{u}\cdot \nabla \right)\mathbf{u}  + \nabla p
 -\nu\Delta \mathbf{u} - \mathbf{f}
\right)
\left( \tau_{\text{supg}}\mathbf{u} \cdot \nabla \mathbf{v} \right)
\, dx;
\end{equation}
the form
$\mathcal{R}_i^{\rm pspg}$ is the Pressure-Stabilized Petrov–Galerkin (PSPG) term \cite{Hughes1986} that is
added to 
the {mass conservation equation}
to eliminate spurious modes in the pressure solution when considering the same polynomial order for pressure and velocity,
\begin{equation}
\label{eq:local_problems_stab_pspg}
\mathcal{R}_i^{\rm pspg}({\mathbf{u}},\,{{p}}, q )
-
\sum_k 
\int_{\texttt{D}_k^i}
\tau_\text{pspg}
\left(
(\mathbf{u}\cdot \nabla)\mathbf{u} + \nabla p
 -\nu\Delta \mathbf{u}
 - \mathbf{f}
\right) \cdot 
\nabla q
\, dx;
\end{equation}
finally, $\mathcal{R}_i^{\rm lsic}$ is the 
least-squares incompressibility constraint (LSIC) stabilization term that is
added to the  momentum equation to improve accuracy and conditioning of the discrete problem \cite{Franca1992, Gelhard2005, Braack2007},
\begin{equation}
\label{eq:local_problems_stab_lsic}
\mathcal{R}_i^{\rm lsic}({\mathbf{u}}, \mathbf{v} )
=
\sum_k 
\int_{\texttt{D}_k^i}
\left(
\nabla\cdot \mathbf{u}
\right)
\tau_{\text{lsic}}\nabla\cdot \mathbf{v}
\, dx.
\end{equation}
\end{subequations}

In the numerical experiments, following
\cite{Peterson2018}, 
we  select the parameters $\tau_{\text{supg}}$, $\tau_{\text{pspg}}$, and $\tau_{\text{lsic}}$   as $\tau_{\text{supg}} = \tau_{\text{pspg}} = \alpha_{\text{supg}} \left[ \left(\frac{2|\mathbf{u}_i|}{h_{k,i}}\right)^2 + 9\left(\frac{4\nu}{h_{k,i}^2}\right)^2 \right]^{-\frac{1}{2}}$, $\tau_{\text{lsic}} = \frac{h_{k,i}^2}{r^2 \tau_{\text{supg}}}$ , where $0\leq \alpha_{\text{supg}} \leq 1$ is a constant that enables the adjustment of $\tau_{\text{supg}}$ for higher-order elements.
In the PTC formulation (cf. \eqref{eq:lsq_discr_final_PTC}), we  
 modify  the coefficients $\tau_{\rm supg}$ and $\tau_{\rm pspg}$  to account for the time step 
 $\tau_{\text{supg}} = \tau_{\text{pspg}} = \alpha_{\text{supg}} \left[ \left(\frac{2}{\Delta t}\right)^2 + \left(\frac{2| \mathbf{u}_i|}{h_{k,i}}\right)^2 + 9\left(\frac{4\nu}{h_{k,i}^2}\right)^2 \right]^{-\frac{1}{2}}$.

\section{Justification of the pressure jump in the minimization formulation}
\label{appendix:pressure_jump}
We consider the configuration depicted in Figure 
\ref{fig:two-subdomains} and we assume that the meshes of $\Omega_1$ and $\Omega_2$ are conforming on $\Gamma_0$. We denote by $\{  \boldsymbol{\phi}_i  \}_{i=1}^{N_{\mathbf{w}}}$ the Lagrangian basis associated with the global space $\mathcal{X}^{\rm hf}$;
we denote by $\mathcal{I}_1, \mathcal{I}_2$ the degrees of freedom associated with the domains $\Omega_1$ and $\Omega_2$, respectively.
We further denote by 
$\mathcal{I}_0 =
\mathcal{I}_1\cap \mathcal{I}_2$ the nodes on the interface $\Gamma_0$;
we introduce the local and global Dirichlet nodes $\mathcal{I}_{1,\rm dir},\mathcal{I}_{2,\rm dir} \subset \{1,\ldots, N_{\mathbf{w}}\}$ and 
 $\mathcal{I}_{\rm dir}=\mathcal{I}_{1,\rm dir} \cap \mathcal{I}_{2,\rm dir}$.
 By construction,  $\mathcal{I}_{\rm dir}\cap \mathcal{I}_0 = \emptyset$ (cf. 
 Figure 
\ref{fig:two-subdomains}).
Finally, we recall the definition of the global  problem
\begin{equation}
\label{eq:global_problem_appendix}
\mathbf{w}^{\rm hf}(1:2)|_{\Gamma_{\rm dir}} = \bs{\Phi}_{\mathbf{u}_{\rm in}},
\quad
\mathcal{R}^{\rm hf}(
\mathbf{w}^{\rm hf},  \mathbf{z} )
= 0
\quad
\forall \,  \mathbf{z} \in \mathcal{X}_{0}^{\rm hf},
\end{equation}
and the two   local problems 
\begin{equation}
\label{eq:local_problems_appendix}
\mathbf{w}_i^{\rm hf}(1:2)|_{\Gamma_{i,\rm dir}} = \bs{\Phi}_{i,\mathbf{u}_{\rm in}},
\quad
\mathcal{R}_i^{\rm hf}(
\mathbf{w}_i^{\rm hf},  \mathbf{z} )
+ (-1)^i \int_{\Gamma_0} \mathbf{s} \cdot    
\mathbf{z}  \, dx = 0
\quad
\forall \,  \mathbf{z} \in \mathcal{X}_{i,0}^{\rm hf},
\quad
i=1,2;
\end{equation}
which depend on the control $\mathbf{s}$.

Since the meshes are conforming, it is possible to verify that 
$$
\mathcal{X}_{i}^{\rm hf}
=
{\rm span} \{ \boldsymbol{\phi}_j|_{\Omega_i}  : j\in   
\mathcal{I}_{i} \},
\quad
i=1,2.
$$
Furthermore, the global residual can be expressed as\footnote{The proof of \eqref{eq:tedious_identity}
exploits the expressions of the residuals
\eqref{eq:local_problems_stab} and 
\eqref{eq:local_problems_d}. We omit the details. We further emphasize that at the right hand side of \eqref{eq:tedious_identity} we should use notation
${\mathcal{R}}_i^{\rm hf}(
\mathbf{w}^{\rm hf}|_{\Omega_i},  \mathbf{z}|_{\Omega_i} )$ for $i=1,2$.
}
\begin{equation}
\label{eq:tedious_identity}
\mathcal{R}^{\rm hf}(
\mathbf{w}^{\rm hf},  \mathbf{z} )
=
{\mathcal{R}}_1^{\rm hf}(
\mathbf{w}^{\rm hf}|_{\Omega_1},  \mathbf{z}|_{\Omega_1} )
+
{\mathcal{R}}_2^{\rm hf}(
\mathbf{w}^{\rm hf}|_{\Omega_2},  \mathbf{z}|_{\Omega_2} )
\quad
\forall \, \mathbf{z} \in \mathcal{X}_{0}^{\rm hf}.
\end{equation}
 Identity \eqref{eq:tedious_identity} implies that 
\begin{equation}
\label{eq:tedious_identity_onemore}
\mathcal{R}_i^{\rm hf}(
\mathbf{w}^{\rm hf},  \boldsymbol{\phi}_j )
=
0
\quad
\forall \, j \in \mathcal{I}_i \setminus \mathcal{I}_0,
\quad
i=1,2;
\end{equation}
therefore, 
since the bilinear form $a(\mathbf{w},  \mathbf{z} ) = \int_{\Gamma_0} \mathbf{w} \cdot   \mathbf{z} \, dx$ is coercive in $\mathcal{Y}:= {\rm span} \{ \boldsymbol{\phi}_j : j\in   
\mathcal{I}_{0} \}$, there exists a unique $\mathbf{s}^{\star} =
{\rm vec}(\mathbf{g}^{\star}, h^{\star})
\in \mathcal{Y}$ such that
\begin{equation}
\label{eq:sstar_appendix}
\mathcal{R}_i^{\rm hf}(
\mathbf{w}^{\rm hf},  \mathbf{z} )
+ (-1)^i \int_{\Gamma_0} \mathbf{s}^{\star}  \cdot    
\mathbf{z}  \, dx = 0
\quad
\forall \,  \mathbf{z} \in \mathcal{X}_{i,0}^{\rm hf},
\end{equation}
for $i=1,2$.

Exploiting the previous discussion, we can prove the following result.
\begin{lemma}
\label{th:tedious_proof}
Let 
$\mathbf{w}^{\rm hf}$ be a  solution to \eqref{eq:global_problem_appendix}. The following hold.
\begin{enumerate}
\item
The triplet 
$(\mathbf{w}^{\rm hf}|_{\Omega_1},
\mathbf{w}^{\rm hf}|_{\Omega_2},
\mathbf{s}^{\star})$ where $\mathbf{s}^{\star}$ satisfies \eqref{eq:sstar_appendix} is a global minimum of 
\eqref{eq:lsq_discr_final_b} for $\delta=0$.
\item
Any global minimum of \eqref{eq:lsq_discr_final_b} for 
$\delta=0$ solves \eqref{eq:global_problem_appendix}; in particular, if the solution 
$\mathbf{w}^{\rm hf}$
to \eqref{eq:global_problem_appendix} is unique, the optimization problem 
\eqref{eq:lsq_discr_final_a} admits a unique solution for $\delta=0$.
\end{enumerate}
\end{lemma}

\begin{proof}
Equation \eqref{eq:sstar_appendix} implies that
the triplet $(\mathbf{w}^{\rm hf}|_{\Omega_1},
\mathbf{w}^{\rm hf}|_{\Omega_2},
\mathbf{s}^{\star})$ satisfies the constraints of \eqref{eq:lsq_discr_final_a}
(cf. \eqref{eq:local_problems_appendix});
since $\mathbf{w}^{\rm hf}$ is continuous, the objective function of \eqref{eq:lsq_discr_final_a}
(cf. \eqref{eq:lsq_discr_final_b})
 is equal to zero for $\delta=0$. Since the function  \eqref{eq:lsq_discr_final_b}
 is non-negative, we conclude that $(\mathbf{w}^{\rm hf}|_{\Omega_1},
\mathbf{w}^{\rm hf}|_{\Omega_2},
\mathbf{s}^{\star})$ is a global minimum of  \eqref{eq:lsq_discr_final_a}.

Exploiting the first part of the proof, we find that any
 global minimum
$(\mathbf{w}_1,
\mathbf{w}_2, \mathbf{s})$ of \eqref{eq:lsq_discr_final_a}
 satisfies $\mathcal{F}_{\delta=0}\left(
 {\mathbf{w}}_1, {\mathbf{w}}_2,\mathbf{s}
 \right)$
$\mathcal{F}_{\delta=0} (\mathbf{w}^{\rm hf}|_{\Omega_1},
\mathbf{w}^{\rm hf}|_{\Omega_2},
\mathbf{s}^{\star})  =0$. This implies that the function
 $\mathbf{w}:\Omega \to \mathbb{R}^3$ such that
 $\mathbf{w}|_{\Omega_1}={\mathbf{w}}_1$ and 
  $\mathbf{w}|_{\Omega_2}={\mathbf{w}}_2$ is continuous and belongs to $\mathcal{X}^{\rm hf}$.
Recalling
\eqref{eq:local_problems_appendix}, we have that $\mathbf{w}$ satisfies 
 $\mathbf{w}(1:2)|_{\Gamma_{\rm dir}} = \bs{\Phi}_{\mathbf{u}_{\rm in}}$.
 Furthermore, 
 since $ \mathbf{z}|_{\Omega_i} \in \mathcal{X}_{i,0}$ for any
  $\mathbf{z}\in \mathcal{X}^{\rm hf}$, 
  we have that
$$
\mathcal{R}^{\rm hf}(
\mathbf{w},  \mathbf{z} )
\overset{\eqref{eq:tedious_identity}}{=}
\mathcal{R}_1^{\rm hf}(
\mathbf{w}_1,   \mathbf{z}|_{\Omega_1} )
+
\mathcal{R}_2^{\rm hf}(
\mathbf{w}_2,   \mathbf{z}|_{\Omega_2} )
\overset{\eqref{eq:local_problems_appendix}}{=}
-\int_{\Gamma_0} \mathbf{s}\cdot    \mathbf{z} \, dx
+\int_{\Gamma_0} \mathbf{s}\cdot    \mathbf{z} \, dx
=0,
$$
which is \eqref{eq:global_problem_appendix}.
We conclude that $\mathbf{w}$ solves \eqref{eq:global_problem_appendix}.
If the solution to \eqref{eq:global_problem_appendix} is unique, exploiting the previous argument, 
any solution 
$(\mathbf{w}_1,\mathbf{w}_2, \mathbf{s})$ should satisfy 
 $\mathbf{w}_1={\mathbf{w}}^{\rm hf}|_{\Omega_1}$ and 
  $\mathbf{w}_2={\mathbf{w}}^{\rm hf}|_{\Omega_2}$. Furthermore, since the solution to \eqref{eq:sstar_appendix} is unique, we also find $\mathbf{s}=\mathbf{s}^\star$. In conclusion,   \eqref{eq:lsq_discr_final_a} has a unique global minimum.
\end{proof}

Lemma \ref{th:tedious_proof} illustrates the connection between the monolithic problem and the solution to the optimization problem 
\eqref{eq:lsq_discr_final_a}; the well-posedness analysis in \cite{Gunzburger_Lee_2000} shows that in the continuous limit (i.e., $N_{\mathbf{w}}\to \infty$)  $h^{\star}=0$; nevertheless, in general  $h^{\star}\neq 0$ for finite-dimensional discretizations. To illustrate this fact, we consider the solution to the Stokes problem (see Figure \ref{fig:two-subdomains} for the definitions of the boundary subdomains)
$$
\left\{
\begin{array}{ll}
-  \Delta \mathbf{u}   + \nabla p
= 
{\rm vec}(x_1, \cos(x_2^2) )
& {\rm in} \; \Omega=(0,1)^2, \\
\nabla \cdot \mathbf{u} = 0 
& {\rm in} \; \Omega, \\
\mathbf{u}|_{\Gamma_{\rm dir}}={\rm vec}((1-x_2)x_2, 0 ),
\;\;
\mathbf{u}|_{\Gamma_{\rm dir}^0}=0,
\;\;
\left(
\nabla \mathbf{u} - p\mathbf{I}
\right) \mathbf{n}
|_{\Gamma_{\rm neu}}=0,
&
\\
\end{array}
\right.
$$
based on a P3-P2 Taylor-Hood discretization for three meshes of increasing size.
Figure \ref{fig:stokes_model_problem}(a) shows the final mesh used for computations  whereas the blue dots indicate the  interface $\Gamma_0$;
Figure \ref{fig:stokes_model_problem}(b) shows the behavior of $h^{\star}$ for three meshes with $4^2, 8^2, 16^2$ global elements: as expected, as we increase the size of the mesh, the magnitude of 
$h^{\star}$  decreases.

\begin{figure}[H]
\centering
\subfloat[]{
\includegraphics[width=8.3cm]{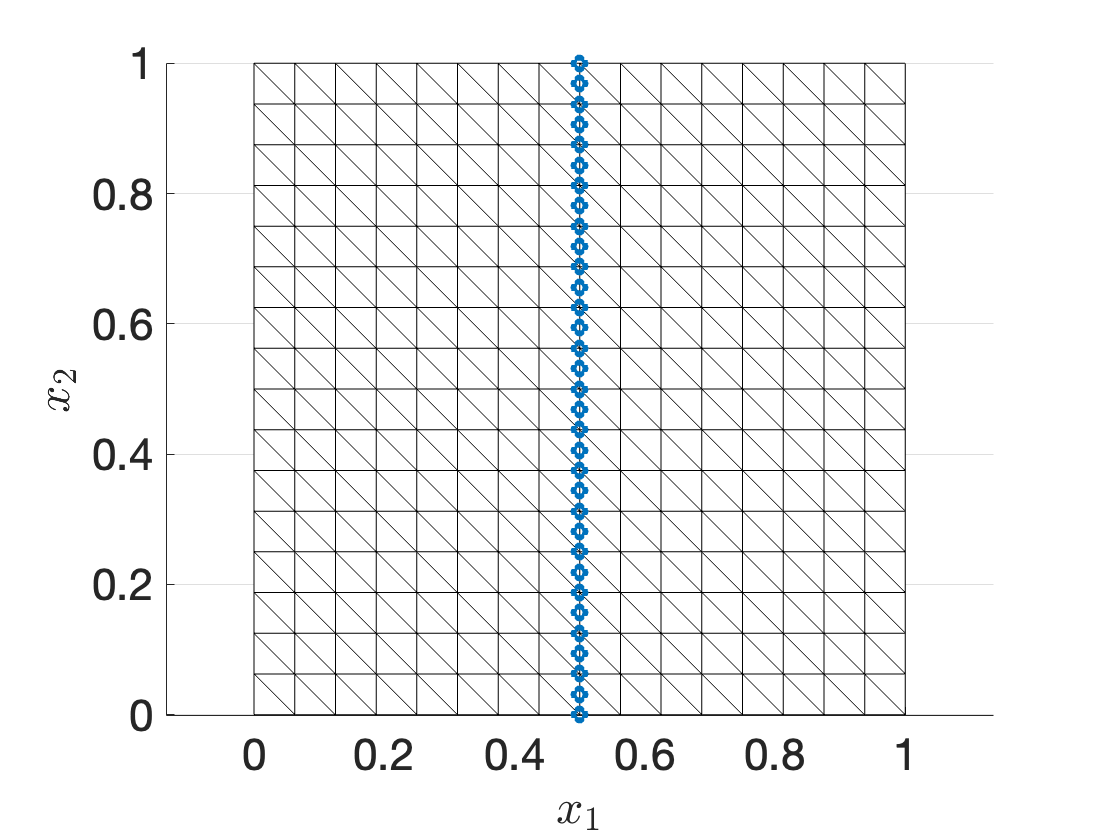}}
~~
\subfloat[]{
\includegraphics[width=7.3cm]{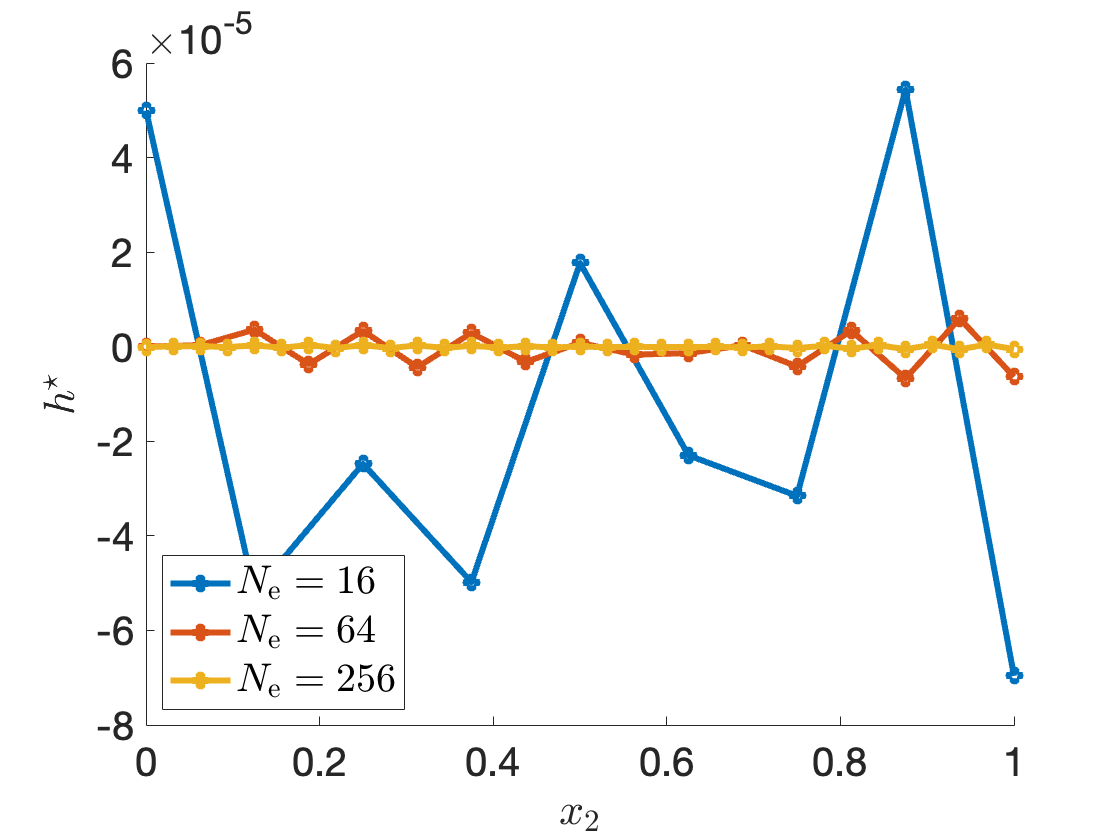}}

\caption{justification of the pressure jump; Stokes model problem.}
\label{fig:stokes_model_problem}
\end{figure}

\section{Justification of the enrichment strategy}
\label{appendix:enrichment}

We consider the algebraic problem:
\begin{equation}
\label{eq:toy_appendix1}
\min_{\mathbf{w}\in \mathbb{R}^N, \mathbf{s}\in \mathbb{R}^M} \Big| \mathbf{C} \mathbf{w} - \mathbf{b}  \Big|
\quad
{\rm s.t.} \;\;
\mathbf{A} \mathbf{w} +
\mathbf{B} \mathbf{s}  =
  \mathbf{f},
\end{equation}
with $\mathbf{A}\in \mathbb{R}^{N\times N}$, $\mathbf{B}\in \mathbb{R}^{N\times M}$, $\mathbf{C}\in \mathbb{R}^{M\times N}$,
$ \mathbf{b} \in \mathbf{R}^M$, 
$ \mathbf{f} \in \mathbf{R}^N$  
 and $N>M$.
 If $\mathbf{A}$ is full rank,
any solution $(\mathbf{w}^{\star}, \mathbf{s}^{\star})$ to
\eqref{eq:toy_appendix1}
 satisfies
 $\mathbf{w}^{\star} = \mathbf{A}^{-1} \left(
 \mathbf{f}  - \mathbf{B} 
 \mathbf{s}^{\star} 
 \right)$ and 
 $ \mathbf{s}^{\star}  = {\rm arg}
 \min_{\mathbf{s}\in \mathbb{R}^M} \Big| \mathbf{D} \mathbf{s} - \mathbf{c}  \Big|$,
 with 
 $\mathbf{D}= \mathbf{C} \mathbf{A}^{-1} \mathbf{B}$ and 
 $\mathbf{c} = \mathbf{b} - \mathbf{A}^{-1}   \mathbf{f} $. Therefore, provided that $\mathbf{C}$ is full rank, \eqref{eq:toy_appendix1} is well-posed if and only if $\mathbf{A}^{-1} \mathbf{B}$
is full rank.

Let $\mathbf{Z}
=[\boldsymbol{\zeta}_1,\ldots,\boldsymbol{\zeta}_n]
\in \mathbb{R}^{N\times n}$,   $\mathbf{W}
=[\boldsymbol{\eta}_1,\ldots,\boldsymbol{\eta}_m]
\in \mathbb{R}^{M\times m}$ and
$\mathbf{Y}\in \mathbb{R}^{N\times n}$ be orthogonal matrices with $n<N$ and $m<M$; exploiting these definitions, we define the projected problem
 \begin{equation}
\label{eq:toy_appendix1_projected}
\min_{\bs{\alpha}\in \mathbb{R}^n, \bs{\beta}\in \mathbb{R}^m} \Big| \overline{\mathbf{C}} \bs{\alpha} - \mathbf{b}  \Big|
\quad
{\rm s.t.} \;\;
\overline{\mathbf{A}}  \bs{\alpha} +
\overline{\mathbf{B}}  \bs{\beta}  =
\overline{\mathbf{f}},
\end{equation}
 with $\overline{\mathbf{C}} = \mathbf{C} \mathbf{Z}$, 
 $\overline{\mathbf{A}} =\mathbf{Y}^\top \mathbf{A} \mathbf{Z}$, 
  $\overline{\mathbf{B}} =\mathbf{Y}^\top \mathbf{B} \mathbf{W}$ and 
   $\overline{\mathbf{f}} =\mathbf{Y}^\top \mathbf{f}$. It is straightforward to prove the following result:
   here, ${\rm col}(\mathbf{X})$ denotes the linear space spanned by the columns of the matrix $\mathbf{X}$, while 
   ${\rm orth}(\mathbf{X})$ is the orthogonal matrix that is obtained by orthogonalizing the columns of $\mathbf{X}$.
   
\begin{lemma}
\label{th:silly_result}
Let $ {\rm col}(\mathbf{A}^{-1} 
\mathbf{B} \mathbf{W}   )    \subset   {\rm col}(\mathbf{Z})$ and let $\mathbf{Y} = {\rm orth}(  \mathbf{A} \mathbf{Z}    )$. Then, 
$\overline{\mathbf{A}}^{-1} \overline{\mathbf{B}}$ is full rank, and \eqref{eq:toy_appendix1_projected} is well-posed.
\end{lemma}
\begin{proof}
We first  prove that $\overline{\mathbf{A}}\in \mathbb{R}^{n\times n}$ is invertible.
By contradiction, there exists $\bs{\alpha}\in \mathbb{R}^n$ such that 
$\overline{\mathbf{A}}\bs{\alpha}=0$. Since $\mathbf{Y} = {\rm orth}(  \mathbf{A} \mathbf{Z}    )$, there exists $\bs{\beta}\in \mathbb{R}^n$ such that
$\mathbf{A}  \mathbf{Z} \bs{\alpha} =  \mathbf{Y} \bs{\beta}$.
We hence find 
$$
0 = \bs{\beta}^\top \overline{\mathbf{A}} \bs{\alpha}
=
( \mathbf{Y} \bs{\beta} )^\top
 \mathbf{A} \mathbf{Z} 
 \bs{\alpha}
=
| \mathbf{A} \mathbf{Z}   \bs{\alpha} |^2.
$$
The latter implies that $\mathbf{Z}   \bs{\alpha}$  is a non-trivial element of  the kernel   of $\mathbf{A}$: this is in contradiction with the hypothesis that $\mathbf{A}$ is invertible.

Exploiting the same  argument, we prove that $\overline{\mathbf{B}}$ is full rank.
By contradiction, there exists $\bs{\beta}\in \mathbb{R}^m$ such that 
$\overline{\mathbf{B}}\bs{\beta}=0$.
Since ${\rm col}(\mathbf{Y}) = {\rm col}(\mathbf{A} \mathbf{Z})$ and 
${\rm col}(\mathbf{A}^{-1}\mathbf{B} \mathbf{W})  \subset {\rm col}( \mathbf{Z})$, there exist 
 $\bs{\alpha}, \bs{\alpha}' \in \mathbb{R}^n$ such that
$\mathbf{B}  \mathbf{W} \bs{\beta} = 
\mathbf{A}  \mathbf{Z} \bs{\alpha}' =  \mathbf{Y} \bs{\alpha}$.
We hence find 
$$
0 = \bs{\alpha}^\top \overline{\mathbf{B}} \bs{\beta}
=
( \mathbf{Y} \bs{\alpha} )^\top
 \mathbf{B} \mathbf{W} 
\bs{\beta}
=
| \mathbf{B} \mathbf{W}   \bs{\beta} |^2.
$$
The latter implies that $\mathbf{W}   \bs{\beta} $  is a non-trivial element of  the kernel sof  $\mathbf{B}$: this is in contradiction with the hypothesis that $\mathbf{B}$ is  full-rank.
\end{proof}
 
Lemma \ref{th:silly_result} provides a rigorous justification of the enrichment strategy in section \ref{sec:enrichment_basis}. 
The matrix $\widetilde{\mathbf{W}} =  - 
\mathbf{A}^{-1} \mathbf{B} \mathbf{W}$ corresponds to the derivative of the state $\mathbf{w}$ with respect to the control $\widehat{\mathbf{s}} = \mathbf{W} \bs{\beta}$; the columns 
$\widetilde{\mathbf{w}}_1,\ldots,  \widetilde{\mathbf{w}}_m$
of the matrix 
 $\widetilde{\mathbf{W}}$ satisfy 
 $$
 \mathbf{A} \widetilde{\mathbf{w}}_k + \mathbf{B} \bs{\eta}_k = 0,
 \quad
 k=1,\ldots,m, 
 $$
 which corresponds to \eqref{eq:enriching_modes}. Similarly, as discussed in 
 \cite{taddei2021space,taddei2021discretize}, the choice of the test space in  Lemma \ref{th:silly_result}  is consistent with 
\eqref{eq:test_space_a}.
 
\bibliographystyle{abbrv}	
\bibliography{references}

\end{document}